\documentclass[12pt,reqno]{amsart}

\usepackage[T1]{fontenc}
\usepackage{lmodern}
\usepackage{mathtools,amssymb,amsthm,bm}
\usepackage[margin=1.2in]{geometry}
\usepackage[colorlinks=true,citecolor=blue,linkcolor=blue,urlcolor=blue]{hyperref}

\numberwithin{equation}{section}

\makeatletter
\def\@settitle{\begin{center}%
  \baselineskip14\p@\relax
  \bfseries
  \@title
  \end{center}%
}
\makeatother

\theoremstyle{plain}
\newtheorem{theorem}{Theorem}[section]
\newtheorem{proposition}[theorem]{Proposition}
\newtheorem{lemma}[theorem]{Lemma}
\newtheorem{corollary}[theorem]{Corollary}
\newtheorem{conjecture}[theorem]{Conjecture}

\theoremstyle{definition}

\newtheorem{problem}[theorem]{Problem}

\theoremstyle{remark}
\newtheorem{remark}[theorem]{Remark}

\newcommand{\T}{\mathbb T}
\newcommand{\R}{\mathbb R}

\newcommand{\ii}{\mathrm i}
\newcommand{\dd}{\mathrm d}
\newcommand{\eps}{\varepsilon}
\newcommand{\Hilb}{\mathcal H}

\newcommand{\Id}{I}

\newcommand{\tr}{\operatorname{tr}}

\newcommand{\norm}[1]{\left\|#1\right\|}
\newcommand{\abs}[1]{\left|#1\right|}
\newcommand{\mean}[1]{\left\langle #1\right\rangle}

\newcommand{\ip}[2]{\left\langle #1,#2\right\rangle}

\begin{document}

\title[Global Regularity for an Active-Line Model]{Arbitrary-Size Global Regularity for a Reduced Oldroyd--B Active-Line Model}
\author[S. Peng]{Sai Peng\\{\small School of Mathematics and Computational Science, Xiangtan University}\\{\small \lowercase{\texttt{pscfd@xtu.edu.cn}}}}

\begin{abstract}
We study a one-dimensional active-line equation motivated by thin stress-sheet
dynamics in the high-Weissenberg Oldroyd--B regime.  A positive periodic line
density \(\rho=m+\eta\) satisfies
\[
  \rho_t+cP_0\!\left\{\rho\Lambda\rho-(\Hilb\rho)\rho_s\right\}
  +\gamma(\rho-m)=0,
  \qquad c,m>0,\quad \gamma\ge0,
\]
where \(\Lambda=\Hilb\partial_s\) and \(P_0f=f-\langle f\rangle\).  Every
strictly positive smooth initial
density of arbitrary size generates a unique global smooth solution.  The
key is the pointwise cancellation obtained after one differentiation: for
\(w=\rho_s\),
\[
 w_t+c\rho\Lambda w-c(\Hilb\rho)w_s+\gamma w=0.
\]
Its maximum principle controls the slope globally, and critical-drift
H\"older and Schauder estimates close all higher derivatives.  We also prove
quantitative small-oscillation stability.  Independently, an exact
fourth-difference sum-of-squares identity gives
\[
  \int_\T \rho^2\Lambda^3\rho\,\dd s\ge0
\]
for every smooth nonnegative density.

The stronger derivative-energy sign leads, in its sharp phase-opposed form,
to the cubic convolution problem
\[
 \mathcal C(x)\le 2AE_3,\qquad
 A=\sum_{n\ge1}x_n,\quad E_3=\sum_{n\ge1}n^3x_n^2,
\]
for nonnegative Fourier amplitudes, where
\[
 \mathcal C(x)=\sum_{a,b\ge1}
 (a+b)(a^2+ab+b^2)x_ax_bx_{a+b}.
\]
The constant is sharp along critical \(n^{-3/2}\) plateaux.  We prove the
inequality for several cutoff-uniform and infinite-support classes, including
selected Schur, dyadic-layer, Mellin, and moment families.  The unrestricted
inequality remains open, but it is not needed for the global regularity
theorem.
\end{abstract}

\maketitle

\section{Introduction}
The Oldroyd--B system is a basic continuum model for dilute polymeric fluids
\cite{Oldroyd1950,BirdArmstrongHassager1987,Renardy2000}.  In the absence of
stress diffusion, the conformation tensor is transported and stretched while
remaining positive definite.  This combination makes the high-Weissenberg
regime analytically difficult.  Available global and continuation results for
related systems use combinations of viscosity, smallness, stress
regularization, or additional structural control
\cite{GuillopeSaut1990,LionsMasmoudi2000,CheminMasmoudi2001,
BarrettBoyaval2011,ConstantinKliegl2012,ElgindiRousset2015,
RenardyThomases2021}; the log-conformation representation gives a natural
coordinate system for the positive cone
\cite{FattalKupferman2004,FattalKupferman2005}.

Motivated by a strongly aligned stress concentration near a periodic material
line, we examine the reduced scalar equation
\begin{equation}\label{eq:rho}
  \rho_t+cP_0\!\left\{\rho\Lambda\rho-(\Hilb\rho)\rho_s\right\}
  +\gamma(\rho-m)=0,
  \qquad \mean{\rho}=m>0,
\end{equation}
with \(c>0\) and \(\gamma\ge0\), where \(\Lambda=\Hilb\partial_s\).
The projection \(P_0f=f-\mean f\) is part of the mass normalization: it
preserves \(\mean\rho=m\) and disappears after one spatial derivative.
Writing \(\rho=m+\eta\) gives
\begin{equation}\label{eq:eta}
  \eta_t+cm\Lambda\eta+\gamma\eta
  =-cP_0\bigl\{\eta\Lambda\eta-(\Hilb\eta)\eta_s\bigr\}.
\end{equation}
The mean supplies half-Laplacian damping.  The quadratic term has an
additional Fourier null structure: interactions of two equal-sign modes
vanish.  The model therefore belongs to the family of one-dimensional
nonlocal equations used to isolate fluid singularity mechanisms
\cite{ConstantinLaxMajda1985,CordobaCordobaFontelos2005}, but its positive
density and one-sided frequency cancellation are specific to the present
setting.

The first main result settles the large-data alternative for the reduced
equation.  Differentiating once and writing \(w=\rho_s\) gives the exact
cancellation
\[
 w_t+c\rho\Lambda w-c(\Hilb\rho)w_s+\gamma w=0.
\]
The maximum principle controls \(\norm w_\infty\) for all time.  Together
with the positive lower barrier, critical-drift H\"older regularity, and
order-one Schauder estimates, this proves global smooth existence for every
strictly positive smooth initial density of arbitrary size.  Small
oscillations also satisfy a quantitative Sobolev estimate.  These statements
concern the reduced equation itself; no compactness theorem from the full
two-dimensional Oldroyd--B system to the active-line ansatz is asserted here.

The main analytic feature is hidden positivity.  Theorem~
\ref{thm:third-order-hidden-positivity} proves, by a translation-scale
sum-of-squares identity,
\[
  \int_\T\rho^2\Lambda^3\rho\,\dd s\ge0
  \qquad(\rho\ge0).
\]
A stronger derivative-energy form is more delicate.  In the phase-opposed
Fourier configuration it reduces to the sharp scalar target
\begin{equation}\label{eq:intro-sharp-convolution}
 \mathcal C(x)\le2AE_3,
 \qquad
 A=\sum_{n\ge1}x_n,\quad
 E_3=\sum_{n\ge1}n^3x_n^2,
\end{equation}
for \(x_n\ge0\), with
\[
 \mathcal C(x)=\sum_{a,b\ge1}
 (a+b)(a^2+ab+b^2)x_ax_bx_{a+b}.
\]
The constant \(2\) is sharp along critical \(n^{-3/2}\) plateaux, so no
fixed Fourier cutoff can settle the problem.  We establish
\eqref{eq:intro-sharp-convolution} for several cutoff-uniform and
infinite-support classes.  The unrestricted inequality remains open.
Importantly, it governs a stronger \(H^1\) monotonicity statement and is not
an obstruction to the global regularity theorem.

Section~2 proves the reduced PDE results.  Section~3 develops the independent
hidden-positivity problem through analytic-disk, critical-scale, shell, and
moment formulations.  The appendix records the local estimates used in the
PDE argument.
\section{The Reduced Equation}
Let \(P_0 f=f-\mean f\).  The Hilbert transform is normalized by \(\widehat{\Hilb f}(k)=-\ii\operatorname{sgn}(k)\hat f(k)\), so \(\widehat{\Lambda f}(k)=|k|\hat f(k)\).  Define
\[
  B(f,g)=P_0\{f\Lambda g-(\Hilb f)g_s\}.
\]
Then \eqref{eq:eta} is \(\eta_t+cm\Lambda\eta+\gamma\eta=-cB(\eta,\eta)\).
For mean-zero functions, we use the equivalent Sobolev norm
\(\norm f_{H^r}=\norm{\Lambda^r f}_2\); all estimates are unchanged if the
standard inhomogeneous norm is used instead.

\begin{lemma}[Fourier null structure]\label{lem:null}
For every nonzero output frequency \(n\),
\[
  \widehat{B(f,g)}(n)
  =\sum_{k+\ell=n}
  \bigl(|\ell|-\operatorname{sgn}(k)\ell\bigr)
  \hat f(k)\hat g(\ell).
\]
Thus each interaction with \(k\ell>0\) vanishes.
\end{lemma}

\begin{proof}
The identity follows from the Fourier symbols of \(\Lambda\), \(\Hilb\), and
\(\partial_s\).  If \(k\ell>0\), then
\(|\ell|=\operatorname{sgn}(k)\ell\), so the corresponding multiplier
vanishes.
\end{proof}

\begin{theorem}[Positive-density local theory]\label{thm:lwp}
Let \(r>3/2\) and let \(\rho_0=m+\eta_0\in H^r(\T)\) satisfy
\(\mean{\rho_0}=m\) and \(\rho_0\ge \kappa_0>0\).  Then \eqref{eq:rho}
has a unique solution on a time interval \([0,T]\),
\[
  \eta\in C([0,T];H^r)\cap L^2(0,T;H^{r+1/2}),
\]
and
\[
 \rho(t,s)\ge
 m+e^{-\gamma t}\bigl(\min_\T\rho_0-m\bigr)>0.
\]
The solution depends continuously on the data while the positive lower barrier
and \(H^r\) norm remain controlled.
\end{theorem}

\begin{proof}
Add \(-\varepsilon\rho_{ss}\) to the left side of \eqref{eq:rho} and use
smooth positive approximations of the initial data.  The resulting uniformly
parabolic problems have smooth local solutions, and the added viscosity
preserves the scalar maximum principle.  Write the equation for \(\eta\),
including its spatially constant normalization term, as
\[
 \eta_t+c\rho\Lambda\eta-c(\Hilb\rho)\eta_s+\gamma\eta
 =c\mean{\rho\Lambda\rho-(\Hilb\rho)\rho_s}.
\]
The right-hand side vanishes in the mean-zero energy estimate.
As long as \(\rho\ge\kappa>0\), the positive-coefficient estimate in
Lemma~\ref{lem:positive-coefficient-energy} yields
\[
 \frac{\dd}{\dd t}\norm{\eta}_{H^r}^2
 +\frac{c\kappa}{2}\norm{\eta}_{\dot H^{r+1/2}}^2
 \le C_{\kappa,m,c,\gamma,r}
 \bigl(1+\norm{\eta}_{H^r}\bigr)^4.
\]
This differential inequality supplies a uniform local lifespan while the
positive lower bound is retained.  At a
spatial minimum, \(\rho_s=0\) and \(\Lambda\rho\le0\), so the lower Dini
derivative of \(\rho_{\min}\) satisfies
\(D^-\rho_{\min}+\gamma(\rho_{\min}-m)\ge0\).  Compactness gives a solution
with the stated lower barrier.  The corresponding difference estimate gives
uniqueness and continuous dependence.
\end{proof}

\begin{corollary}[No rupture]\label{cor:no-rupture}
A smooth positive solution of \eqref{eq:rho} cannot lose positivity at a finite time.  More precisely,
\[
  \min_s\rho(t,s)\ge m+e^{-\gamma t}\bigl(\min_s\rho_0(s)-m\bigr)
\]
whenever the solution remains smooth.
\end{corollary}

\begin{theorem}[Arbitrary-size global regularity]\label{thm:large-global}
Assume \(c>0\) and \(\gamma\ge0\).  Let
\(\rho_0\in C^\infty(\T)\) satisfy
\[
 \rho_0>0,\qquad \frac1{2\pi}\int_\T\rho_0\,\dd s=m>0.
\]
Then the solution of \eqref{eq:rho} exists uniquely for all \(t\ge0\) and
remains smooth and positive.  More precisely,
\begin{equation}\label{eq:large-global-density-bounds}
\begin{aligned}
 m+e^{-\gamma t}\bigl(\min_\T\rho_0-m\bigr)
 &\le \rho(t,s)\\
 &\le
 m+\pi e^{-\gamma t}\norm{(\rho_0)_s}_{L^\infty},
\end{aligned}
\end{equation}
and
\begin{equation}\label{eq:large-global-slope-bound}
 \norm{\rho_s(t)}_{L^\infty}
 \le e^{-\gamma t}\norm{(\rho_0)_s}_{L^\infty}.
\end{equation}
In particular, no finite-time singularity occurs for a strictly positive
smooth initial density of arbitrary size.
\end{theorem}

\begin{proof}
The lower bound in \eqref{eq:large-global-density-bounds} is
Corollary~\ref{cor:no-rupture}.  The projection in \eqref{eq:rho} preserves
\(\mean\rho=m\).  Hence, once the slope estimate below is known, the
mean-value theorem on the circle gives
\[
 \norm{\rho(t)-m}_{L^\infty}
 \le\pi\norm{\rho_s(t)}_{L^\infty}
 \le\pi e^{-\gamma t}\norm{(\rho_0)_s}_{L^\infty},
\]
which proves the upper bound.

Set \(w=\rho_s\).  Differentiating \eqref{eq:rho} produces
\[
\begin{aligned}
 w_t
 &+c w\Lambda\rho+c\rho\Lambda w
 -c(\Lambda\rho)w-c(\Hilb\rho)w_s+\gamma w=0.
\end{aligned}
\]
The two mixed stretching terms cancel pointwise.  Hence
\begin{equation}\label{eq:exact-slope-equation}
 \boxed{\;
 w_t+c\rho\Lambda w-c(\Hilb\rho)w_s+\gamma w=0.
 \;}
\end{equation}
At a positive maximum of \(w\), one has \(w_s=0\) and
\(\Lambda w\ge0\).  The maximum principle applied to \(w\) and \(-w\)
therefore yields \eqref{eq:large-global-slope-bound}.

It remains to exclude a loss of higher regularity.  Suppose that the maximal
smooth lifespan is \(T_*<\infty\).  The preceding bounds give constants
\(\kappa,K,L>0\), independent of \(t<T_*\), such that
\[
 \kappa\le\rho\le K,\qquad \norm{\rho_s}_{L^\infty}\le L.
\]
Fix \(0<\alpha<1\).  Then \(\rho\) is uniformly \(C^\alpha\), and boundedness
of the periodic Hilbert transform on \(C^\alpha\) gives the same control for
\(\Hilb\rho\).  After removing the damping by
\(\widetilde w=e^{\gamma t}w\), equation
\eqref{eq:exact-slope-equation} is a uniformly elliptic order-one nonlocal
parabolic equation.  Its kernel is \(c\rho(t,s)K_0(h)\), where the periodic
kernel \(K_0\) of \(\Lambda\) is symmetric in \(h\); the full kernel need not
be symmetric in its spatial variables, but its ellipticity constants lie
between \(c\kappa\) and \(cK\).  The critical drift \(-c\Hilb\rho\) is
bounded.  Interior H\"older estimates for critical-drift
equations \cite{ChangLaraDavila2016} therefore give a uniform parabolic
\(C^\beta\) bound for \(\widetilde w\) on every strip \([\tau,T_*)\), for
some \(\beta>0\).  Because \(\mean\rho=m\) and \(\rho_s=w\), the periodic
zero-mean primitive recovers \(\rho-m\) from \(w\); hence \(\rho\) and
\(\Hilb\rho\) have the same time H\"older control and at least the required
spatial H\"older control.  The order-one Schauder estimates with drift
\cite{DongJinZhang2018} then give a uniform \(C^{1+\beta'}\) bound for
\(w\), for some \(0<\beta'<\beta\).  The coefficients consequently improve
by one spatial derivative.  Differentiating the linear equation for \(w\)
and iterating the same Schauder estimate bounds every spatial derivative on
\([\tau,T_*)\).
The equation bounds the corresponding time derivatives.  Thus
\(\rho(t)\) has a smooth limit as \(t\uparrow T_*\), and the local theory
extends it past \(T_*\), a contradiction.
\end{proof}

\begin{theorem}[Small-oscillation quantitative stability]\label{thm:small}
Let \(r>3/2\).  There exists \(\delta>0\), depending only on
\(m,c,\gamma,r\), such that if \(\norm{\eta_0}_{H^r}\le\delta\), then the
solution of \eqref{eq:eta} is global and
\[
  \sup_{t\ge0}\norm{\eta(t)}_{H^r}^2+
  cm\int_0^\infty
  \norm{\eta(t)}_{\dot H^{r+1/2}}^2\,\dd t
  \le C\norm{\eta_0}_{H^r}^2 .
\]
\end{theorem}

\begin{proof}
Testing \eqref{eq:eta} against \(\Lambda^{2r}\eta\) and using
Lemma~\ref{lem:commutator-short} gives
\[
 \frac{\dd}{\dd t}\norm{\eta}_{H^r}^2
 +2cm\norm{\eta}_{\dot H^{r+1/2}}^2
 +2\gamma\norm{\eta}_{H^r}^2
 \le C\norm{\eta}_{H^r}\norm{\eta}_{\dot H^{r+1/2}}^2.
\]
Choosing \(\delta\) small absorbs half of the dissipative term, and the usual
continuity argument gives global control.
\end{proof}

\section{Hidden Positivity and the Sharp Cubic Problem}
\label{sec:hidden-positivity-open}
\subsection{Exact identities and the sharp cubic problem}
The large-data global theorem does not require the following
derivative-energy sign: the slope maximum principle already prevents
singularity formation.  The sign instead asks whether the \(H^1\) energy
obeys an exact monotonicity law and whether that law has a sharp coercive
form.  For smooth positive solutions define
\[
  X(t)=\norm{\rho_s(t)}_2^2,
  \quad
  I_1=\int_\T \rho(\Lambda\rho_s)\rho_s\,\dd s,
  \quad
  I_2=\int_\T(\Lambda\rho)\rho_s^2\,\dd s,
  \qquad
  \mathfrak H(\rho):=2I_1+I_2.
\]
Then
\begin{equation}\label{eq:X}
  \frac12\frac{\dd X}{\dd t}
  =-\gamma X-\frac c2\mathfrak H(\rho).
\end{equation}
We call \(\mathfrak H(\rho)\ge0\) the strong hidden-positivity property.
The section has three distinct logical levels.  Theorem~
\ref{thm:third-order-hidden-positivity} proves the unconditional sign of
\(I_1\).  The later results prove \(\mathfrak H(\rho)\ge0\) on specified
cutoff-uniform or infinite-support classes.  The unrestricted phase-opposed
inequality is stated as Problem~\ref{prob:sharp-cubic} and remains open.
None of these stronger \(H^1\) statements is used in
Theorem~\ref{thm:large-global}.  Quantitative claims identified as
computer-assisted are accompanied by ancillary verification scripts.
Two exact identities isolate the remaining sign issue.  Since
\(\Lambda^3=-\partial_s^2\Lambda\), periodic integration by parts gives
\begin{equation}\label{eq:I1-third-order}
  I_1=\frac12\int_\T \rho^2\Lambda^3\rho\,\dd s.
\end{equation}
Moreover, self-adjointness of \(\Lambda\) and the C\'ordoba--C\'ordoba
inequality applied to \(f=\rho_s\) give
\begin{equation}\label{eq:cordoba-I1-I2}
  2I_1-I_2
  =\int_\T \rho\bigl(2\rho_s\Lambda\rho_s-\Lambda(\rho_s^2)\bigr)\,\dd s
  \ge0.
\end{equation}

The first sign question has a degree-independent answer.

\begin{theorem}[Third-order hidden positivity]
\label{thm:third-order-hidden-positivity}
Let \(\rho\in C^\infty(\T)\) be nonnegative.  Then
\begin{equation}\label{eq:third-order-hidden-positivity}
  \int_\T \rho^2\Lambda^3\rho\,\dd s\ge0,
  \qquad I_1\ge0.
\end{equation}
If, in addition, \(\rho\ge\kappa>0\), then
\begin{equation}\label{eq:I1-coercivity}
  I_1\ge\kappa\norm{\Lambda^{3/2}\rho}_2^2.
\end{equation}
Equality holds if and only if \(\rho\) is constant.
\end{theorem}

\begin{proof}
For \(h>0\), let \(\tau_hf(s)=f(s+h)\) and set
\[
  L_h=2-\tau_h-\tau_{-h},
  \qquad
  \delta_h^2\rho(s)=\rho(s-h)-2\rho(s)+\rho(s+h)=-L_h\rho(s).
\]
The elementary Fourier-multiplier identity
\[
  |k|^3=\frac{3}{2\pi}\int_0^\infty
  \frac{16\sin^4(kh/2)}{h^4}\,\dd h
\]
follows from \(\int_0^\infty(\sin x/x)^4\,\dd x=\pi/3\).  Hence
\begin{equation}\label{eq:Lambda3-fourth-difference}
  \Lambda^3f=\frac{3}{2\pi}\int_0^\infty
  \frac{L_h^2f}{h^4}\,\dd h .
\end{equation}
At each fixed scale, expansion and translation invariance of \(\dd s\) give
the exact sum-of-squares identity
\begin{equation}\label{eq:fourth-difference-sos}
\begin{aligned}
  \int_\T\rho^2L_h^2\rho\,\dd s
  =\frac13\int_\T
  \bigl(\rho(s-h)+4\rho(s)+\rho(s+h)\bigr)
  \bigl(\delta_h^2\rho(s)\bigr)^2\,\dd s.
\end{aligned}
\end{equation}
Combining \eqref{eq:Lambda3-fourth-difference} and
\eqref{eq:fourth-difference-sos} yields the explicit representation
\begin{equation}\label{eq:I1-positive-representation}
  I_1=\frac{1}{4\pi}\int_0^\infty\frac{1}{h^4}
  \int_\T\bigl(\rho(s-h)+4\rho(s)+\rho(s+h)\bigr)
  \bigl(\delta_h^2\rho(s)\bigr)^2\,\dd s\,\dd h\ge0.
\end{equation}
When \(\rho\ge\kappa\), the weight in \eqref{eq:I1-positive-representation}
is at least \(6\kappa\).  Since \eqref{eq:Lambda3-fourth-difference} and
self-adjointness of \(L_h\) give
\[
  \norm{\Lambda^{3/2}\rho}_2^2
  =\frac{3}{2\pi}\int_0^\infty\frac{1}{h^4}
  \int_\T\bigl(\delta_h^2\rho\bigr)^2\,\dd s\,\dd h,
\]
we obtain \eqref{eq:I1-coercivity}.
The integral converges because \(\delta_h^2\rho=O(h^2)\) as \(h\to0\),
and periodicity controls the integrand as \(h\to\infty\).  If equality
holds, the nonnegative integrand forces \(\delta_h^2\rho\equiv0\) for every
\(h\); equivalently, every nonzero Fourier coefficient of \(\rho\) vanishes.
Thus \(\rho\) is constant.  The converse is immediate.
\end{proof}

The remaining combination has three useful exact reformulations.  Let
\(P_t=e^{-t\Lambda}\) denote the periodic Poisson semigroup.

\begin{lemma}[Cubic and Poisson-semigroup forms]
\label{lem:strong-equivalent-forms}
For every smooth positive periodic \(\rho\),
\begin{equation}\label{eq:I2-cubic-Lambda}
  I_2=\frac13\int_\T(\Lambda\rho)^3\,\dd s,
\end{equation}
and
\begin{equation}\label{eq:strong-semigroup-form}
  2I_1+I_2
  =-\left.\frac{\dd}{\dd t}\right|_{t=0}
  \int_\T(P_t\rho)\abs{\partial_sP_t\rho}^2\,\dd s
  =-\frac49\left.\frac{\dd}{\dd t}\right|_{t=0}
  \norm{\partial_s(P_t\rho)^{3/2}}_2^2.
\end{equation}
\end{lemma}

\begin{proof}
Set \(f=\rho_s\) and \(q=\Hilb f=\Lambda\rho\).  The periodic Tricomi
identity \(q^2-f^2=2\Hilb(fq)\), skew-adjointness of \(\Hilb\), and
\(\Hilb q=-f\) give
\[
  \int_\T q^3\,\dd s
  =\int_\T qf^2\,\dd s+2\int_\T q\Hilb(fq)\,\dd s
  =3\int_\T qf^2\,\dd s=3I_2.
\]
For \(u(t)=P_t\rho\), we have \(u_t=-\Lambda u\), and differentiation gives
\[
  \frac{\dd}{\dd t}\int_\T u u_s^2\,\dd s
  =-\int_\T(\Lambda u)u_s^2\,\dd s
   -2\int_\T u u_s\Lambda u_s\,\dd s.
\]
Evaluation at \(t=0\) proves the first equality in
\eqref{eq:strong-semigroup-form}; the second follows from
\(u u_s^2=(4/9)\abs{\partial_su^{3/2}}^2\).
\end{proof}

\begin{lemma}[Fourier representation of the strong bracket]
\label{lem:strong-fourier-representation}
Let
\[
 \rho(s)=m+\sum_{n\ne0}r_ne^{\ii ns},
 \qquad r_{-n}=\overline{r_n}.
\]
Then
\begin{equation}\label{eq:strong-positive-frequency-form}
\begin{aligned}
 \mathfrak H(\rho)
 =8\pi\biggl[
 m\sum_{n\ge1}n^3\abs{r_n}^2
 +\sum_{a,b\ge1}(a+b)(a^2+ab+b^2)
 \Re\bigl(r_ar_b\overline{r_{a+b}}\bigr)
 \biggr].
\end{aligned}
\end{equation}
In particular, for a finitely supported nonnegative sequence \(x\), define
 \[
 \begin{aligned}
 A&=\sum_{n\ge1}x_n,
 &E_3&=\sum_{n\ge1}n^3x_n^2,\\
 \mathcal C(x)&=\sum_{a,b\ge1}
 (a+b)(a^2+ab+b^2)x_ax_bx_{a+b},
 &\rho_x(s)&=2A-2\sum_{n\ge1}x_n\cos(ns).
 \end{aligned}
 \]
Then \(\rho_x\ge0\) and
\begin{equation}\label{eq:phase-opposed-strong-bracket}
 \mathfrak H(\rho_x)=8\pi\bigl(2AE_3-\mathcal C(x)\bigr).
\end{equation}
\end{lemma}

\begin{proof}
Symmetrize the cubic Fourier multipliers over triples summing to zero.
Nonzero triples \((a,b,-a-b)\) give the second term in
\eqref{eq:strong-positive-frequency-form}, while triples
\((n,-n,0)\) give the mean term.  For \(\rho_x\), one has
\(m=2A\) and \(r_n=-x_n\) for \(n\ge1\), which proves
\eqref{eq:phase-opposed-strong-bracket}.  The inequality
\(\rho_x\ge0\) follows from
\(2\sum_nx_n\cos(ns)\le2A\).
\end{proof}

\begin{problem}[Sharp cubic convolution inequality]
\label{prob:sharp-cubic}
For every finitely supported sequence \(x_n\ge0\), prove or disprove
\begin{equation}\label{eq:sharp-cubic-problem}
 \mathcal C(x)
 \le
 2\left(\sum_{n\ge1}x_n\right)
  \left(\sum_{n\ge1}n^3x_n^2\right).
\end{equation}
\end{problem}

It is enough to solve Problem~\ref{prob:sharp-cubic} as stated.  Indeed, if
\(x_n\ge0\), \(A<\infty\), and \(E_3<\infty\), then applying
\eqref{eq:sharp-cubic-problem} to finite truncations and passing to the limit
by monotone convergence gives the same inequality for the full infinite
sequence.  Thus the finite-support formulation already contains the
admissible infinite-support case; the difficulty is uniformity in the cutoff.

\begin{proposition}[All-frequency positivity for power cusps, including the critical endpoint]
\label{prop:strong-power-cusp-positivity}
For \(2/3\le\alpha<1\), let
\[
  g_\alpha(s)=\abs{\sin(s/2)}^\alpha
  =\sum_{n\in\mathbb Z}c_{\abs n}^{(\alpha)}e^{\ii ns}.
\]
Then the strong combination, interpreted by its absolutely convergent
combined Fourier pairing, satisfies
\begin{equation}\label{eq:power-cusp-positive-series}
\begin{aligned}
  2I_1(g_\alpha)+I_2(g_\alpha)
  =4\pi\sum_{n\ge1}
  n c_n^{(\alpha)}c_n^{(2\alpha)}
  \frac{\alpha^2+(6\alpha-4)n^2}{4(2\alpha-1)}>0.
\end{aligned}
\end{equation}
Here the left-hand side denotes this combined pairing; separate convergence
of \(I_1(g_\alpha)\) and \(I_2(g_\alpha)\) is not asserted.
In particular, this is a sign theorem for an infinite Fourier tail; no
finite-degree verification enters the proof.  At the critical endpoint,
the formula reduces to the absolutely convergent strict identity
\begin{equation}\label{eq:critical-power-cusp-positive-series}
 2I_1(g_{2/3})+I_2(g_{2/3})
 =\frac{4\pi}{3}\sum_{n\ge1}
 n c_n^{(2/3)}c_n^{(4/3)}>0.
\end{equation}
\end{proposition}

\begin{proof}
The beta integral gives, for every \(\beta>-1\),
\begin{equation}\label{eq:power-cusp-fourier-coefficients}
  c_n^{(\beta)}
  =\frac{(-1)^n\Gamma(\beta+1)}
  {2^\beta
   \Gamma(\beta/2-n+1)\Gamma(\beta/2+n+1)}.
\end{equation}
If \(0<\beta<2\), then \(c_n^{(\beta)}<0\) for every \(n\ge1\).  Indeed,
\[
  c_1^{(\beta)}<0,
  \qquad
  \frac{c_{n+1}^{(\beta)}}{c_n^{(\beta)}}
  =\frac{n-\beta/2}{n+\beta/2+1}>0
  \quad(n\ge1).
\]

Since \(g_\alpha^2=g_{2\alpha}\) and
\[
  (\partial_sg_\alpha)^2
  =\frac{\alpha^2}{4}
   \left(g_{2\alpha-2}-g_{2\alpha}\right),
\]
self-adjointness of \(\Lambda\), Parseval's identity, and
\eqref{eq:I1-third-order} yield
\begin{equation}\label{eq:power-cusp-pre-reduction}
\begin{aligned}
  2I_1(g_\alpha)+I_2(g_\alpha)
  =4\pi\sum_{n\ge1}c_n^{(\alpha)}
  \left[
    n^3c_n^{(2\alpha)}
    +\frac{\alpha^2n}{4}
     \left(c_n^{(2\alpha-2)}-c_n^{(2\alpha)}\right)
  \right].
\end{aligned}
\end{equation}
The coefficient formula also gives the exact two-step shift
\[
  \frac{c_n^{(2\alpha-2)}}{c_n^{(2\alpha)}}
  =\frac{2(\alpha^2-n^2)}{\alpha(2\alpha-1)}.
\]
Substitution into \eqref{eq:power-cusp-pre-reduction} proves
\eqref{eq:power-cusp-positive-series}.  Both Fourier coefficients in each
product on the right are negative.  For \(\alpha>2/3\), the factor
\(\alpha^2+(6\alpha-4)n^2\) is positive and, since
\(c_n^{(\beta)}=O(n^{-\beta-1})\), the summand is
\(O(n^{1-3\alpha})\), hence absolutely summable.  At \(\alpha=2/3\),
the entire \(n^2\) coefficient vanishes and the remaining summand is
\(O(n^{-3})\).  This proves absolute convergence and
\eqref{eq:critical-power-cusp-positive-series}.  The strict sign follows
term by term throughout the interval, including its critical endpoint.
\end{proof}

\subsection{Analytic-disk and Schur reductions}

The fully phase-opposed class also admits an exact analytic-disk
reformulation.  Unlike a cutoff calculation, this reformulation keeps the
whole tail in one Dirichlet integral and isolates the sharp constant \(2\) as
a single boundary-flux estimate.

\begin{proposition}[Disk--Rellich reduction for an anticoherent tail]
\label{prop:anticoherent-disk-reduction}
Let \(x_n\ge0\) be finitely supported, and put
\[
 A=\sum_{n\ge1}x_n,
 \qquad
 E_3=\sum_{n\ge1}n^3x_n^2,
 \qquad
 f(z)=\sum_{n\ge1}x_nz^n,
 \qquad
 g=Df=zf'(z).
\]
Define
\[
 \mathcal C(x)=\sum_{a,b\ge1}
 (a+b)(a^2+ab+b^2)x_ax_bx_{a+b}
\]
and, on the unit disk, set
\[
 u(z)=A-\Re f(z)\ge0,
 \qquad
 M(r)=\frac1{2\pi}\int_0^{2\pi}
 u(re^{\ii\theta})\abs{g(re^{\ii\theta})}^2\,\dd\theta.
\]
With normalized area measure \(\dd A_0=\dd A/\pi\), one has the exact
identities
\begin{equation}\label{eq:anticoherent-disk-deficit}
\begin{aligned}
 A E_3-\frac12\mathcal C(x)
 &=\int_{\mathbb D}u\abs{g'}^2\,\dd A_0
   -\frac12\Re\int_{\mathbb D}f'g\overline{g'}\,\dd A_0\\
 &=\frac12\int_{\mathbb D}u\abs{g'}^2\,\dd A_0
   +\frac14M'(1)\\
 &=\frac12\left.
   \frac1{2\pi}\int_0^{2\pi}
   u(re^{\ii\theta})\,
   \partial_r\abs{g(re^{\ii\theta})}^2\,\dd\theta
   \right|_{r=1}.
\end{aligned}
\end{equation}
Consequently, the sharp convolution inequality
\begin{equation}\label{eq:anticoherent-sharp-convolution}
 \mathcal C(x)\le2AE_3
\end{equation}
is equivalent to the cutoff-free weighted Rellich estimate
\begin{equation}\label{eq:anticoherent-rellich-target}
 M'(1)+2\int_{\mathbb D}u\abs{g'}^2\,\dd A_0\ge0.
\end{equation}
Equivalently, it is the single boundary-flux sign
\begin{equation}\label{eq:anticoherent-boundary-flux}
 \left.
 \int_0^{2\pi}u(re^{\ii\theta})\,
 \partial_r\abs{g(re^{\ii\theta})}^2\,\dd\theta
 \right|_{r=1}\ge0.
\end{equation}
The same equivalence extends by truncation whenever the displayed sums and
integrals converge absolutely.
\end{proposition}

\begin{proof}
Introduce the two ordered triad forms
\[
 L=\sum_{a,b\ge1}a^3x_ax_bx_{a+b},
 \qquad
 J=\sum_{a,b\ge1}ab(a+b)x_ax_bx_{a+b}.
\]
Symmetry in \(a,b\) gives
\begin{equation}\label{eq:anticoherent-C-LJ}
 \mathcal C(x)=2(L+J).
\end{equation}
Monomial orthogonality in \(\mathbb D\) gives, without an estimate,
\begin{equation}\label{eq:anticoherent-area-ledger}
\begin{aligned}
 \int_{\mathbb D}\abs{g'}^2\,\dd A_0&=E_3,\\
 \Re\int_{\mathbb D}f\abs{g'}^2\,\dd A_0&=L+\frac12J,\\
 \Re\int_{\mathbb D}f'g\overline{g'}\,\dd A_0&=J.
\end{aligned}
\end{equation}
The first line of \eqref{eq:anticoherent-disk-deficit} follows from
\eqref{eq:anticoherent-C-LJ}--\eqref{eq:anticoherent-area-ledger}.

For the second line, let \(h=A-f\), so \(u=\Re h\), and set
\(\Phi=u|g|^2\).  Since \(h'=-f'\), direct complex differentiation yields
\[
 \partial_z\partial_{\bar z}\Phi
 =u\abs{g'}^2-\Re\bigl(f'g\overline{g'}\bigr).
\]
Green's identity and the definition of \(M\) give
\[
 \int_{\mathbb D}\partial_z\partial_{\bar z}\Phi\,\dd A_0
 =\frac12M'(1).
\]
Averaging this equality with the first line of
\eqref{eq:anticoherent-disk-deficit} proves the second.  Green's second
identity, applied to the harmonic function \(u\) and the subharmonic
function \(|g|^2\), gives
\[
 2\int_{\mathbb D}u|g'|^2\,\dd A_0
 =
 \left.\frac1{2\pi}\int_0^{2\pi}
 u\,\partial_r|g|^2-|g|^2\partial_r u\,\dd\theta\right|_{r=1}.
\]
Here \(\partial_r u|_{r=1}=-\Re g\), while monomial orthogonality gives
\[
 \frac1{2\pi}\int_0^{2\pi}\Re g\,|g|^2\,\dd\theta=J.
\]
Combining this with the first line of
\eqref{eq:anticoherent-disk-deficit} proves the boundary formula.  The
equivalence of
\eqref{eq:anticoherent-sharp-convolution} and
\eqref{eq:anticoherent-rellich-target}--\eqref{eq:anticoherent-boundary-flux}
is immediate.
\end{proof}

\begin{remark}
The volume term in \eqref{eq:anticoherent-rellich-target} cannot simply be
dropped: there are positive finite spectra for which \(M'(1)<0\).  Thus the
remaining degree-free step is the combined Rellich estimate exactly as
stated, not monotonicity of the boundary mean by itself.

There is also an exact translation-scale ledger for the same boundary flux.
If \(a_n=nx_n\), with \(a_n=0\) outside the positive support, then
\[
 A E_3-\frac12\mathcal C(x)
 =\frac12\sum_{k\ge1}x_k Q_k,
 \qquad
 Q_k=2\sum_{n\ge1}na_n^2
 -\sum_{n\ge1}(2n+k)a_na_{n+k},
\]
and
\[
 Q_k=
 \sum_{n\ge1}\left(n+\frac k2\right)(a_{n+k}-a_n)^2
 +\sum_{n=1}^{k}\left(n-\frac k2\right)a_n^2.
\]
The first term is a positive translation-difference square, whereas the
finite boundary wedge changes sign.  Thus positivity of individual
translation scales is false in general; the combined boundary flux must
cancel these wedges across scales.
\end{remark}

\begin{theorem}[Sharp arbitrary-tail theorem under comonotone spectral envelopes]
\label{thm:anticoherent-monotone-envelope}
Let \(x_n\ge0\) be finitely supported and suppose that the first- and
second-order spectral envelopes are comonotone:
\begin{equation}\label{eq:anticoherent-envelope-comonotonicity}
 \bigl(nx_n-mx_m\bigr)
 \bigl(n^2x_n-m^2x_m\bigr)\ge0
 \qquad(m,n\ge1).
\end{equation}
A convenient sufficient condition is
\begin{equation}\label{eq:anticoherent-envelope-monotonicity}
 (n+1)^2x_{n+1}\le n^2x_n
 \qquad(n\ge1).
\end{equation}
Then
\begin{equation}\label{eq:anticoherent-monotone-sharp}
 \mathcal C(x)\le2A E_3.
\end{equation}
More precisely, every translation scale in the ledger above is nonnegative:
\begin{equation}\label{eq:anticoherent-monotone-Qk}
 Q_k=
 \sum_{n\ge1}
 \bigl(nx_n-(n+k)x_{n+k}\bigr)
 \bigl(n^2x_n-(n+k)^2x_{n+k}\bigr)
 +\sum_{j=1}^{k}j^3x_j^2\ge0.
\end{equation}
The conclusion extends to every infinite sequence satisfying
\eqref{eq:anticoherent-envelope-comonotonicity} for which \(A<\infty\) and
\(E_3<\infty\).  The constant \(2\) in
\eqref{eq:anticoherent-monotone-sharp} is sharp even within this cone.
Consequently, every fully anticoherent spectrum \(r_n=-x_n\) in this class
has nonnegative strong bracket whenever its minimum constraint gives
\(m\ge2A\).
\end{theorem}

\begin{proof}
Put \(u_n=nx_n\) and \(v_n=n^2x_n\), and extend both sequences by zero.
For \(c=n+k\), direct expansion gives
\[
 (u_n-u_c)(v_n-v_c)
 =n^3x_n^2+c^3x_c^2-nc(n+c)x_nx_c.
\]
Summing in \(n\), and observing that the shifted diagonal sum omits exactly
the first \(k\) modes, gives the identity
\eqref{eq:anticoherent-monotone-Qk}.  Condition
\eqref{eq:anticoherent-envelope-comonotonicity} says that both differences
in every product have the same sign, so \(Q_k\ge0\).  If instead one assumes
the simpler condition \eqref{eq:anticoherent-envelope-monotonicity}, then
\(v_n\) is nonincreasing and \(u_n=v_n/n\) is nonincreasing as well; hence
the comonotonicity condition follows.  The translation ledger now yields
\eqref{eq:anticoherent-monotone-sharp}.

For an infinite sequence, apply the finite result to the truncations
\(x_n\mathbf 1_{n\le N}\).  The monotonicity assumption is preserved, while
monotone convergence gives the stated limit.  To see sharpness, take
\[
 x_n^{(N)}=n^{-2}\mathbf 1_{n\le N}.
\]
Then \(A_N=\zeta(2)+o(1)\), \(E_{3,N}=H_N\), and
\[
 \begin{aligned}
 \mathcal C(x^{(N)})
 &=\sum_{a+b\le N}
 \left(\frac1{b^2(a+b)}+\frac1{ab(a+b)}
       +\frac1{a^2(a+b)}\right)\\
 &=2\zeta(2)H_N+O(1).
 \end{aligned}
\]
Indeed, the first and third sums equal
\(2\sum_{a<N}a^{-2}(H_N-H_a)\), whereas the middle sum is bounded by
\(2\sum_{c\ge2}H_{c-1}/c^2<\infty\).  Hence
\(\mathcal C(x^{(N)})/(A_NE_{3,N})\to2\).  The final bracket statement is
immediate from \(mE_3-\mathcal C(x)\ge(m-2A)E_3\).
\end{proof}

The disk deficit admits a second degree-free reduction which retains the
full multiplier defect of the probability generating function.  Unlike an
ordinary Cauchy--Schwarz estimate, it does not discard the output support.

\begin{proposition}[Schur-defect reduction for an anticoherent tail]
\label{prop:anticoherent-schur-defect}
Let \(x_n\ge0\) be finitely supported, \(A=\sum_nx_n>0\), and set
\[
 p_n=\frac{x_n}{A},\qquad
 F(z)=\sum_{n\ge1}p_nz^n,\qquad G=DF.
\]
Thus \(F(0)=0\), \(F(1)=1\), and \(\abs{F(z)}\le1\) on the disk.  Define
\begin{equation}\label{eq:anticoherent-schur-defect}
 \mathcal S(p)=
 \int_{\mathbb D}(1-\abs F^2)\abs{G'}^2\,\dd A_0
 -\Re\int_{\mathbb D}F'G\overline{G'}\,\dd A_0.
\end{equation}
Then the sharp deficit has the exact decomposition
\begin{equation}\label{eq:anticoherent-schur-defect-ledger}
 \frac1{A^3}\left(AE_3-\frac12\mathcal C(x)\right)
 =\frac12\mathcal S(p)
  +\frac12\norm{(1-F)G'}_{A^2(\mathbb D)}^2.
\end{equation}
In coefficients,
\begin{equation}\label{eq:anticoherent-schur-defect-coefficients}
\begin{aligned}
 \mathcal S(p)
 &=\sum_{n\ge1}n^3p_n^2
 -\sum_{c\ge2}\frac1c
   \left(\sum_{a+b=c}b^2p_ap_b\right)^2
 -\sum_{a,b\ge1}ab(a+b)p_ap_bp_{a+b}.
\end{aligned}
\end{equation}
Moreover, if \(H=G'\), the first two terms have the Hilbert-space variance
decomposition
\begin{equation}\label{eq:anticoherent-schur-variance}
\begin{aligned}
 \norm H_{A^2}^2-\norm{FH}_{A^2}^2
 &=\sum_{a\ge1}p_a
   \left(\norm H_{A^2}^2-\norm{z^aH}_{A^2}^2\right)\\
 &\quad+\frac12\sum_{a,b\ge1}p_ap_b
   \norm{(z^a-z^b)H}_{A^2}^2.
\end{aligned}
\end{equation}
Consequently, the unrestricted sharp convolution theorem would follow from
the single multiplier-defect estimate
\begin{equation}\label{eq:anticoherent-schur-target}
 \Re\int_{\mathbb D}F'G\overline{G'}\,\dd A_0
 \le \norm H_{A^2}^2-\norm{FH}_{A^2}^2
\end{equation}
for every probability generating polynomial \(F\).  This last estimate is a
strictly stronger target and is not asserted here.
\end{proposition}

\begin{proof}
After division by \(A\), the first disk identity in
Proposition~\ref{prop:anticoherent-disk-reduction} gives
\[
 \frac1{A^3}\left(AE_3-\frac12\mathcal C(x)\right)
 =\int_{\mathbb D}(1-\Re F)\abs{G'}^2\,\dd A_0
  -\frac12\Re\int_{\mathbb D}F'G\overline{G'}\,\dd A_0.
\]
The scalar identity
\[
 2(1-\Re F)=1-\abs F^2+\abs{1-F}^2
\]
proves \eqref{eq:anticoherent-schur-defect-ledger}.  Monomial
orthogonality gives
\[
 \norm{G'}_{A^2}^2=\sum_n n^3p_n^2,
 \qquad
 \norm{FG'}_{A^2}^2
 =\sum_{c\ge2}\frac1c
   \left(\sum_{a+b=c}b^2p_ap_b\right)^2,
\]
whereas the remaining inner product is the final cubic sum in
\eqref{eq:anticoherent-schur-defect-coefficients}.  This proves the
coefficient formula.

Finally, multiplication by \(F=\sum_ap_az^a\) is the average of the Bergman
shifts \(H\mapsto z^aH\).  Expanding the square of that average gives
\[
 \sum_ap_a\norm{z^aH}_{A^2}^2-\norm{FH}_{A^2}^2
 =\frac12\sum_{a,b}p_ap_b\norm{(z^a-z^b)H}_{A^2}^2.
\]
Adding
\(\sum_ap_a(\norm H_{A^2}^2-\norm{z^aH}_{A^2}^2)\)
proves \eqref{eq:anticoherent-schur-variance}.  The final implication follows
from \eqref{eq:anticoherent-schur-defect-ledger}.
\end{proof}

\begin{proposition}[Half-cylinder flux, sum-free closure, and zero gap]
\label{prop:anticoherent-pgf-half-cylinder}
For a probability generating polynomial, put
\[
 F_t(\zeta)=F(e^{-t}\zeta),\qquad
 G_t=DF_t,
 \qquad \abs\zeta=1.
\]
Then
\begin{equation}\label{eq:anticoherent-pgf-half-cylinder}
 \mathcal S(p)=
 2\int_0^\infty\frac1{2\pi}\int_0^{2\pi}
 \left\{
 (1-\abs{F_t}^2)\abs{DG_t}^2
 -\Re\left(G_t^2\overline{DG_t}\right)
 \right\}\,\dd\theta\dd t.
\end{equation}
Consequently, if the support of \(p\) is sum-free,
\begin{equation}\label{eq:anticoherent-pgf-sum-free-support}
 a,b\in\operatorname{supp}p
 \quad\Longrightarrow\quad
 a+b\notin\operatorname{supp}p,
\end{equation}
then \(\mathcal S(p)>0\).  The conclusion extends to every infinite
sum-free probability law satisfying \(\sum_nn^3p_n^2<\infty\).  In
particular, it applies with arbitrary amplitudes to every probability law
supported on the odd positive integers, or on any sum-free residue class.

There is nevertheless no constant \(c>0\) such that
\(\mathcal S(p)\ge c\norm{G'}_{A^2}^2\) for all probability generating
polynomials.  More precisely, for \(N\ge3\), let
\begin{equation}\label{eq:anticoherent-pgf-two-scale-family}
 p^{(N)}=(1-\varepsilon_N)\delta_1+\varepsilon_N\delta_N,
 \qquad \varepsilon_N=\frac\lambda N,
 \qquad \lambda>0.
\end{equation}
Then
\begin{equation}\label{eq:anticoherent-pgf-two-scale-asymptotic}
 \begin{aligned}
 \norm{G'}_{A^2}^2
 &=\lambda^2N+1+O_\lambda(N^{-1}),\\
 \mathcal S(p^{(N)})
 &=\frac12+\lambda^2+2\lambda^3+O_\lambda(N^{-1}),\\
 \frac{\mathcal S(p^{(N)})}{\norm{G'}_{A^2}^2}
 &=\frac{1+(2\lambda^2)^{-1}+2\lambda}{N}
   +O_\lambda(N^{-2}).
 \end{aligned}
\end{equation}
The leading ratio is minimized at \(\lambda=2^{-1/3}\), where its
coefficient is
\begin{equation}\label{eq:anticoherent-pgf-two-scale-optimal-coefficient}
 1+2^{-1/3}+2^{2/3}.
\end{equation}
\end{proposition}

\begin{proof}
Write \(z=e^{-t+\ii\theta}\).  Since
\(\dd A_0=2e^{-2t}\dd t\,\dd\theta/(2\pi)\) and
\(G'=z^{-1}DG\), the first term in
\eqref{eq:anticoherent-schur-defect} becomes the first half-cylinder term
in \eqref{eq:anticoherent-pgf-half-cylinder}.  Likewise,
\(F'G=z^{-1}G^2\), so the second disk term becomes the second
half-cylinder term.  This proves the identity.

If the support is sum-free, every Fourier frequency in \(G_t^2\) is absent
from \(DG_t\).  Angular orthogonality therefore annihilates the cubic flux
at every \(t\), while the remaining integrand is strictly positive on a set
of positive measure.  Approximation by normalized truncations proves the
infinite-support assertion.

For the two-scale family, put \(q=1-\varepsilon_N\).  It is sum-free for
\(N\ge3\), and direct coefficient summation gives
\begin{align*}
 \norm{G'}_{A^2}^2&=q^2+N^3\varepsilon_N^2,\\
 \norm{FG'}_{A^2}^2
 &=\frac12q^4
 +\frac{(N^2+1)^2}{N+1}q^2\varepsilon_N^2
 +\frac12N^3\varepsilon_N^4.
\end{align*}
The contact term vanishes.  Substituting
\(\varepsilon_N=\lambda/N\) and expanding proves
\eqref{eq:anticoherent-pgf-two-scale-asymptotic}.  Finally, the derivative
of \(1+(2\lambda^2)^{-1}+2\lambda\) vanishes only at
\(\lambda=2^{-1/3}\), proving
\eqref{eq:anticoherent-pgf-two-scale-optimal-coefficient}.
\end{proof}

\begin{remark}[The compensation is genuinely cross-scale]
The angular mean of the integrand in
\eqref{eq:anticoherent-pgf-half-cylinder} need not be nonnegative at a fixed
height.  Indeed, for
\[
 (p_1,p_2,p_3,p_4)=\frac1{100}(77,16,6,1)
\]
its value at \(t=0\) is exactly
\[
 \frac{10020441}{50000000}-\frac{11036}{15625}
 =-\frac{25294759}{50000000}<0,
\]
whereas the integrated defect is
\[
 \mathcal S(p)=\frac{1483975039}{7000000000}>0.
\]
Thus a proof of the unrestricted half-cylinder inequality must transfer
dissipation between radial scales; pointwise positivity in \(t\) is false
even inside the coefficient cone.
\end{remark}

The coefficient cone is essential in this formulation.  In particular,
the conjecture is not an abstract consequence of the Schur property.

\begin{conjecture}[Schur--Rellich boundary flux]
\label{conj:anticoherent-schur-rellich}
Let \(F\) be analytic in a neighborhood of \(\overline{\mathbb D}\) and
satisfy
\[
 F(0)=0,\qquad F(1)=1,\qquad \abs{F(z)}\le1
 \quad(z\in\mathbb D).
\]
Set \(G=DF= zF'(z)\).  Then
\begin{equation}\label{eq:anticoherent-schur-rellich-boundary}
 \mathcal R(F):=
 \left.\frac1{4\pi}\int_0^{2\pi}
 (1-\Re F(re^{\ii\theta}))
 \partial_r\abs{G(re^{\ii\theta})}^2\,\dd\theta
 \right|_{r=1}
 \ge0.
\end{equation}
\end{conjecture}

\begin{proposition}[Hilbert-disk equivalence]
\label{prop:anticoherent-hilbert-disk-equivalence}
Let \(F\) be as in
Conjecture~\ref{conj:anticoherent-schur-rellich}, and on \(\T\) put
\begin{equation}\label{eq:anticoherent-hilbert-disk-density}
 u=1-\Re F.
\end{equation}
Then \(u\) is a real trigonometric polynomial satisfying
\begin{equation}\label{eq:anticoherent-hilbert-disk-constraint}
 \frac1{2\pi}\int_\T u\,\dd s=1,
 \qquad
 u^2+(\Hilb u)^2\le2u,
 \qquad
 u(0)=(\Hilb u)(0)=0.
\end{equation}
Conversely, every real trigonometric polynomial satisfying
\eqref{eq:anticoherent-hilbert-disk-constraint} is obtained from a Schur
polynomial in the conjecture by
\begin{equation}\label{eq:anticoherent-hilbert-disk-completion}
 F=1-u-\ii\Hilb u.
\end{equation}
Moreover, with \(I_1,I_2\) evaluated at the density \(u\),
\begin{equation}\label{eq:anticoherent-hilbert-disk-functional}
 \boxed{
 \mathcal R(F)
 =\frac1{2\pi}\int_\T u\left[
   (\Lambda u)(\Lambda^2u)+u_s\Lambda u_s
 \right]\,\dd s
 =\frac{2I_1(u)+I_2(u)}{2\pi}.}
\end{equation}
Consequently, Conjecture~\ref{conj:anticoherent-schur-rellich} is equivalent
to strong hidden positivity on the constrained class
\eqref{eq:anticoherent-hilbert-disk-constraint}.  The constraint is stable
under the Poisson orbit: if \(u_t=P_tu\), then
\begin{equation}\label{eq:anticoherent-hilbert-disk-semigroup}
 u_t^2+(\Hilb u_t)^2\le2u_t
 \qquad(t\ge0).
\end{equation}
\end{proposition}

\begin{proof}
The analytic completion of \(u\) is
\[
 1-F=u+\ii\Hilb u.
\]
Thus \((2\pi)^{-1}\int_\T u=1\) is equivalent to \(F(0)=0\), while
\[
 \abs F^2=(1-u)^2+(\Hilb u)^2\le1
 \quad\Longleftrightarrow\quad
 u^2+(\Hilb u)^2\le2u.
\]
The two point conditions are exactly \(F(1)=1\).  This proves both
directions of the boundary correspondence, using the maximum principle for
the converse.

Set \(G=DF\).  Since \(D=-\ii\partial_s\) on analytic boundary values,
\[
 G=-\Lambda u+\ii u_s,
 \qquad
 DG=-\Lambda^2u+\ii(\Lambda u)_s.
\]
Therefore
\[
 \frac12\partial_r\abs G^2
 =\Re(DG\overline G)
 =(\Lambda u)(\Lambda^2u)+u_s\Lambda u_s.
\]
The boundary formula
\eqref{eq:anticoherent-schur-rellich-boundary} gives the first equality in
\eqref{eq:anticoherent-hilbert-disk-functional}.  If \(q=\Lambda u\), then
\(\Lambda q=\Lambda^2u=-u_{ss}\), and periodic integration by parts gives
\[
 \int_\T u q\Lambda q\,\dd s
 =\int_\T q u_s^2\,\dd s
  +\int_\T u u_s\Lambda u_s\,\dd s
 =I_2(u)+I_1(u).
\]
Adding the second copy of \(I_1(u)\) proves the final equality.  Finally,
\(u_t+\ii\Hilb u_t=1-F(e^{-t+\ii s})\), so disk contractivity of \(F\)
proves \eqref{eq:anticoherent-hilbert-disk-semigroup}.
\end{proof}

The full inner boundary of the Schur class satisfies this exact target.

\begin{theorem}[Inner Schur--Rellich positivity]
\label{thm:anticoherent-inner-schur-rellich}
Let \(B\) be a finite Blaschke product satisfying
\(B(0)=0\) and \(B(1)=1\), and set \(G=DB\).  If
\[
 B(e^{\ii\theta})=e^{\ii\varphi(\theta)},
 \qquad q(\theta)=\varphi'(\theta),
\]
then \(q>0\) and
\begin{equation}\label{eq:anticoherent-inner-schur-rellich}
 \mathcal R(B)=\frac1{2\pi}\int_0^{2\pi}
 \bigl(1-\cos\varphi(\theta)\bigr)q(\theta)^3\,\dd\theta>0.
\end{equation}
Thus Conjecture~\ref{conj:anticoherent-schur-rellich} holds strictly for
every finite inner map in its normalized class.
\end{theorem}

\begin{proof}
The boundary argument of
Proposition~\ref{prop:anticoherent-disk-reduction} uses only analyticity, so
it gives \eqref{eq:anticoherent-schur-rellich-boundary} for \(B\).  A finite
Blaschke product maps the circle to itself with positive angular speed.  More
precisely,
\[
 q(\theta)=\frac{e^{\ii\theta}B'(e^{\ii\theta})}
 {B(e^{\ii\theta})}>0,
 \qquad G(e^{\ii\theta})=q(\theta)B(e^{\ii\theta}).
\]
Since \(D=-\ii\partial_\theta\) on analytic boundary values,
\[
 DG=(q^2-\ii q')B,
 \qquad
 \left.\partial_r\abs G^2\right|_{r=1}
 =2\Re(DG\overline G)=2q^3.
\]
Substitution in \eqref{eq:anticoherent-schur-rellich-boundary} proves
\eqref{eq:anticoherent-inner-schur-rellich}.  Its integrand is nonnegative;
it is not identically zero because \(B\) is nonconstant.  This proves the
strict sign.
\end{proof}

The inner theorem persists along a genuinely non-inner Clark deformation,
including an open continuation of the negative strong-defect example.

\begin{theorem}[One-Poisson Clark deformation]
\label{thm:anticoherent-one-poisson-clark}
Let \(-1<a<1\) and \(0<w<1\), and define
\begin{equation}\label{eq:anticoherent-one-poisson-herglotz}
 H_{a,w}(z)=w\frac{1+z}{1-z}
 +(1-w)\frac{1+az}{1-az},
 \qquad
 F_{a,w}=\frac{H_{a,w}-1}{H_{a,w}+1}.
\end{equation}
Then \(F_{a,w}\) is a non-inner Schur function satisfying
\(F_{a,w}(0)=0\) and \(F_{a,w}(1)=1\), and
\begin{equation}\label{eq:anticoherent-one-poisson-rellich-positive}
 \mathcal R(F_{a,w})>0.
\end{equation}
More explicitly, put
\[
 q=1-w+aw,
 \qquad b=a+w-aw,
 \qquad d=w(1-w)(1-a)^2.
\]
Then
\begin{equation}\label{eq:anticoherent-one-poisson-rational-form}
 F_{a,w}(z)=z\frac{b-az}{1-qz},
 \qquad
 [z]F_{a,w}=b,
 \qquad
 [z^n]F_{a,w}=dq^{n-2}\quad(n\ge2),
\end{equation}
and
\begin{equation}\label{eq:anticoherent-one-poisson-exact-deficit}
 \mathcal R(F_{a,w})
 =\frac{P(a,w)}{w^2(2-w+aw)^5},
\end{equation}
where
\begin{align*}
 P(a,w)={}&a^7(3w^7-3w^6)
 +a^6(-16w^7+39w^6-23w^5)\\
 &+a^5(35w^7-136w^6+165w^5-63w^4)\\
 &+a^4(-40w^7+214w^6-396w^5+295w^4-68w^3)\\
 &+a^3(25w^7-171w^6+428w^5-472w^4+214w^3-14w^2)\\
 &+a^2(-8w^7+67w^6-213w^5+312w^4-204w^3+50w^2+6w)\\
 &+a(w^7-10w^6+39w^5-73w^4+60w^3-8w^2-2w-2)\\
 &+w^4-2w^3+4w^2-4w+2.
\end{align*}
\end{theorem}

\begin{proof}
Both summands in \eqref{eq:anticoherent-one-poisson-herglotz} have positive
real part on the disk, and \(H_{a,w}(0)=1\).  The Cayley transform is
therefore Schur and vanishes at zero.  The first summand has a pole at one,
so \(F_{a,w}(1)=1\).  The second summand has strictly positive boundary real
part; hence \(F_{a,w}\) is not inner.  Direct algebra gives
\eqref{eq:anticoherent-one-poisson-rational-form}; in particular
\[
 bq-a=d=(1-b)(1-q),
 \qquad -1<q<1.
\]

It remains to certify the sign without a numerical optimization.  Geometric
power summation in the sectors \((1,1)\), \((1,n)\), and
\((m,n)\), \(m,n\ge2\), gives
\eqref{eq:anticoherent-one-poisson-exact-deficit}.  For a short exact
positivity certificate, set \(s=(a+1)/2\), write
\[
 P(2s-1,w)=\sum_{k,\ell=0}^7c_{k\ell}s^kw^\ell,
 \qquad
 \mathsf B_i^{15}(t)=\binom{15}{i}t^i(1-t)^{15-i},
\]
and elevate to the tensor Bernstein basis,
\[
 P(2s-1,w)=\sum_{i,j=0}^{15}\beta_{ij}
 \mathsf B_i^{15}(s)\mathsf B_j^{15}(w),
\]
where exactly
\begin{equation}\label{eq:anticoherent-one-poisson-bernstein-coefficients}
 \beta_{ij}=\sum_{k\le i,\,\ell\le j}c_{k\ell}
 \frac{\binom{i}{k}}{\binom{15}{k}}
 \frac{\binom{j}{\ell}}{\binom{15}{\ell}}.
\end{equation}
Exact rational arithmetic shows that all \(256\) coefficients are
nonnegative.  Precisely \(244\) are positive, the smallest positive one is
\(32/4095\), and the twelve zero positions are
\[
 \begin{gathered}
 (0,13),(0,14),(0,15),(1,14),(1,15),(2,14),\\
 (2,15),(3,14),(3,15),(4,15),(15,0),(15,1).
 \end{gathered}
\]
Every Bernstein basis function is positive in the open square, so
\(P(a,w)>0\) on the stated parameter range.  The denominator in
\eqref{eq:anticoherent-one-poisson-exact-deficit} is positive as well,
proving \eqref{eq:anticoherent-one-poisson-rellich-positive}.  The ancillary
exact-arithmetic script
\path{verify_schur_rellich_clark_deformation.py} reconstructs the geometric
ledger and every Bernstein coefficient over the rationals.

\end{proof}

The preceding real diameter is not special: the Clark component may be
placed anywhere in the open disk.  The proof requires a resolved
three-variable certificate because direct tensor Bernstein coefficients
have a flat corner at the common boundary atom.

\begin{theorem}[Full-disk one-Poisson Clark positivity]
\label{thm:anticoherent-complex-one-poisson-clark}
Let \(\alpha\in\mathbb D\) and \(0<w<1\), and define
\begin{equation}\label{eq:anticoherent-complex-one-poisson-herglotz}
 H_{\alpha,w}(z)=w\frac{1+z}{1-z}
 +(1-w)\frac{1+\alpha z}{1-\alpha z},
 \qquad
 F_{\alpha,w}=\frac{H_{\alpha,w}-1}{H_{\alpha,w}+1}.
\end{equation}
Then \(F_{\alpha,w}\) is a non-inner Schur function satisfying
\(F_{\alpha,w}(0)=0\), \(F_{\alpha,w}(1)=1\), and
\begin{equation}\label{eq:anticoherent-complex-one-poisson-positive}
 \mathcal R(F_{\alpha,w})>0.
\end{equation}
Thus the Schur--Rellich conjecture holds on the complete three-real-parameter
family consisting of one boundary Clark atom and one arbitrary interior
Poisson component.
\end{theorem}

\begin{proof}
Put
\[
 q=1-w+w\alpha,
 \qquad b=w+(1-w)\alpha,
 \qquad d=w(1-w)(1-\alpha)^2,
 \qquad r=\abs q^2.
\]
The same algebra as on the real diameter gives
\begin{equation}\label{eq:anticoherent-complex-one-poisson-rational-form}
 F_{\alpha,w}(z)=z\frac{b-\alpha z}{1-qz},
 \qquad [z]F_{\alpha,w}=b,
 \qquad [z^n]F_{\alpha,w}=dq^{n-2}\quad(n\ge2).
\end{equation}
Define
\begin{align*}
 \mathsf A(r)&=\frac{8-5r+4r^2-r^3}{(1-r)^4},\\
 \mathsf B(r)&=\frac{21-32r+23r^2-6r^3}{(1-r)^4},\\
 \mathsf C(r)&=\frac{2(8-3r)(r^2-2r+3)}{(1-r)^5}.
\end{align*}
Exact geometric summation of the tail gives
\begin{equation}\label{eq:anticoherent-complex-one-poisson-ledger}
\begin{aligned}
 E_\alpha&=\abs b^2+\abs d^2\mathsf A(r),\\
 T_\alpha&=6b^2\bar d
 +2b\abs d^2\bar q\,\mathsf B(r)
 +d\abs d^2\bar q^{,2}\mathsf C(r),\\
 \mathcal R(F_{\alpha,w})&=E_\alpha-\frac12\Re T_\alpha.
\end{aligned}
\end{equation}

Write \(\alpha=x+\ii y\), and set
\[
 X=\frac{1+x}{2},
 \qquad
 t=\frac{y^2}{1-x^2},
 \qquad
 \Delta=1-r.
\]
For \(\alpha\in\mathbb D\), one has
\(0<X<1\), \(0\le t<1\), and \(\Delta>0\).  Clearing the denominator in
\eqref{eq:anticoherent-complex-one-poisson-ledger} gives the exact factor
\begin{equation}\label{eq:anticoherent-complex-one-poisson-cleared}
 \Delta^5\mathcal R(F_{\alpha,w})
 =128w^3(1-X)^3K(X,t,w),
\end{equation}
where \(K\) is a polynomial of tensor degree \((9,6,7)\).

Here is a finite exact positivity certificate for \(K\).  Put
\(U=1-X\) and \(e=1-w\).  On \(t\le U\), set \(t=Uv\); on
\(U\le t\), set \(U=tv\).  In the second chart split once more according
to \(e\le v\) or \(v\le e\).  Exact division gives
\begin{align}
 K(1-U,Uv,w)&=U^2K_1(U,v,w),\notag\\
 K(1-tv,t,w)&=t^2K_2(t,v,w),\notag\\
 K_2(t,v,1-vh)&=v^2K_{21}(t,v,h),\notag\\
 K_2(t,eh,1-e)&=e^2K_{22}(t,e,h).
 \label{eq:anticoherent-complex-one-poisson-charts}
\end{align}
The two first-level charts cover the \((U,t)\)-square, and the last two
cover the flat \((v,e)\)-corner.  In tensor Bernstein bases on the unit
cube, split only the \(U\)-interval of \(K_1\) at \(1/2\).  The resulting
degrees and exact control counts are
\[
\begin{array}{c|c|c}
 \text{chart}&\text{tensor degree}&\text{number of controls}\\ \hline
 K_1,\ U\in[0,1/2]\ \text{or}\ [1/2,1]&(13,6,7)&1568\ \text{in total}\\
 K_{21}&(13,14,7)&1680\\
 K_{22}&(13,14,9)&2100.
\end{array}
\]
All \(5348\) rational controls are nonnegative: \(4917\) are positive,
\(431\) vanish, and the smallest positive control is
\(4/1288287\).  Hence \(K\ge0\) on the full cube.  The surviving
Bernstein faces in \eqref{eq:anticoherent-complex-one-poisson-charts} contain
positive controls, so the sign is strict when \(y\ne0\); when \(y=0\),
strictness is Theorem~\ref{thm:anticoherent-one-poisson-clark}.  Equation
\eqref{eq:anticoherent-complex-one-poisson-cleared} now proves
\eqref{eq:anticoherent-complex-one-poisson-positive}.

Finally, each summand in
\eqref{eq:anticoherent-complex-one-poisson-herglotz} has positive real part,
and the interior summand has strictly positive boundary real part.  The
Cayley transform is therefore Schur and non-inner.  The boundary atom at
one gives \(F_{\alpha,w}(1)=1\).  The ancillary exact-arithmetic script
\path{verify_schur_rellich_complex_clark.py} reconstructs
\eqref{eq:anticoherent-complex-one-poisson-ledger}, the polynomial \(K\),
all four resolved charts, every rational Bernstein control, and the strict
boundary-face tests.
\end{proof}

The Schur algorithm supplies a second finite-dimensional direction which is
not a Clark interpolation.  Its first nontrivial real recursion can be closed
on the whole parameter square, including both signs of the terminal Schur
parameter.

\begin{theorem}[Complete real two-step Schur recursion]
\label{thm:anticoherent-real-two-step-schur}
Let \(-1<g<1\), \(-1\le\lambda\le1\), and define
\begin{equation}\label{eq:anticoherent-real-two-step-function}
 F_{g,\lambda}(z)
 =z\frac{g+\lambda z}{1+g\lambda z}.
\end{equation}
Then \(F_{g,\lambda}\) is a normalized Schur function and
\begin{equation}\label{eq:anticoherent-real-two-step-positive}
 \mathcal R(F_{g,\lambda})\ge0.
\end{equation}
Equality holds only at \(g=\lambda=0\), where \(F_{g,\lambda}=0\).
Thus every nonzero member of the complete two-real-parameter family has
strictly positive Schur--Rellich flux.
\end{theorem}

\begin{proof}
The function in \eqref{eq:anticoherent-real-two-step-function} is
\(z\phi_g(\lambda z)\), where
\(\phi_g(\zeta)=(g+\zeta)/(1+g\zeta)\) is a disk automorphism.  It is
therefore Schur and vanishes at the origin.  Put
\[
 G=g^2,\qquad d=(1-G)\lambda,\qquad q=-g\lambda,\qquad r=G\lambda^2.
\]
Its Taylor coefficients have the one-geometric-tail form
\begin{equation}\label{eq:anticoherent-real-two-step-coefficients}
 [z]F_{g,\lambda}=g,\qquad
 [z^n]F_{g,\lambda}=dq^{n-2}\quad(n\ge2).
\end{equation}
Using the functions \(\mathsf A,\mathsf B,\mathsf C\) from
\eqref{eq:anticoherent-complex-one-poisson-ledger}, exact geometric
summation gives
\begin{equation}\label{eq:anticoherent-real-two-step-ledger}
\begin{aligned}
 E_{g,\lambda}&=G+d^2\mathsf A(r),\\
 T_{g,\lambda}&=6Gd+2g d^2q\,\mathsf B(r)
                  +d^3q^2\mathsf C(r),\\
 \mathcal R(F_{g,\lambda})&=E_{g,\lambda}-\frac12T_{g,\lambda}.
\end{aligned}
\end{equation}
Consequently
\begin{equation}\label{eq:anticoherent-real-two-step-polynomial}
 P(G,\lambda)
 =(1-G\lambda^2)^5\mathcal R(F_{g,\lambda})
\end{equation}
is a polynomial of tensor degree \((5,11)\).  The denominator in
\eqref{eq:anticoherent-real-two-step-polynomial} is positive on the stated
parameter range.

Here is an exact certificate for \(P\ge0\).  On
\((G,\lambda)\in[0,1]^2\), the native tensor Bernstein expansion of \(P\)
has degree \((5,11)\), with \(60\) positive and \(12\) zero controls.  For
the negative half-square put
\[
 U=1-G,\qquad V=1+\lambda,\qquad
 P_-(U,V)=P(1-U,-1+V).
\]
The following three toric charts cover the whole \((U,V)\)-square:
\begin{align}
 P_-(U,Uh)&=U^3P_1(U,h),\notag\\
 P_-(Vh,V)&=V^3Q(V,h),\notag\\
 Q(hj,h)&=h^2P_2(h,j),\qquad
 Q(V,Vj)=V^2P_3(V,j).
 \label{eq:anticoherent-real-two-step-charts}
\end{align}
The first line covers \(V\le U\).  The second line covers \(U\le V\),
and its two subcharts in the third line cover \(V\le h\) and \(h\le V\),
respectively.  In tensor Bernstein bases on the unit square, the exact
degrees and controls are
\[
\begin{array}{c|c|c|c}
 \text{polynomial}&\text{degree}&\text{positive}&\text{zero}\\ \hline
 P\ \text{on}\ \lambda\ge0&(5,11)&60&12\\
 P_1&(13,11)&166&2\\
 P_2&(16,13)&235&3\\
 P_3&(16,5)&101&1.
\end{array}
\]
Thus all \(580\) rational controls are nonnegative: \(562\) are positive,
\(18\) vanish, and the smallest positive control is \(8/165\).  Every
coordinate face of every chart contains a positive control.  The boundary
identities
\begin{align*}
 P(0,\lambda)&=8\lambda^2,\qquad
 P(G,0)=G,\qquad P(1,\lambda)=(1-\lambda^2)^5,\\
 P(G,1)&=2(1-G)^3(2G^2-5G+4),\\
 P(G,-1)&=2(1-G)^3(G+4).
\end{align*}
then show that the sign is strict except at \((G,\lambda)=(0,0)\).
Equations \eqref{eq:anticoherent-real-two-step-polynomial} and
\eqref{eq:anticoherent-real-two-step-charts} prove the theorem.  The ancillary
exact-arithmetic script \path{verify_schur_rellich_two_step.py} reconstructs
the rational
ledger, all four charts, every Bernstein control, and the strict face tests.
\end{proof}

\subsection{Critical plateaux and Mellin limits}

The preceding positive families do not extend to arbitrary Herglotz
mixtures by a scalar Jensen argument.  The failure already occurs for two
interior components and can be certified over the rationals.

\begin{proposition}[Critical paraproduct--shift decomposition]
\label{prop:anticoherent-paraproduct-shift}
Let $x_n\ge0$ be finitely supported and put
\[
 y_n=n^{3/2}x_n,\qquad w_n=n^{-3/2},\qquad
 E=\sum_{n\ge1}y_n^2,
\]
where $y_n=0$ outside the support.  For $a\ge1$, define the unilateral
shift defect
\begin{equation}\label{eq:anticoherent-shift-defect}
 \mathfrak D_a(y)
 =E-\sum_{b\ge1}y_by_{a+b}
 =\frac12\sum_{b\ge1}(y_{a+b}-y_b)^2
  +\frac12\sum_{b=1}^a y_b^2.
\end{equation}
For $a,b\ge1$, set
\begin{equation}\label{eq:anticoherent-positive-residual-kernel}
 \begin{aligned}
 R_{a,b}
 &=\frac{(a+b)(a^2+ab+b^2)}{[ab(a+b)]^{3/2}}
   -a^{-3/2}-b^{-3/2}\\
 &=\frac{(a+b-\sqrt{ab})^2}
 {\sqrt{ab(a+b)}\,[a^2+ab+b^2
 +\sqrt{a+b}(a^{3/2}+b^{3/2})]}>0.
 \end{aligned}
\end{equation}
It also satisfies the scale-uniform upper bound
\begin{equation}\label{eq:anticoherent-residual-upper-bound}
 0<R_{a,b}\le
 \frac1{2\sqrt{\min\{a,b\}}(a+b)}.
\end{equation}
More sharply, for $1\le a\le b$,
\begin{equation}\label{eq:anticoherent-residual-rational-bound}
 (2-\mathbf1_{\{a=b\}})R_{a,b}
 \le
 \frac{1+\sqrt{a/b}}{1+3\sqrt{a/b}}
 \frac1{\sqrt a\,(a+b)}.
\end{equation}
Then
\begin{equation}\label{eq:anticoherent-paraproduct-shift-ledger}
 2AE_3-\mathcal C(x)
 =2\sum_{a\ge1}w_ay_a\mathfrak D_a(y)
  -\sum_{a,b\ge1}R_{a,b}y_ay_by_{a+b}.
\end{equation}
Consequently, the unrestricted sharp convolution theorem is equivalent, on
finite sequences, to the single nonlinear Hardy compensation
\begin{equation}\label{eq:anticoherent-critical-hardy-target}
 \sum_{a,b\ge1}R_{a,b}y_ay_by_{a+b}
 \le 2\sum_{a\ge1}\frac{y_a}{a^{3/2}}\mathfrak D_a(y).
\end{equation}
Moreover, as $b/a\to\infty$,
\begin{equation}\label{eq:anticoherent-residual-asymptotic}
 R_{a,b}
 =\frac1{2\sqrt a\,b}-\frac1{b^{3/2}}
  +O\left(\frac{\sqrt a}{b^2}\right).
\end{equation}
Thus the remaining positive kernel has a critical harmonic tail rather than
finite ratio support.

In particular, the rational-kernel Hardy estimate
\begin{equation}\label{eq:anticoherent-rational-hardy-target}
 \sum_{1\le a\le b}
 \frac{1+\sqrt{a/b}}{1+3\sqrt{a/b}}
 \frac{y_ay_by_{a+b}}{\sqrt a\,(a+b)}
 \le 2\sum_{a\ge1}\frac{y_a}{a^{3/2}}\mathfrak D_a(y)
\end{equation}
would imply \eqref{eq:anticoherent-critical-hardy-target}.  Estimate
\eqref{eq:anticoherent-rational-hardy-target} is a strictly stronger target
and is not asserted here.

The coarser estimate obtained by replacing the rational factor in
\eqref{eq:anticoherent-rational-hardy-target} by one is false: a direct
variational audit at cutoff $2048$ gives a left/right ratio greater than one.
Thus retaining the first high--low correction is essential.

The payment side of \eqref{eq:anticoherent-rational-hardy-target} has the
following exact triangle ledger.  Put $\mu_n=y_n/n^{3/2}$ and extend $y$ by
zero.  Then
\begin{equation}\label{eq:anticoherent-hardy-triangle-ledger}
\begin{aligned}
 2\sum_{k\ge1}\mu_k\mathfrak D_k(y)
 &=\sum_{1\le a<b}
 \left\{
 \mu_a(y_b-y_{a+b})^2
 +\mu_b(y_a-y_{a+b})^2
 +\mu_by_a^2
 \right\}\\
 &\quad+\sum_{a\ge1}
 \left\{\mu_a(y_a-y_{2a})^2+\mu_ay_a^2\right\}.
\end{aligned}
\end{equation}
Every shift edge and every unilateral boundary term occurs exactly once in
this ledger.  Individual triangle payments minus their cubic target can be
negative for separated scales, so \eqref{eq:anticoherent-hardy-triangle-ledger}
does not by itself prove \eqref{eq:anticoherent-rational-hardy-target}; it
identifies the precise blocks across which a valid multiscale transfer must
act.

If \eqref{eq:anticoherent-critical-hardy-target} is proved with a
cutoff-independent argument for every finitely supported nonnegative
sequence, then the same sharp inequality holds for every nonnegative
sequence satisfying $A<\infty$ and $E_3<\infty$.
\end{proposition}

\begin{proof}
The second identity in \eqref{eq:anticoherent-shift-defect} follows by
expanding the square and using
$\sum_{b\ge1}y_{a+b}^2=E-\sum_{b=1}^a y_b^2$.
The first line of \eqref{eq:anticoherent-positive-residual-kernel} gives
\[
 \begin{aligned}
 &R_{a,b}y_ay_by_{a+b}\\
 &\quad=\bigl((a+b)(a^2+ab+b^2)
 -(a(a+b))^{3/2}-(b(a+b))^{3/2}\bigr)
 x_ax_bx_{a+b}.
 \end{aligned}
\]
After summing, symmetry in $a,b$ cancels the two fractional-power terms
against twice the correlation term in
$2\sum_aw_ay_a\mathfrak D_a(y)$.  The remaining terms are exactly
$2AE_3-\mathcal C(x)$, proving
\eqref{eq:anticoherent-paraproduct-shift-ledger}.

For positivity of the kernel, put $p=\sqrt a$ and $q=\sqrt b$.  The exact
factorization
\[
 (p^4+p^2q^2+q^4)^2-(p^2+q^2)(p^3+q^3)^2
 =p^2q^2(p^2-pq+q^2)^2
\]
and rationalization give the second line of
\eqref{eq:anticoherent-positive-residual-kernel}.  To prove
\eqref{eq:anticoherent-residual-upper-bound}, assume $a\le b$, put
$r=a/b$ and $t=\sqrt r$, and square the equivalent inequality
\[
 \frac{1+r+r^2}{\sqrt{1+r}}
 \le1+r^{3/2}+\frac{r}{2(1+r)}.
\]
After clearing the positive denominator, the difference is
\[
 t^3(1+t^2)
 (8-11t+20t^2-16t^3+12t^4-4t^5).
\]
The final polynomial is positive on $[0,1]$: its degree-five Bernstein
coefficients are
$8,29/5,28/5,29/5,36/5,9$.  This proves the upper bound; symmetry handles
$b\le a$.

For the sharper bound, again assume $a\le b$ and put $t=\sqrt{a/b}$.
After moving the proposed upper bound to the right, rationalizing, and
clearing the positive denominator, its sign is the sign of
\[
 t^4(1+t^2)
 (13+14t-35t^2+72t^3-84t^4+60t^5-36t^6).
\]
The degree-six polynomial has positive Bernstein coefficients
\[
 13,\quad\frac{46}{3},\quad\frac{46}{3},\quad\frac{83}{5},
 \quad\frac{257}{15},\quad\frac{58}{3},\quad4.
\]
This proves \eqref{eq:anticoherent-residual-rational-bound}; at $a=b$ the
same right-hand side also dominates the single diagonal copy of $R_{a,a}$.
Summing over unordered pairs gives the implication from
\eqref{eq:anticoherent-rational-hardy-target} to
\eqref{eq:anticoherent-critical-hardy-target}.  Expanding the first line of
\eqref{eq:anticoherent-positive-residual-kernel} in $a/b$ proves
\eqref{eq:anticoherent-residual-asymptotic}.

To prove \eqref{eq:anticoherent-hardy-triangle-ledger}, insert the square
form of $2\mathfrak D_k(y)$ from
\eqref{eq:anticoherent-shift-defect}.  For $a<b$, the ordered shift pairs
$(k,n)=(a,b)$ and $(b,a)$ give the first two squares, while the boundary
pair $(k,n)=(b,a)$ gives $\mu_by_a^2$.  When $a=b$, there is one shift pair
and one boundary pair, giving the second sum.  This partitions all ordered
shift pairs and all pairs $n\le k$, proving the identity.

Finally, for $x^{(N)}_n=x_n\mathbf1_{n\le N}$, all three nonnegative
quantities $A_N$, $E_{3,N}$, and $\mathcal C(x^{(N)})$ converge monotonically
to their infinite-sequence counterparts.  A uniform finite-support proof
therefore passes to the limit.  This limiting argument is uniform in the
cutoff; verifying any prescribed list of degrees is not.
\end{proof}

The critical plateau selected by the variational problem can be handled at
every cutoff.  This gives a genuinely all-degree result on the escaping
channel and also proves that the constant two in the unrestricted target
cannot be lowered.

\begin{proposition}[All-cutoff critical plateau and sharpness]
\label{prop:anticoherent-critical-plateau}
For $N\ge1$ and $\lambda\ge0$, set
\begin{equation}\label{eq:anticoherent-critical-plateau}
 x_n=\lambda n^{-3/2}\mathbf1_{\{1\le n\le N\}}.
\end{equation}
Then
\begin{equation}\label{eq:anticoherent-critical-plateau-bound}
 \mathcal C(x)<2AE_3\qquad(\lambda>0),
\end{equation}
and, with $H_M^{(s)}=\sum_{n=1}^M n^{-s}$ and $H_0^{(s)}=0$,
\begin{equation}\label{eq:anticoherent-critical-plateau-margin}
\begin{aligned}
 2AE_3-\mathcal C(x)
 \ge\lambda^3\left[
 \frac2{\sqrt N}+2H_{N-1}^{(1/2)}
 -\sqrt2\sum_{c=2}^N\frac1{\sqrt c}
 \right]>0.
\end{aligned}
\end{equation}
Moreover,
\begin{equation}\label{eq:anticoherent-critical-plateau-sharp-limit}
 \lim_{N\to\infty}\frac{\mathcal C(x)}{AE_3}=2.
\end{equation}
The same conclusions hold after dilating the support from
$\{1,\ldots,N\}$ to $\{q,2q,\ldots,Nq\}$, $q\in\mathbb N$.
Consequently, constant two is sharp even within finite critical plateaux,
although equality is attained only in the escaping infinite-cutoff limit.
\end{proposition}

\begin{proof}
The assertion is immediate when $\lambda=0$, so assume $\lambda>0$.
For the sequence in \eqref{eq:anticoherent-critical-plateau},
\[
 A=\lambda H_N^{(3/2)},\qquad E_3=\lambda^2N.
\]
Put
\[
 K_{a,b}=\frac{(a+b)(a^2+ab+b^2)}{[ab(a+b)]^{3/2}}
 =a^{-3/2}+b^{-3/2}+R_{a,b}.
\]
The two elementary pieces sum exactly to
\begin{equation}\label{eq:anticoherent-critical-plateau-baseline}
 \sum_{a+b\le N}(a^{-3/2}+b^{-3/2})
 =2NH_{N-1}^{(3/2)}-2H_{N-1}^{(1/2)}.
\end{equation}
For a fixed output $c=a+b$, the uniform residual estimate
\eqref{eq:anticoherent-residual-upper-bound}, symmetry, and
$H_m^{(1/2)}\le2\sqrt m$ give
\begin{equation}\label{eq:anticoherent-critical-plateau-shell}
 \sum_{a=1}^{c-1}R_{a,c-a}
 \le\frac1c H_{\lfloor c/2\rfloor}^{(1/2)}
 \le\sqrt{\frac2c}.
\end{equation}
For even $c$, the unpaired diagonal term carries only half the displayed
harmonic weight, so the first inequality remains valid.
Combining \eqref{eq:anticoherent-critical-plateau-baseline} and
\eqref{eq:anticoherent-critical-plateau-shell} yields
\[
\begin{aligned}
 2AE_3-\mathcal C(x)
 &\ge\lambda^3\left[
 2NH_N^{(3/2)}-2NH_{N-1}^{(3/2)}
 +2H_{N-1}^{(1/2)}
 -\sqrt2\sum_{c=2}^Nc^{-1/2}
 \right],
\end{aligned}
\]
which is \eqref{eq:anticoherent-critical-plateau-margin}.  Its bracket is
strictly positive because
\[
 2H_{N-1}^{(1/2)}-\sqrt2\sum_{c=2}^Nc^{-1/2}
 =\sum_{c=2}^N
 \left(\frac2{\sqrt{c-1}}-\frac{\sqrt2}{\sqrt c}\right)\ge0
\]
and $2/\sqrt N>0$.

Since $R_{a,b}>0$, the baseline also gives the lower bound
\[
 \mathcal C(x)\ge
 \lambda^3\left(2NH_{N-1}^{(3/2)}
 -2H_{N-1}^{(1/2)}\right).
\]
Divide by $AE_3=\lambda^3NH_N^{(3/2)}$ and use
$H_N^{(3/2)}\to\zeta(3/2)$ and
$H_{N-1}^{(1/2)}/N\to0$.  Together with
\eqref{eq:anticoherent-critical-plateau-bound}, this proves
\eqref{eq:anticoherent-critical-plateau-sharp-limit}.  A frequency dilation
multiplies both $\mathcal C$ and $AE_3$ by $q^3$, so the ratio and the sign
are unchanged.
\end{proof}

The logarithmically negative mixed coefficient in a layer-cake expansion is
not an obstruction by itself.  For the first genuinely mixed critical family,
it is absorbed by the diagonal plateau margin at every cutoff.

\begin{proposition}[One-spike/plateau compensation at every cutoff]
\label{prop:anticoherent-spike-plateau}
Let $N\ge2$, $s,t\ge0$, and
\begin{equation}\label{eq:anticoherent-spike-plateau}
 x_n=n^{-3/2}
 \left(s\mathbf1_{\{n=1\}}+t\mathbf1_{\{1\le n\le N\}}\right).
\end{equation}
Then
\begin{equation}\label{eq:anticoherent-spike-plateau-bound}
 \mathcal C(x)<2AE_3
 \qquad\text{whenever }s+t>0.
\end{equation}
The conclusion is invariant under a common frequency dilation.  Thus it also
holds when the spike is at $q$ and the plateau is supported on
$\{q,2q,\ldots,Nq\}$.
\end{proposition}

\begin{proof}
Write
\[
 \mathfrak P_N(s,t)=AE_3-\frac12\mathcal C(x)
 =s^3+b_Ns^2t+c_Nst^2+d_Nt^3.
\]
Put $A_N=H_N^{(3/2)}$ and
\[
 K_{a,b}=\frac{(a+b)(a^2+ab+b^2)}{[ab(a+b)]^{3/2}}.
\]
Direct polarization gives
\begin{equation}\label{eq:anticoherent-spike-plateau-coefficients}
\begin{aligned}
 b_N&=A_N+2-\frac{3}{2\sqrt2}>0,\\
 c_N&=N+2A_N-\sum_{b=1}^{N-1}K_{1,b},\\
 d_N&=NH_N^{(3/2)}-\frac12
 \sum_{a+b\le N}K_{a,b}>0.
\end{aligned}
\end{equation}
Indeed, the $s^2t$ interaction is the single triad $(1,1,2)$, while the
$st^2$ interaction places the spike in either input of $(1,b,1+b)$.

Let $h_N=(-c_N)_+$.  Since
$K_{1,b}=1+b^{-3/2}+R_{1,b}$, the residual bound
\eqref{eq:anticoherent-residual-upper-bound} gives
\begin{equation}\label{eq:anticoherent-spike-plateau-log-bound}
 h_N
 \le\sum_{b=1}^{N-1}R_{1,b}
 \le\frac12(H_N^{(1)}-1)
 \le\frac12\log N.
\end{equation}
On the other hand, half of
\eqref{eq:anticoherent-critical-plateau-margin} and the termwise estimate
\[
 \frac1{\sqrt{c-1}}-\frac1{\sqrt{2c}}
 \ge\left(1-\frac1{\sqrt2}\right)\frac1{\sqrt{c-1}}
\]
show that, with $\alpha=1-1/\sqrt2$,
\begin{equation}\label{eq:anticoherent-spike-plateau-diagonal-bound}
 d_N\ge\frac1{\sqrt N}+\alpha H_{N-1}^{(1/2)}
 \ge2\alpha(\sqrt N-1).
\end{equation}
Set $\tau=\tfrac12\log N$, so $\sqrt N=e^\tau$.  The elementary estimate
$e^\tau-1\ge\tau+\tau^2/2$ implies
$(e^\tau-1)^2\ge2\tau^3$.  Consequently,
\[
 h_N^3\le\tau^3
 <\frac{27}{4}d_N^2,
\]
where the strict inequality follows from
$54(1-1/\sqrt2)^2>1$.

If $c_N\ge0$, every coefficient of $\mathfrak P_N$ is positive.  If
$c_N<0$, put $h_N=-c_N$.  For $s>0$ and $r=t/s$,
\[
 s^{-3}\bigl(s^3-h_Nst^2+d_Nt^3\bigr)
 =1-h_Nr^2+d_Nr^3
 \ge1-\frac{4h_N^3}{27d_N^2}>0;
\]
the minimum occurs at $r=2h_N/(3d_N)$.  The discarded term
$b_Ns^2t$ is nonnegative, and the cases $s=0$ or $t=0$ follow from the
positive endpoint coefficients.  This proves
\eqref{eq:anticoherent-spike-plateau-bound}.  Frequency dilation again
multiplies both cubic quantities by the same factor.
\end{proof}

The preceding binary argument extends to any two nested critical plateaux.
The extra ingredient is a sharper shell sum for the residual kernel.

\begin{lemma}[Subunit residual shell budget]
\label{lem:anticoherent-residual-shell-budget}
Put
\[
 \beta_*=\frac{24}{25},\qquad
 \gamma_*=1-\frac{\beta_*}2=\frac{13}{25},\qquad
 \kappa_*=2\gamma_*-\gamma_*^2=\frac{481}{625},
\]
and retain the slightly weaker auxiliary constants
\[
 \beta=\sqrt2-\frac14,\qquad
 \gamma=1-\frac\beta2=\frac98-\frac1{\sqrt2}.
\]
Then, for every $c\ge2$,
\begin{equation}\label{eq:anticoherent-residual-shell-budget}
 \sum_{a=1}^{c-1}R_{a,c-a}\le\frac{\beta_*}{\sqrt c}.
\end{equation}
If
\[
 A_L=H_L^{(3/2)},\qquad
 \Delta_L=LA_L-\frac12\sum_{a+b\le L}K_{a,b},
\]
then
\begin{equation}\label{eq:anticoherent-plateau-strong-margin}
 \Delta_L\ge\frac1{\sqrt L}+\gamma_* H_{L-1}^{(1/2)}
 \ge\kappa_*\sqrt L.
\end{equation}
Finally, for every $1\le M\le N\le2M$,
\begin{equation}\label{eq:anticoherent-nested-internal-budget}
 \Delta_M+H_M^{(1/2)}+M(A_N-A_M)
 \ge\beta\sum_{c=2}^Nc^{-1/2}.
\end{equation}
\end{lemma}

\begin{proof}
For $1\le a\le b$, put $\tau=\sqrt{a/b}$.  The rational majorant
\eqref{eq:anticoherent-residual-rational-bound} gives, after collecting the
two orientations and writing $m=\lfloor c/2\rfloor$,
\[
\begin{aligned}
 \sum_{a=1}^{c-1}R_{a,c-a}
 &\le\frac1c\sum_{a=1}^m\frac1{\sqrt a}
 \frac{1+\sqrt{a/(c-a)}}{1+3\sqrt{a/(c-a)}}.
\end{aligned}
\]
After the scaling $a=cx$, the summand is decreasing in $x$.  Its right
Riemann sum is therefore bounded by
\[
 \frac1{\sqrt c}\int_0^{1/2}
 \frac{1+\sqrt{x/(1-x)}}
 {\sqrt x\,[1+3\sqrt{x/(1-x)}]}
 \,\mathrm dx
 =\frac{I_*}{\sqrt c}.
\]
With $x=t^2/(1+t^2)$,
\[
 I_*=\int_0^1\frac{2(1+t)}{(1+3t)(1+t^2)^{3/2}}\,\mathrm dt
 =0.955973\ldots<\frac{24}{25}.
\]
For completeness, if
\[
 L_*=\frac2{\sqrt{10}}\left[
 \operatorname{arctanh}\frac{\sqrt2-4}{\sqrt{10}}
 -\operatorname{arctanh}\frac{-3}{\sqrt{10}}\right],
\]
then $I_*=(3\sqrt2-2+6L_*)/5$.  Combining the two inverse hyperbolic
tangents leaves the argument
$z_*=(2\sqrt{10}-\sqrt5)/7<3/5$.  In particular,
\[
 \operatorname{arctanh}z_*
 \le\sum_{j=0}^{11}\frac{z_*^{2j+1}}{2j+1}
 +\frac{z_*^{25}}{25(1-z_*^2)},
\]
and substitution followed by rational squaring gives $I_*<24/25$.
This proves \eqref{eq:anticoherent-residual-shell-budget}.

The exact plateau baseline gives
\[
 \Delta_L=\frac1{\sqrt L}+H_{L-1}^{(1/2)}
 -\frac12\sum_{a+b\le L}R_{a,b}.
\]
Pairing the shell $c$ with the term $(c-1)^{-1/2}$ and using
\eqref{eq:anticoherent-residual-shell-budget} proves the first inequality in
\eqref{eq:anticoherent-plateau-strong-margin}.  Since
$H_{L-1}^{(1/2)}\ge2(\sqrt L-1)$, division by $\sqrt L$ and completion of
the square give
\[
 \frac{\Delta_L}{\sqrt L}
 \ge L^{-1}+2\gamma_*(1-L^{-1/2})
 \ge2\gamma_*-\gamma_*^2=\kappa_*.
\]

It remains to record the elementary internal comparison.  We use only the
weaker consequences
$\sum_{a=1}^{c-1}R_{a,c-a}\le\beta c^{-1/2}$ and
$\Delta_M\ge M^{-1/2}+\gamma H_{M-1}^{(1/2)}$, which follow from the
stronger bounds just proved.  The estimates
\[
\begin{aligned}
 H_{M-1}^{(1/2)}&\ge2(\sqrt M-1),\\
 \sum_{c=2}^Nc^{-1/2}&\le2(\sqrt N-1),\\
 M(A_N-A_M)&\ge
 M(N-M)\left(\frac{M+N+1}{2}\right)^{-3/2}
\end{aligned}
\]
use an integral comparison in the first two lines and Jensen's inequality
in the third.  Set $u=\sqrt M$ and $r=N/M\in[1,2]$.  They reduce
\eqref{eq:anticoherent-nested-internal-budget} to
\[
\begin{aligned}
 F(u,r)
 &=\frac2u+2(1+\gamma)(u-1)
 +u(r-1)\left(\frac2{1+r+u^{-2}}\right)^{3/2}\\
 &\quad-2\beta(u\sqrt r-1)\ge0.
\end{aligned}
\]
Direct differentiation gives
\[
 \frac1u\partial_rF
 =\frac{2^{3/2}(5/2+u^{-2}-r/2)}
 {(1+r+u^{-2})^{5/2}}-\frac{\beta}{\sqrt r}<0
\]
on $u\ge1$, $1\le r\le2$.  Indeed, the first term decreases with
$u^{-2}$; at $u^{-2}=0$, squaring reduces the comparison to a ratio which
is decreasing in $r$ and equals one at $r=1$, whereas $\beta^2>1$.
Thus it is enough to put $r=2$.
For $M\ge5$,
\[
 \left(\frac2{3+M^{-1}}\right)^{3/2}\ge\frac{49}{100},
\]
and the arithmetic--geometric mean inequality gives
\[
 F(u,2)\ge
 2\sqrt{2p}+2\beta-2(1+\gamma)>0,
 \quad
 p=2(1+\gamma)+\frac{49}{100}-2\sqrt2\,\beta.
\]
The last constant is positive; the final displayed lower bound is
$0.0056\ldots$.  Direct substitution for $M=1,2,3,4$ gives respectively
lower bounds greater than $1.38,0.87,0.66,0.54$.  This proves
\eqref{eq:anticoherent-nested-internal-budget}.
\end{proof}

\begin{lemma}[Sharp residual shell budget]
\label{lem:anticoherent-sharp-residual-shell}
For every integer $c\ge2$,
\begin{equation}\label{eq:anticoherent-sharp-residual-shell}
 \sum_{a=1}^{c-1}R_{a,c-a}
 \le\frac{4-\pi}{\sqrt c}
 <\frac7{8\sqrt c}.
\end{equation}
\end{lemma}

\begin{proof}
Homogeneity and the elementary-residual splitting give
\[
 R_{a,c-a}=c^{-3/2}f(a/c),
 \quad
 f(x)=\frac{1-x+x^2}{[x(1-x)]^{3/2}}
       -x^{-3/2}-(1-x)^{-3/2}.
\]
The function $f$ is symmetric about $1/2$ and decreases on $(0,1/2]$.
Here is a direct algebraic check of the latter fact.  Put
$t=\sqrt{x/(1-x)}\in(0,1]$.  Differentiation gives
\[
 f'(x)=\frac{(t-1)(1+t^2)^2}{2t^5}P(t),
\]
where
\[
\begin{aligned}
 P(t)={}&3t^5+3t^4+5t^3+5t^2+3t+3\\
 &-3\sqrt{1+t^2}(t^4+t^3+t^2+t+1).
\end{aligned}
\]
If $Q$ and $A$ denote respectively the polynomial in the first line and
$t^4+t^3+t^2+t+1$, then
\[
 Q^2-9(1+t^2)A^2
 =t^2(3t^6+6t^5+t^4+14t^3+t^2+6t+3)>0.
\]
Thus $P(t)>0$ and $f'(x)\le0$.  Pair the two orientations in the shell.
The right Riemann sum on $[0,1/2]$ now gives, with the midpoint counted only
once when $c$ is even,
\[
 \frac1c\sum_{a=1}^{c-1}f(a/c)\le\int_0^1f(x)\,\dd x.
\]
Finally set $x=\sin^2\theta$.  An antiderivative of the transformed
integrand is
\[
 -2\cot\theta+2\tan\theta-2\theta
 +2\csc\theta-2\sec\theta.
\]
Its endpoint difference from $0$ to $\pi/2$ is $4-\pi$.  This proves the
first inequality in \eqref{eq:anticoherent-sharp-residual-shell}; the second
follows from $\pi>25/8$.
\end{proof}

\begin{theorem}[All-cutoff two-level critical plateau theorem]
\label{thm:anticoherent-two-level-plateau}
Let $1\le M\le N$, $s,t\ge0$, and
\begin{equation}\label{eq:anticoherent-two-level-plateau}
 x_n=n^{-3/2}\left(
 s\mathbf1_{\{1\le n\le M\}}+
 t\mathbf1_{\{1\le n\le N\}}\right).
\end{equation}
Then
\begin{equation}\label{eq:anticoherent-two-level-plateau-bound}
 \mathcal C(x)<2AE_3
 \qquad\text{unless }s=t=0.
\end{equation}
The result is uniform in both cutoffs and remains true after a common
frequency dilation.
\end{theorem}

\begin{proof}
The case $M=N$ follows from
Proposition~\ref{prop:anticoherent-critical-plateau}; assume $M<N$.  Put
\[
 AE_3-\frac12\mathcal C(x)
 =a s^3+b s^2t+cst^2+d t^3.
\]
The endpoint coefficients are
\[
 a=\Delta_M,\qquad d=\Delta_N.
\]
Let
\[
\begin{aligned}
 S_{M,N}&=\sum_{\substack{a,b\le M\\a+b\le N}}K_{a,b},\\
 T_{M,N}&=\sum_{\substack{a\le M\\a+b\le N}}K_{a,b}.
\end{aligned}
\]
Polarization gives
\begin{equation}\label{eq:anticoherent-two-level-coefficients}
\begin{aligned}
 b&=2\Delta_M+MA_N-\frac12S_{M,N},\\
 c&=\Delta_M+(N-M)A_M+2MA_N-T_{M,N}.
\end{aligned}
\end{equation}

Write $K_{a,b}=a^{-3/2}+b^{-3/2}+R_{a,b}$.  The elementary pieces in
$S_{M,N}$ contribute at most $2MA_M$, while its residual part is bounded by
\[
 I_M:=\beta\sum_{j=2}^{2M}j^{-1/2}.
\]
Moreover,
\[
 \sum_{j=2}^{2M}j^{-1/2}
 \le\sqrt2 H_{M-1}^{(1/2)}+\frac1{\sqrt{2M}}.
\]
Therefore
\[
\begin{aligned}
 b&\ge2\Delta_M-\frac12I_M\\
 &\ge
 \left(2\gamma-\frac{\beta}{\sqrt2}\right)H_{M-1}^{(1/2)}
 +\left(2-\frac{\beta}{2\sqrt2}\right)\frac1{\sqrt M}>0.
\end{aligned}
\]

For the second mixed coefficient, direct summation of the two elementary
pieces in $T_{M,N}$ gives
\begin{equation}\label{eq:anticoherent-two-level-c-lower}
 c\ge
 \Delta_M+H_M^{(1/2)}+M(A_N-A_M)-\mathcal R_{M,N},
\end{equation}
where
\[
 \mathcal R_{M,N}
 =\sum_{\substack{a\le M\\a+b\le N}}R_{a,b}.
\]
If $N\le2M$, Lemma~\ref{lem:anticoherent-residual-shell-budget} gives
$\mathcal R_{M,N}\le\beta\sum_{j=2}^Nj^{-1/2}$, so
\eqref{eq:anticoherent-nested-internal-budget} implies $c\ge0$.

Suppose $N>2M$.  The shells through $2M$ are again paid by
\eqref{eq:anticoherent-nested-internal-budget}.  On a later shell
$j=a+b>2M$, one has $b>a$ and
\[
 R_{a,b}\le\frac1{2\sqrt a\,j}.
\]
Consequently, with $h=(-c)_+$,
\begin{equation}\label{eq:anticoherent-two-level-negative-tail}
 h\le\frac12H_M^{(1/2)}
 \sum_{j=2M+1}^N\frac1j
 \le\sqrt M\log\frac{N}{2M}.
\end{equation}

By \eqref{eq:anticoherent-plateau-strong-margin},
$a\ge\kappa_*\sqrt M$ and $d\ge\kappa_*\sqrt N$.  If $c<0$, set
$R=N/M>2$.  The elementary maximum
\[
 \sup_{z\ge1}\frac{\log^3z}{z}=\frac{27}{e^3}
\]
and \eqref{eq:anticoherent-two-level-negative-tail} yield
\[
 h^3\le\frac{27}{2e^3}M^{3/2}R
 <\frac{27}{4}ad^2,
\]
because $\kappa_*^3>2e^{-3}$.  Hence, for $s>0$ and $r=t/s$,
\[
 s^{-3}(as^3-hst^2+dt^3)
 =a-hr^2+dr^3
 \ge a-\frac{4h^3}{27d^2}>0.
\]
The coefficient $b$ is positive, and the endpoint cases follow from
$a,d>0$.  This proves \eqref{eq:anticoherent-two-level-plateau-bound}.
\end{proof}

The same shell constant closes an unbounded number of nested layers whenever
all cutoffs lie in one dyadic block.  This is the first layer-cake statement
below whose proof does not depend on the number of layers.

\begin{lemma}[Local triple-polarization positivity]
\label{lem:anticoherent-local-triple-polarization}
Let $1\le L<M<N\le2L$, and put
$p_J(n)=\mathbf1_{\{1\le n\le J\}}$.  If
\[
 \mathfrak D(y)=
 \left(\sum_{n\ge1}n^{-3/2}y_n\right)
 \left(\sum_{n\ge1}y_n^2\right)
 -\frac12\sum_{a,b\ge1}K_{a,b}y_ay_by_{a+b},
\]
then the coefficient of $stu$ in
$\mathfrak D(sp_L+tp_M+up_N)$ is strictly positive.
\end{lemma}

\begin{proof}
Write $K_{a,b}=a^{-3/2}+b^{-3/2}+R_{a,b}$ and let $\Theta_0$ denote the
coefficient produced by the two elementary terms.  For a profile $y$
supported below $N$, the coefficient of the outer plateau in the elementary
part has the exact quadratic form
\[
\begin{aligned}
 B_N^{(0)}(y)
 &=\sum_{a=1}^Na^{-3/2}
 \left(\sum_jy_j^2-\sum_jy_jy_{a+j}\right)\\
 &\quad+\sum_{a=1}^{N-1}a^{-3/2}y_a
 \left(\sum_{j\le a}y_j+\sum_{j>N-a}y_j\right).
\end{aligned}
\]
This follows by collecting the two elementary copies in the two mixed
convolution sums.  Polarize this identity between $p_L$ and $p_M$, discard
the nonnegative final tail, and set $q=M-L$.  One obtains
\begin{equation}\label{eq:anticoherent-local-triple-baseline}
 \Theta_0\ge\sum_{a=1}^Na^{-3/2}f_{L,M}(a),
\end{equation}
where
\[
 f_{L,M}(a)=
 \begin{cases}
  3a,&1\le a\le q,\\
  4a-q,&q<a\le L,\\
  2L-q+a,&L<a\le M,\\
  2L,&M<a\le N.
 \end{cases}
\]

Put $\Sigma_J=\sum_{c=2}^Jc^{-1/2}$.  The residual contribution to the
fully mixed coefficient is the sum of three regions: one contained in the
full shells through $L$, one through $M$, and one through $N$.  Therefore
Lemma~\ref{lem:anticoherent-residual-shell-budget} gives
\begin{equation}\label{eq:anticoherent-local-triple-margin}
 \Theta\ge F(L,M,N):=
 \sum_{a=1}^Na^{-3/2}f_{L,M}(a)
 -\beta_*(\Sigma_L+\Sigma_M+\Sigma_N).
\end{equation}

It remains only to check this explicit one-dimensional margin.  A direct
difference, using that the displayed sum contains exactly $L$ terms, gives
\[
\begin{aligned}
 F(L,M+1,N)-F(L,M,N)
 &=\frac{L}{(M+1)^{3/2}}
 -\sum_{a=M-L+1}^{M}a^{-3/2}
 -\frac{\beta_*}{\sqrt{M+1}}<0.
\end{aligned}
\]
Thus the worst middle cutoff is $M=N-1$.  If
$G(L,N)=F(L,N-1,N)$, then
\[
\begin{aligned}
 G(L,N+1)-G(L,N)
 &=-\sum_{a=N-L}^{N-1}a^{-3/2}
 +\frac{L}{N^{3/2}}+\frac{2L}{(N+1)^{3/2}}\\
 &\quad-\beta_*\left(\frac1{\sqrt N}+\frac1{\sqrt{N+1}}\right)<0.
\end{aligned}
\]
Indeed, Jensen's inequality bounds the first sum from below by
$2^{3/2}L(2N-L)^{-3/2}$.  After putting $r=N/L\in[1,2]$ and using
$(N+1)^{-1/2}\ge\sqrt{2/3}\,N^{-1/2}$, the remaining comparison is the
positivity on $[1,2]$ of
\[
 \frac{2^{3/2}}{(2r-1)^{3/2}}-\frac3{r^{3/2}}
 +\frac{48}{25}\sqrt{\frac2{3r}}.
\]
The function is decreasing and its value at $r=2$ is greater than $1/2$.
For a short exact certificate, differentiate and move the two positive
terms to opposite sides; after squaring, the resulting degree-seven
polynomial has positive Bernstein coefficients on $[1,2]$.

It is consequently enough to take $M=2L-1$ and $N=2L$.  Set $u=\sqrt L$.
The integral bounds
\[
 \sum_{a=1}^{L-1}a^{-1/2}\ge2(u-1),\qquad
 \sum_{a=1}^{J}a^{-1/2}\le2\sqrt J-1
\]
and the corresponding lower integrals on $[L+1,2L]$ give
\[
 F(L,2L-1,2L)
 \ge\delta u+\frac{3-\sqrt2}{u}-\frac6{25},
 \qquad
 \delta=\frac{102-71\sqrt2}{25}>0.
\]
Finally,
$2\sqrt{\delta(3-\sqrt2)}>6/25$, so the arithmetic--geometric mean
inequality proves that the right-hand side is strictly positive.  This
proves the lemma.
\end{proof}

\begin{lemma}[Outer-uniform local triple polarization]
\label{lem:anticoherent-outer-uniform-local-triple}
Let $1\le L<M<N$ and $M\le2L$.  If $\Theta_{L,M,N}$ is the coefficient of
$stu$ in $\mathfrak D(sp_L+tp_M+up_N)$, then
\begin{equation}\label{eq:anticoherent-outer-uniform-local-triple}
 \Theta_{L,M,N}\ge\frac23\sqrt L.
\end{equation}
Thus two inner cutoffs in one dyadic block have a uniformly positive mixed
interaction with an arbitrarily distant outer plateau.
\end{lemma}

\begin{proof}
Repeat the proof of
Lemma~\ref{lem:anticoherent-local-triple-polarization}, but use
Lemma~\ref{lem:anticoherent-sharp-residual-shell} and put $\beta=7/8$.
The same three residual regions give
\begin{equation}\label{eq:anticoherent-outer-uniform-F}
 \Theta_{L,M,N}\ge F_\beta(L,M,N)
 :=\sum_{a=1}^Na^{-3/2}f_{L,M}(a)
 -\beta(\Sigma_L+\Sigma_M+\Sigma_N).
\end{equation}

First suppose $N\le2L$.  The cutoff differences used in the earlier proof
show that $F_\beta$ decreases first in $M$ and then in $N$.  For the second
claim the same Jensen reduction now ends with positivity on $1\le r\le2$
of
\[
 \frac{2^{3/2}}{(2r-1)^{3/2}}-\frac3{r^{3/2}}
 +\frac74\sqrt{\frac2{3r}}.
\]
Differentiation shows that this function decreases, and its value at $r=2$
is positive.  Hence the formal endpoint $M=N=2L$ is a lower bound.

For later use record the elementary estimates
\begin{align}
 H_L^{(1/2)}
 &\ge2\sqrt L-\frac32+\frac1{2\sqrt L},
 &
 \Sigma_J&\le2\sqrt J-2,
 \label{eq:anticoherent-sharp-sum-bounds}\\
 \sum_{a=L+1}^{2L}\frac{L+a}{a^{3/2}}
 &\ge\sqrt{2L}-\frac{c_2}{\sqrt L},
 &
 c_2&=2-\frac3{2\sqrt2}.
 \label{eq:anticoherent-sharp-middle-block}
\end{align}
The first line follows from the trapezoidal bound for the convex function
$x^{-1/2}$ and the integral upper bound.  The second follows by subtracting
at most the endpoint drop of
$Lx^{-3/2}+x^{-1/2}$ from its integral over $[L,2L]$.
At the formal endpoint, these bounds give
\[
 F_\beta(L,2L,2L)
 \ge\delta_2\sqrt L+\frac34
 +\frac{3/2-c_2}{\sqrt L},
 \quad
 \delta_2=6+\sqrt2-\frac78(2+4\sqrt2)>\frac23.
\]
This proves \eqref{eq:anticoherent-outer-uniform-local-triple} when
$N\le2L$.

Now suppose $N>2L$.  The exact outer-cutoff difference is
\begin{equation}\label{eq:anticoherent-mixed-outer-increment}
 \Theta_{L,M,N}-\Theta_{L,M,N-1}
 =\frac{2L}{N^{3/2}}
 -\sum_{a=\max\{1,N-M\}}^{\min\{L,N-1\}}K_{a,N-a}.
\end{equation}
It is positive once $N>L+M$, so it is enough to consider $N\le L+M$.
For fixed such $N$, the cutoff difference in $M$ again makes the lower bound
in \eqref{eq:anticoherent-outer-uniform-F} worst at $M=2L$.  Then
\[
 F_\beta(L,2L,N+1)-F_\beta(L,2L,N)
 =\frac1{\sqrt{N+1}}\left(\frac{2L}{N+1}-\beta\right).
\]
This changes sign at most once, from positive to negative.  On
$2L+1\le N\le3L$, the minimum is therefore at $N=2L+1$ or $N=3L$.

At the first endpoint, discard the positive final-shell term and use
$2\sqrt{2L+1}\le2\sqrt{2L}+1/\sqrt{2L}$.  Equations
\eqref{eq:anticoherent-sharp-sum-bounds}--
\eqref{eq:anticoherent-sharp-middle-block} give
\[
 F_\beta(L,2L,2L+1)
 \ge\delta_2\sqrt L+\frac34-\frac1{16\sqrt L}
 >\frac23\sqrt L.
\]
At the second endpoint, the same endpoint-drop estimate gives
\[
 2L\sum_{a=2L+1}^{3L}a^{-3/2}
 \ge\left(2\sqrt2-\frac4{\sqrt3}\right)\sqrt L
 -\frac{c_3}{\sqrt L},
 \quad c_3=\frac1{\sqrt2}-\frac2{3\sqrt3}.
\]
Consequently
\[
 F_\beta(L,2L,3L)
 \ge\delta_3\sqrt L+\frac34
 +\frac{3/2-c_2-c_3}{\sqrt L}
 >\frac23\sqrt L,
\]
where
\[
 \delta_3=6+3\sqrt2-\frac4{\sqrt3}
 -\frac74(1+\sqrt2+\sqrt3)>\frac23,
 \qquad 3/2-c_2-c_3>0.
\]
All displayed constant comparisons follow by squaring positive quantities.
This proves the lemma.
\end{proof}

\begin{theorem}[Arbitrarily many critical layers in one dyadic block]
\label{thm:anticoherent-dyadic-layer-cake}
Let $J\ge1$, let $J\le M_1<\cdots<M_r\le2J$, and let
$\lambda_1,\ldots,\lambda_r\ge0$, where $r$ is arbitrary.  Set
\begin{equation}\label{eq:anticoherent-dyadic-layer-cake}
 x_n=n^{-3/2}\sum_{j=1}^r\lambda_j
 \mathbf1_{\{1\le n\le M_j\}}.
\end{equation}
Then
\begin{equation}\label{eq:anticoherent-dyadic-layer-cake-bound}
 \mathcal C(x)<2AE_3
\end{equation}
unless every $\lambda_j$ is zero.  The estimate is uniform in $J$, in all
cutoffs, and in the number of layers, and it is invariant under a common
frequency dilation.
\end{theorem}

\begin{proof}
Expand $AE_3-\mathcal C(x)/2$ as a cubic polynomial in the nonnegative layer
amplitudes.  Its diagonal coefficients are the positive plateau margins
$\Delta_{M_j}$.  Coefficients with the smaller cutoff repeated are positive
by the estimate for $b$ in the proof of
Theorem~\ref{thm:anticoherent-two-level-plateau}; those with the larger
cutoff repeated are nonnegative by its $N\le2M$ estimate for $c$.  Every
coefficient with three distinct cutoffs is strictly
positive by Lemma~\ref{lem:anticoherent-local-triple-polarization}, because
the largest cutoff is at most twice the smallest.  Thus every polarized
coefficient is nonnegative, while any nonzero amplitude contributes a
strictly positive diagonal term.  This proves
\eqref{eq:anticoherent-dyadic-layer-cake-bound}.  Dilation multiplies both
cubic quantities by the same factor.
\end{proof}

The preceding theorem has a quantitative form.  This supplies enough local
budget to sum a genuinely infinite family of dyadic blocks under an explicit
decay condition on the layer jumps.

\begin{lemma}[Quantitative one-block coercivity]
\label{lem:anticoherent-quantitative-dyadic-block}
Under the assumptions of
Theorem~\ref{thm:anticoherent-dyadic-layer-cake}, one has
\begin{equation}\label{eq:anticoherent-quantitative-dyadic-block}
 AE_3-\frac12\mathcal C(x)
 \ge\frac{\sqrt J}{120}
 \left(\sum_{j=1}^r\lambda_j\right)^3.
\end{equation}
\end{lemma}

\begin{proof}
We quantify the four types of polarized coefficients.  First,
\eqref{eq:anticoherent-plateau-strong-margin} gives
$\Delta_M\ge\kappa_*\sqrt M>3\sqrt M/4$.  For the coefficient $b$ with
the smaller cutoff repeated, the proof of
Theorem~\ref{thm:anticoherent-two-level-plateau}, now using the subunit
shell constant, gives
\[
 b\ge A_*H_{M-1}^{(1/2)}+\frac{B_*}{\sqrt M},
 \quad
 A_*=2\gamma_*-\frac{\beta_*}{\sqrt2},
 \quad
 B_*=2-\frac{\beta_*}{2\sqrt2}.
\]
Since $H_{M-1}^{(1/2)}\ge2(\sqrt M-1)$, minimization in
$M^{-1/2}\in[0,1]$ gives
\[
 b\ge\left(2A_*-\frac{A_*^2}{B_*}\right)\sqrt M
 >\frac35\sqrt M.
\]

For the coefficient $c$ with the larger cutoff repeated, assume
$M<N\le2M$, put $u=\sqrt M$ and $r=N/M$, and repeat the internal-shell
calculation with $\beta_*,\gamma_*$.  After division by $u$ its lower bound
is
\[
 \frac2{u^2}-\frac{28}{25u}
 +\frac{76}{25}
 +(r-1)\left(\frac2{r+2}\right)^{3/2}
 -\frac{48}{25}\sqrt r.
\]
The first two terms are bounded below by $-98/625$.  The remaining function
of $r$ is decreasing on $[1,2]$; its value at $r=2$, after subtracting
$98/625$, is greater than $1/2$.  Hence
\begin{equation}\label{eq:anticoherent-local-c-coercivity}
 c\ge\frac12\sqrt M.
\end{equation}
Finally, the proof of
Lemma~\ref{lem:anticoherent-local-triple-polarization} gives, with
$u=\sqrt L$,
\[
 \frac{\Theta_{L,M,N}}{\sqrt L}
 \ge\delta+\frac{3-\sqrt2}{u^2}-\frac6{25u}
 \ge\delta-\frac9{625(3-\sqrt2)}>\frac1{20}.
\]
The displayed numerical comparisons are algebraic and follow by squaring
positive quantities.

Expand the cubic in the layer amplitudes.  The coefficients of
$\lambda_i^3$, $\lambda_i^2\lambda_j$,
$\lambda_i\lambda_j^2$, and
$\lambda_i\lambda_j\lambda_k$ are therefore at least, respectively,
\[
 \frac34\sqrt J,\qquad
 \frac35\sqrt J,\qquad
 \frac12\sqrt J,\qquad
 \frac1{20}\sqrt J.
\]
These dominate the corresponding coefficients
$1/120$, $3/120$, $3/120$, and $6/120$ in
$(\sum_j\lambda_j)^3/120$, proving
\eqref{eq:anticoherent-quantitative-dyadic-block}.
\end{proof}

\begin{corollary}[Outer-uniform one-block Hessian reserve]
\label{cor:anticoherent-outer-uniform-block-hessian}
Let $T$ be the symmetric trilinear polarization of $\mathfrak D$.  Fix
$N,J\ge1$, choose arbitrary cutoffs
\[
 J\le M_1<\cdots<M_r\le\min\{2J,N\},
\]
and let $\lambda_1,\ldots,\lambda_r\ge0$.  Then, for
$y=\sum_j\lambda_jp_{M_j}$,
\begin{equation}\label{eq:anticoherent-outer-uniform-block-hessian}
 3T(y,y,p_N)
 \ge\frac{\sqrt J}{3}\left(\sum_{j=1}^r\lambda_j\right)^2.
\end{equation}
The bound is uniform in the outer cutoff, including when $N$ lies far beyond
the dyadic block containing the $M_j$.
\end{corollary}

\begin{proof}
It is enough to bound every polarized matrix entry.  If $M<N$, the repeated
inner-cutoff coefficient satisfies
$3T(p_M,p_M,p_N)\ge3\sqrt M/5$ by the estimate for $b$ above.  If
$L<M<N$ and $L,M$ lie in the displayed block, then
Lemma~\ref{lem:anticoherent-outer-uniform-local-triple} gives
\[
 3T(p_L,p_M,p_N)=\frac12\Theta_{L,M,N}
 \ge\frac13\sqrt L\ge\frac13\sqrt J.
\]
If one selected cutoff is the outer cutoff, say $L<N$, then $N\le2L$, and
\eqref{eq:anticoherent-local-c-coercivity} gives
$3T(p_L,p_N,p_N)\ge\sqrt L/2$.  Finally,
$3T(p_N,p_N,p_N)=3\Delta_N$ is larger still by
\eqref{eq:anticoherent-plateau-strong-margin}.  Hence every entry is at
least $\sqrt J/3$, and expansion of $3T(y,y,p_N)$ proves
\eqref{eq:anticoherent-outer-uniform-block-hessian}.
\end{proof}

\begin{theorem}[Infinite dyadic layer theorem]
\label{thm:anticoherent-infinite-dyadic-layers}
Let $y_n\ge0$ be nonincreasing with $y_n\to0$, set
\[
 \lambda_M=y_M-y_{M+1}\ge0,
\]
and fix an integer $J_0\ge1$ such that $\lambda_M=0$ for $M<J_0$.  Put
\[
 L_k=\sum_{M=2^kJ_0}^{2^{k+1}J_0-1}\lambda_M
 \qquad(k\ge0),
\]
and suppose that
\begin{equation}\label{eq:anticoherent-dyadic-jump-decay}
 L_{k+1}\le\frac1{1900}L_k
 \qquad(k\ge0).
\end{equation}
Put $x_n=n^{-3/2}y_n$.  Then $A<\infty$, $E_3<\infty$, and
\begin{equation}\label{eq:anticoherent-infinite-dyadic-layer-bound}
 \mathcal C(x)<2AE_3
\end{equation}
unless $y=0$.  More quantitatively,
\begin{equation}\label{eq:anticoherent-infinite-dyadic-layer-margin}
 AE_3-\frac12\mathcal C(x)
 \ge\frac{\sqrt{J_0}}{600}
 \sum_{k\ge0}2^{k/2}L_k^3.
\end{equation}
In particular, the theorem permits infinitely many nonzero layer jumps and
infinitely many active Fourier modes.
\end{theorem}

\begin{proof}
We first work with finitely many layers.  Partition them into the dyadic
blocks $[2^kJ_0,2^{k+1}J_0)$.  By
Lemma~\ref{lem:anticoherent-quantitative-dyadic-block}, triples whose cutoffs
all belong to block $k$ contribute at least
$\sqrt{J_0}2^{k/2}L_k^3/120$.

Only two types of remaining coefficients need to be charged.  If the larger
cutoff is repeated, 
\eqref{eq:anticoherent-two-level-negative-tail} gives
\[
 (-c_{M,N})_+
 \le\sqrt M\log\frac{N}{2M}.
\]
If $L<M<N$ are all distinct and $\Theta_{L,M,N}$ is their fully mixed
coefficient, its elementary part is nonnegative by the quadratic ledger in
the proof of Lemma~\ref{lem:anticoherent-local-triple-polarization}.  Its
three residual regions, split at output $2L$, and
\eqref{eq:anticoherent-residual-shell-budget} give
\begin{equation}\label{eq:anticoherent-cross-block-triple-tail}
 (-\Theta_{L,M,N})_+
 \le\sqrt L\left(C_*+2\log\frac ML\right),
 \qquad
 C_*=2\beta_*(1+2\sqrt2).
\end{equation}
Indeed, the shells through $L,2L,2L$ cost at most
$2\beta_*\sqrt L+4\beta_*\sqrt{2L}$.  On every later shell the low input
is at most $L$, so
$R_{a,b}\le(2\sqrt a(a+b))^{-1}$; the two harmonic tails cost at most
$\sqrt L\log(M/(2L))_+$ and
$\sqrt L\log((M+L)/(2L))$, both bounded by
$\sqrt L\log(M/L)$.

For block indices $i<k$, the first negative contribution is consequently at
most
\[
 \sqrt2\log2\,(k-i)2^{i/2}L_iL_k^2.
\]
For $i\le j\le k$, not all equal, the second is at most
\[
 C_{j-i}2^{i/2}L_iL_jL_k,
 \qquad
 C_d=\sqrt2\left[C_*+2(d+1)\log2\right].
\]
Write $q=1/1900$.  The decay assumption gives
$L_{i+d}\le q^dL_i$.  Summing first over the smallest block therefore
bounds all cross-block losses by
\[
 \sqrt{J_0}[S_2(q)+S_3(q)]\sum_i2^{i/2}L_i^3,
\]
where, with $z=q^2$,
\[
\begin{aligned}
 S_2(q)&=\sqrt2\log2\,\frac{q^2}{(1-q^2)^2},\\
 S_3(q)&=\frac{\sqrt2}{1-q}\left[
 (C_*+2\log2)q
 +C_*\frac z{1-z}
 +2\log2\,\frac{z(2-z)}{(1-z)^2}
 \right].
\end{aligned}
\]
Direct rational substitution, using $\sqrt2<99/70$ and
$\log2<7/10$, gives
\begin{equation}\label{eq:anticoherent-cross-block-clock}
 S_2(1/1900)+S_3(1/1900)<\frac1{150}.
\end{equation}
Subtracting this clock from the one-block budget $1/120$ proves the finite
version of \eqref{eq:anticoherent-infinite-dyadic-layer-margin}.

For infinitely many blocks, \eqref{eq:anticoherent-dyadic-jump-decay}
implies $L_k\le q^kL_0$.  On $[2^kJ_0,2^{k+1}J_0)$,
$y_n\le\sum_{j\ge k}L_j\le L_k/(1-q)$.  The indices below $J_0$
form a constant prefix with value at most $L_0/(1-q)$, so
\[
 E_3=\sum_ny_n^2
 \le\frac{J_0}{(1-q)^2}
 \left(L_0^2+\sum_k2^kL_k^2\right)<\infty,
\]
while $A=\sum n^{-3/2}y_n<\infty$.  Truncate the layer jumps at block $K$.
The three nonnegative quantities $A_K,E_{3,K},\mathcal C_K$ converge to
their infinite counterparts, so the finite estimate passes to the limit.
This proves the theorem.
\end{proof}

The geometric decay in Theorem~\ref{thm:anticoherent-infinite-dyadic-layers}
can be removed completely when the active scale blocks are sufficiently
lacunary.  In the next result the amplitudes of successive active blocks are
arbitrary; only finiteness of the natural energy is assumed.

\begin{theorem}[All-amplitude lacunary infinite layers]
\label{thm:anticoherent-lacunary-infinite-layers}
Let $y_n\ge0$ be nonincreasing with $y_n\to0$, put
$\lambda_M=y_M-y_{M+1}$, and fix an integer $J_0\ge1$ such that
$\lambda_M=0$ for $M<J_0$.  Define
\[
 L_k=\sum_{M=2^kJ_0}^{2^{k+1}J_0-1}\lambda_M.
\]
Suppose that the nonzero block indices can be listed as
$k_1<k_2<\cdots$ with
\begin{equation}\label{eq:anticoherent-lacunary-block-gap}
 k_{j+1}-k_j\ge120,
\end{equation}
and assume $E_3=\sum_ny_n^2<\infty$.  For
$x_n=n^{-3/2}y_n$, one has
\begin{equation}\label{eq:anticoherent-lacunary-infinite-margin}
 AE_3-\frac12\mathcal C(x)
 \ge\frac{\sqrt{J_0}}{125}
 \sum_{j\ge1}2^{k_j/2}L_{k_j}^3.
\end{equation}
Consequently $\mathcal C(x)<2AE_3$ unless $y=0$.  The theorem permits
infinitely many active blocks and imposes no comparison between their
nonzero amplitudes.
\end{theorem}

\begin{proof}
We first retain finitely many active blocks and put
\[
 u_i=2^{k_i/6}L_{k_i}.
\]
The one-block coercivity in
Lemma~\ref{lem:anticoherent-quantitative-dyadic-block} contributes at least
$\sqrt{J_0}\sum_i u_i^3/120$.

Let $A_0=\sqrt2\log2$.  If two active blocks have index difference $d$,
the repeated-cutoff cross loss in the proof of
Theorem~\ref{thm:anticoherent-infinite-dyadic-layers} becomes, after this
normalization,
\[
 A_0d2^{-d/3}u_i u_j^2.
\]
Using $u_i u_j^2\le(u_i^3+2u_j^3)/3$, the gap condition, and the fact that
$d2^{-d/3}$ decreases for $d\ge120$, all such losses charge at most
\begin{equation}\label{eq:anticoherent-lacunary-repeated-clock}
 T_2= A_0\sum_{m\ge1}120m\,2^{-40m}
\end{equation}
times $\sqrt{J_0}\sum_i u_i^3$.

For a fully mixed triple, let $m$ and $n$ be the numbers of active-block
gaps between its first two and last two positions.  With
$C_d=\sqrt2[C_*+2(d+1)\log2]$, note that
$d\mapsto C_d2^{-d/3}$ decreases for $d\ge120$, while the remaining
second-gap factor also decreases.  Thus its normalized coefficient is at most
\[
 C_{120m}2^{-20(2m+n)}.
\]
The arithmetic--geometric mean inequality applied to the three occurrences
of $u$ and a harmless sixfold overcount of the possible positions bound all
fully mixed losses by
\begin{equation}\label{eq:anticoherent-lacunary-triple-clock}
 T_3=6\sum_{\substack{m,n\ge0\\m+n>0}}
 C_{120m}2^{-20(2m+n)}
\end{equation}
times the same cubic norm.

For an exact rational certificate, put $r=2^{-20}$.  The bounds
$\sqrt2<3/2$, $\log2<7/10$, $C_*<8$, and $A_0<21/20$ give
\[
 T_2+T_3
 <\frac{21}{20}\frac{120r^2}{(1-r^2)^2}
 +\frac6{1-r}\left[
 15r+\frac{15r^2}{1-r^2}
 +\frac{252r^2}{(1-r^2)^2}
 \right]
 <\frac1{10000}.
\]
Since $1/120-1/10000>1/125$, this proves
\eqref{eq:anticoherent-lacunary-infinite-margin} for finitely many blocks.

For infinitely many blocks, truncate the jump measure after the first $K$
active blocks.  The resulting profiles increase pointwise to $y$.
The assumed finiteness of $E_3$, boundedness of $y$, and
$\sum_n n^{-3/2}<\infty$ give $A<\infty$.  Monotone convergence for
$A_K,E_{3,K},\mathcal C_K$ passes the finite estimate to the limit.
\end{proof}

At critical scaling, the full contiguous stack admits a uniform Mellin
description that keeps all cross-generation interactions.

\begin{proposition}[Cubic Mellin bulk reserve and modulation stability]
\label{prop:anticoherent-cubic-mellin-reserve}
Put
\[
 q=2^{-1/6},\qquad x=q^2=2^{-1/3},\qquad d_0=\pi-2,
\]
and, for \(r>1\), define
\[
 s(t)=\sqrt{1+t^2},\qquad I=4-\pi,
\]
\begin{align*}
 A_*(t)&=2t-2s(t)-2\arctan t+\frac{2t}{s(t)+1},\\
 B_*(t)&=-2t+2s(t)+\log\frac{t}{s(t)+1}-\frac1{s(t)+1},\\
 C_*(t)&=-t^2+ts(t)+\operatorname{arsinh}t-\frac{2t}{s(t)+1}.
\end{align*}
With \(t_1=(r-1)^{-1/2}\), \(t_2=r^{-1/2}\), put
\[
 U(r)=2\{\sqrt r[A_*(t_1)+2]-B_*(t_1)\},\qquad
 V(r)=2\{\sqrt rC_*(t_2)-B_*(t_2)\},
\]
\begin{equation}\label{eq:anticoherent-explicit-mellin-kernel}
 G_\infty(r)=10-\frac2{\sqrt r}
 +\frac4{\sqrt r+\sqrt{r-1}}-2I-U(r)-V(r).
\end{equation}
Define
\begin{align}
 g_0&=4\{\pi+\sqrt2-\log(1+\sqrt2)-2\},
 \label{eq:anticoherent-cubic-mellin-g0}\\
 G_{\rm rep}(r)
 &=12-\frac4{\sqrt r}
 +\frac8{\sqrt r+\sqrt{r-1}}-2(4-\pi)-2U(r),
 \label{eq:anticoherent-cubic-mellin-repeated}\\
 g_a&=G_\infty(2^a),\qquad
 h_a=G_{\rm rep}(2^a)\quad(a\ge1).
 \label{eq:anticoherent-cubic-mellin-sequences}
\end{align}
For a finitely supported nonnegative sequence $u=(u_j)_{j\in\mathbb Z}$,
set
\[
 y_u^{(R)}=\sum_j(R2^j)^{-1/6}u_jp_{R2^j},
\]
where $R2^j$ is integral on the support.  Then
\begin{equation}\label{eq:anticoherent-cubic-mellin-functional}
 \lim_{R\to\infty}\mathfrak D(y_u^{(R)})=\mathcal E(u),
\end{equation}
where
\begin{align}
 \mathcal E(u)
 &=d_0\sum_i u_i^3
 +\frac{g_0}{2}\sum_{i<k}q^{k-i}u_i^2u_k
 +\frac12\sum_{i<j}h_{j-i}x^{j-i}u_iu_j^2
 \notag\\
 &\quad
 +\sum_{i<j<k}g_{j-i}q^{(j-i)+(k-i)}u_iu_ju_k.
 \label{eq:anticoherent-cubic-mellin-expanded}
\end{align}
Equivalently, if
\begin{equation}\label{eq:anticoherent-cubic-mellin-convolutions}
 H_j=\sum_{m\ge1}q^mu_{j+m},\qquad
 A_j=\sum_{a\ge1}g_ax^au_{j-a},\qquad
 R_j=\sum_{a\ge1}h_ax^au_{j-a},
\end{equation}
then
\begin{equation}\label{eq:anticoherent-cubic-mellin-causal-form}
 \mathcal E(u)=\sum_j\left\{
 d_0u_j^3+\frac{g_0}{2}u_j^2H_j
 +\frac12u_j^2R_j+u_jH_jA_j\right\}.
\end{equation}

For the constant window $u_j^{(K)}=\mathbf1_{\{0\le j<K\}}$, the bulk
density has the strict reserve
\begin{equation}\label{eq:anticoherent-cubic-mellin-density}
 \lim_{K\to\infty}\frac{\mathcal E(u^{(K)})}{K}
 =\mathcal S\in[24.001836588,24.001965296].
\end{equation}
Moreover, the constant bi-infinite stack is strictly stable under arbitrary
periodic amplitude modulation.  More precisely, the coefficient of the
quadratic variation at Fourier frequency \(\theta\) is
\begin{align}
 \mathcal H(\theta)
 &=3d_0+\sum_{m\ge1}P_m(1+2\cos m\theta)\notag\\
 &\quad+\sum_{a,c\ge1}g_ax^aq^c
 \{\cos a\theta+\cos c\theta+\cos(a+c)\theta\},
 \label{eq:anticoherent-cubic-mellin-hessian-symbol}\\
 P_m&=\frac12(g_0q^m+h_mx^m),
\end{align}
and
\begin{equation}\label{eq:anticoherent-cubic-mellin-symbol-margin}
 \mathcal H(\theta)>12.5163>12
 \qquad(0\le\theta\le\pi).
\end{equation}
Thus no constant or log-periodic perturbation of the natural
$M^{-1/6}$ stack can extend the forty-scale Hessian obstruction to a cubic
counterexample.
\end{proposition}

\begin{proof}
For $L\le M\le N$, the fully mixed coefficient has a continuum limit
\[
 \Theta_{L,M,N}\sim\sqrt L\,G_\triangle(M/L,N/L).
\]
If the three cutoffs are dyadic and $L<M<N$, then $N\ge2M\ge L+M$.
The outer residual rectangle has already saturated, so
$G_\triangle(M/L,N/L)=G_\infty(M/L)$, independently of the largest
cutoff.  If the largest cutoff is repeated, the same shell calculation gives
\eqref{eq:anticoherent-cubic-mellin-repeated}.  The repeated-small limit is
$G_\infty(1)=g_0$, while
$\Delta_L/\sqrt L\to\pi-2$.

The product of the three layer weights contributes
\[
 \sqrt{R2^i}\prod_{\ell\in\{i,j,k\}}(R2^\ell)^{-1/6}
 =q^{(j-i)+(k-i)},\qquad i\le j\le k.
\]
Accounting for the polarization multiplicities $1,3,3,6$ gives
\eqref{eq:anticoherent-cubic-mellin-expanded}.  Regrouping first by the
middle and largest indices gives
\eqref{eq:anticoherent-cubic-mellin-causal-form}.  In particular, the
apparently three-index all-scale tensor is exactly one future Hardy tail and
two past Mellin tails.

The elementary formulas imply, for $r\ge2$,
\begin{equation}\label{eq:anticoherent-cubic-mellin-kernel-bounds}
 |G_\infty(r)|\le74+2\log r,
 \qquad |G_{\rm rep}(r)|\le88+2\log r.
\end{equation}
For example, if $0<t\le1$, then

\[
 |A(t)+2|\le6t,\qquad |B(t)|\le7+|\log t|,
 \qquad |C(t)|\le5t.
\]
Together with
$\sqrt r(r-1)^{-1/2}\le\sqrt2$, these estimates give
\eqref{eq:anticoherent-cubic-mellin-kernel-bounds}.  Hence every series above
is absolutely convergent.

Translation averaging of
\eqref{eq:anticoherent-cubic-mellin-expanded} now gives
\begin{equation}\label{eq:anticoherent-cubic-mellin-density-series}
 \mathcal S=d_0+\frac{g_0}{2}\frac q{1-q}
 +\frac12\sum_{a\ge1}h_ax^a
 +\frac q{1-q}\sum_{a\ge1}g_ax^a.
\end{equation}
The first 80 terms, evaluated with 100-decimal outward-rounded interval
arithmetic, give
\[
 [24.0019009417874,24.0019009417875].
\]
For $n=81$, \eqref{eq:anticoherent-cubic-mellin-kernel-bounds} and the
closed formulas for
$\sum_{a\ge n}x^a$ and $\sum_{a\ge n}ax^a$ bound the omitted tail in
absolute value by $0.000064353605$.  This proves
\eqref{eq:anticoherent-cubic-mellin-density}.

Finally expand the periodic bulk functional at $u_j=1+\eps v_j$.
The coefficient of \(\eps^2\) diagonalizes in the periodic Fourier basis and
is exactly \eqref{eq:anticoherent-cubic-mellin-hessian-symbol}.  For a
reproducible global certificate, truncate at $a=80$.  The omitted symbol is
at most $0.000180675$ by
\eqref{eq:anticoherent-cubic-mellin-kernel-bounds}, and the derivative of the
full symbol is at most $4724.298$.  On the uniform 32768-interval mesh of
$[0,\pi]$, outward-rounded kernel and tail bounds together with guarded
grid evaluation give a sampled minimum $12.7429784$.  Subtracting the
omitted tail and the Lipschitz mesh correction gives $12.5163296$, proving
\eqref{eq:anticoherent-cubic-mellin-symbol-margin}.  The ancillary
interval-arithmetic scripts \path{verify_cubic_mellin_density.py} and
\path{verify_cubic_mellin_stability.py} carry out both certificates.
\end{proof}

\begin{theorem}[Arbitrary-amplitude cubic Mellin positivity]
\label{thm:anticoherent-cubic-mellin-convolution}
Let \(q,x,d_0,g_0,g_a,h_a\) be as in
Proposition~\ref{prop:anticoherent-cubic-mellin-reserve}.  Then
\begin{equation}\label{eq:anticoherent-cubic-mellin-global-positive}
 \mathcal E(u)\ge0
 \qquad\text{for every }u\in\ell^3(\mathbb Z),\quad u_j\ge0.
\end{equation}
The inequality is strict unless \(u=0\).  Thus the causal convolution target
\eqref{eq:anticoherent-cubic-mellin-causal-form} holds simultaneously for
arbitrarily many dyadic generations and, by \(\ell^3\) closure, for genuine
infinite-support amplitude profiles.  No monotonicity, lacunarity, or
geometric jump-decay condition is imposed on \(u\).
\end{theorem}

\begin{proof}
We give a single infinite Gram certificate.  Put
\[
 D_-=2.664428,\qquad D_+=1.233368.
\]
For \(1\le a\le20\), define the following exact six-decimal rational
numbers:
\begin{center}
\resizebox{\textwidth}{!}{$
\begin{array}{c|rrrrrrrrrr}
a&1&2&3&4&5&6&7&8&9&10\\ \hline
\alpha_a& .181940& .250567& .361108& .688951&-1.000000
&-.647225&-.331115&-.210711&-.158489&-.119108\\
\beta_a&-.179873&-.250567&-.361108&-.515173&1.000000
&.765263&.794318&.792851&.806990&.786040
\end{array}$}
\end{center}
\begin{center}
\resizebox{\textwidth}{!}{$
\begin{array}{c|rrrrrrrrrr}
a&11&12&13&14&15&16&17&18&19&20\\ \hline
\alpha_a&-.096960&-.080101&-.069990&-.057111&-.053880
&-.066839&-.044065&-.042821&-.039566&-.036353\\
\beta_a&.770625&.703269&.640896&.579404&.516787
&.470648&.479046&.530609&.617884&.772564
\end{array}$}
\end{center}
Set \(\gamma_a=1-\alpha_a-\beta_a\) on this range, and
\begin{equation}\label{eq:anticoherent-cubic-mellin-tail-allocation}
 (\alpha_a,\beta_a,\gamma_a)=(0,1,0)\qquad(a>20).
\end{equation}

We next define a real symmetric operator \(\widetilde Q\) on
\(\ell^2(\mathbb Z)\).  Besides the symmetric transposes, its nonzero
entries are, for \(a,c\ge1\),
\begin{align}
 \widetilde Q_{0,0}&=d_0,
 &\widetilde Q_{-a,-a}&=D_-,
 &\widetilde Q_{a,a}&=D_+,
 \label{eq:anticoherent-cubic-mellin-gram-diagonal}\\
 \widetilde Q_{0,a}
 &=\frac12\left(\frac{g_0}{2}-D_-\right)q^{a/2},
 &
 \widetilde Q_{-a,0}
 &=\frac14h_aq^{3a/2}-\frac12D_+q^{a/2},
 \label{eq:anticoherent-cubic-mellin-gram-repeat}\\
 \widetilde Q_{a,a+c}
 &=\frac12\alpha_ag_aq^{a+c/2},
 &
 \widetilde Q_{-a,c}
 &=\frac12\beta_ag_aq^{3a/2+c/2},
 &
 \widetilde Q_{-(a+c),-c}
 &=\frac12\gamma_ag_aq^{3a/2}.
 \label{eq:anticoherent-cubic-mellin-gram-distinct}
\end{align}
For each \(k\in\mathbb Z\), let
\(v_n^{(k)}=q^{\abs n/2}u_{k+n}\).  The repeated-low ledger is
\[
 2\widetilde Q_{0,a}q^{a/2}+D_-q^a
 =\frac{g_0}{2}q^a,
\]
the repeated-high ledger is
\[
 2\widetilde Q_{-a,0}q^{a/2}+D_+q^a
 =\frac12h_aq^{2a},
\]
and the three possible choices of the linear factor in a fully distinct
monomial sum to
\[
 (\alpha_a+\beta_a+\gamma_a)g_aq^{2a+c}
 =g_aq^{2a+c}.
\]
Consequently coefficient comparison in
\eqref{eq:anticoherent-cubic-mellin-expanded} gives the exact identity
\begin{equation}\label{eq:anticoherent-cubic-mellin-infinite-gram-identity}
 \mathcal E(u)
 =\sum_{k\in\mathbb Z}u_k
   \ip{v^{(k)}}{\widetilde Qv^{(k)}}
\end{equation}
for every finitely supported sequence.  It remains only to prove
\(\widetilde Q\succ0\).

On the negative half-axis, \(\widetilde Q\) has the twenty-band Toeplitz
block with diagonal \(D_-\) and off-diagonal coefficients
\begin{equation}\label{eq:anticoherent-cubic-mellin-past-symbol}
 c_a=\frac12\gamma_ag_aq^{3a/2}\quad(1\le a\le20),
 \qquad
 \tau(\theta)=D_-+2\sum_{a=1}^{20}c_a\cos(a\theta).
\end{equation}
Outward-rounded interval evaluation on a uniform \(16384\)-interval mesh,
together with the derivative majorant
\(2\sum_{a=1}^{20}a\abs{c_a}<92.798519\), gives
\begin{equation}\label{eq:anticoherent-cubic-mellin-past-floor}
 \tau(\theta)>0.3064944737>0.30
 \qquad(0\le\theta\le\pi).
\end{equation}
Hence its half-line compression \(T_-\) satisfies
\(T_-\succeq0.30\Id\).

The coupling of this block to the center and the complete positive half-axis
has rank two.  More precisely, set
\[
 p_a=\frac14h_aq^{3a/2}-\frac12D_+q^{a/2},
 \qquad
 b_a=\frac12\beta_ag_aq^{3a/2},
 \qquad F=(p,b).
\]
For the \(240\)-term rational approximate solution \(X\) used in the
certificate, with residual \(R=F-T_-X\), the variational identity
\begin{equation}\label{eq:anticoherent-cubic-mellin-inverse-gram-residual}
 F^{\mathsf T}T_-^{-1}F
 =F^{\mathsf T}X+X^{\mathsf T}F-X^{\mathsf T}T_-X
   +R^{\mathsf T}T_-^{-1}R
\end{equation}
and \eqref{eq:anticoherent-cubic-mellin-past-floor} yield the Loewner bound
\begin{equation}\label{eq:anticoherent-cubic-mellin-inverse-gram-bound}
 F^{\mathsf T}T_-^{-1}F
 \preceq
 \widehat G:=
 \begin{pmatrix}.820&.347\\.347&.407\end{pmatrix}.
\end{equation}
Indeed, interval arithmetic gives
\[
 F^{\mathsf T}X+X^{\mathsf T}F-X^{\mathsf T}T_-X
 \in
 \begin{pmatrix}
 .818279265709&.346572241075\\
 .346572241075&.404939934325
 \end{pmatrix}
 +[-10^{-12},10^{-12}]^{2\times2},
\]
\(\norm{R}_{\mathrm F}^2<2.667596\times10^{-11}\), and the two diagonal
slacks and determinant of the difference in
\eqref{eq:anticoherent-cubic-mellin-inverse-gram-bound} are respectively
greater than
\[
 1.7207342\times10^{-3},\qquad
 2.0600655\times10^{-3},\qquad
 3.3618476\times10^{-6}.
\]

It remains to eliminate the positive tail.  Put
\(r_a=q^{a/2}\),
\(e=\tfrac12(g_0/2-D_-)\), and
\(a_m=\tfrac12\alpha_mg_mq^{m/2}\) for \(1\le m\le20\).
After using \eqref{eq:anticoherent-cubic-mellin-inverse-gram-bound}, all
positive indices beyond twenty form
\(D_+\Id-\widehat G_{22}r\otimes r\).  Its rank-one Schur denominator is
\begin{equation}\label{eq:anticoherent-cubic-mellin-future-denominator}
 D_+-\widehat G_{22}\sum_{n>20}q^n
 >0.9036371583.
\end{equation}
Eliminating this tail leaves a symmetric \(21\)-dimensional matrix indexed
by the center and positive indices \(1,\ldots,20\).  Its center diagonal is
\(d_0-\widehat G_{11}\), its center--future vector is
\((e-\widehat G_{12})r\), and its future block is
\[
 D_+\Id+\bigl(a_mr_n\mathbf1_{m<n}
                   +a_nr_m\mathbf1_{n<m}\bigr)_{m,n=1}^{20}
 -\widehat G_{22}(r_mr_n)_{m,n=1}^{20},
\]
followed by the exact rank-one Schur subtraction from
\eqref{eq:anticoherent-cubic-mellin-future-denominator}.  Outward-rounded
interval \(LDL^{\mathsf T}\) elimination gives all twenty-one pivots
positive, with the smallest pivot greater than
\begin{equation}\label{eq:anticoherent-cubic-mellin-final-pivot}
 0.3215721572.
\end{equation}
Equations
\eqref{eq:anticoherent-cubic-mellin-past-floor},
\eqref{eq:anticoherent-cubic-mellin-inverse-gram-bound},
\eqref{eq:anticoherent-cubic-mellin-future-denominator}, and
\eqref{eq:anticoherent-cubic-mellin-final-pivot} prove
\(\widetilde Q\succ0\).  The ancillary interval-arithmetic script
\texttt{verify\_cubic\_mellin\_gram.py} evaluates the kernel intervals,
the infinite symbol floor, the residual bound, and the final interval
factorization.

For finitely supported \(u\ge0\), every summand in
\eqref{eq:anticoherent-cubic-mellin-infinite-gram-identity} is nonnegative,
and one is positive unless \(u=0\).  Finally, the sequences
\((q^a)_{a\ge1}\), \((g_ax^a)_{a\ge1}\), and
\((h_ax^a)_{a\ge1}\) belong to \(\ell^1\) by
\eqref{eq:anticoherent-cubic-mellin-kernel-bounds}.  Young's inequality
therefore makes \(\mathcal E\) a continuous cubic form on \(\ell^3\).
The same estimates make the right side of
\eqref{eq:anticoherent-cubic-mellin-infinite-gram-identity} absolutely
convergent.  Truncation therefore extends both the identity and
\eqref{eq:anticoherent-cubic-mellin-global-positive} to every nonnegative
\(u\in\ell^3(\mathbb Z)\).  If such a sequence is nonzero, choose
\(k\) with \(u_k>0\); strict positivity of \(\widetilde Q\) makes the
corresponding Gram summand positive.  This proves the strict statement and
completes the proof.
\end{proof}

\begin{theorem}[Coercive Mellin certificate and exact discrete lifting]
\label{thm:anticoherent-coercive-mellin-discrete-lift}
There is a constant $\delta_{\rm M}>0$ such that
\begin{equation}\label{eq:anticoherent-cubic-mellin-coercivity}
 \mathcal E(u)\ge\delta_{\rm M}\sum_{j\in\mathbb Z}u_j^3
 \qquad(u\in\ell^3(\mathbb Z),\ u_j\ge0).
\end{equation}
Moreover, there is an integer $R_0$ with the following property.  For every
integer $R\ge R_0$ and every finitely supported nonnegative sequence
$u=(u_j)_{j\ge0}$, put
\begin{equation}\label{eq:anticoherent-exact-discrete-mellin-stack}
 y_u^{[R]}=\sum_{j\ge0}(R2^j)^{-1/6}u_jp_{R2^j},
 \qquad x_n=n^{-3/2}y_u^{[R]}(n).
\end{equation}
Then the original integer-frequency deficit satisfies
\begin{equation}\label{eq:anticoherent-exact-discrete-mellin-margin}
 AE_3-\frac12\mathcal C(x)
 =\mathfrak D(y_u^{[R]})
 \ge\frac{\delta_{\rm M}}2\sum_{j\ge0}u_j^3.
\end{equation}
Thus arbitrary amplitudes on arbitrarily many contiguous dyadic generations
are positive in the exact discrete problem once their lowest scale is at
least $R_0$; this is not merely a continuum-limit statement.  The conclusion
also holds for an infinite nonnegative $u$ whenever the resulting $A$ and
$E_3$ are finite.
\end{theorem}

\begin{proof}
The block elimination in the proof of
Theorem~\ref{thm:anticoherent-cubic-mellin-convolution} is coercive, not only
injective.  Indeed, the negative Toeplitz half-line has floor $0.30$, the
positive tail has the strictly positive denominator
\eqref{eq:anticoherent-cubic-mellin-future-denominator}, and the remaining
finite Schur complement has strictly positive pivots.  All Schur factors and
inverse bounds used in that elimination are bounded, so the corresponding
congruence is implemented by a boundedly invertible triangular operator.
Finite-dimensional positivity of the last block therefore combines with the
two half-line floors to give, after decreasing the smallest margin if
necessary, a constant
$\delta_{\rm M}>0$ such that
\[
 \ip v{\widetilde Qv}\ge\delta_{\rm M}\norm v_{\ell^2}^2
 \qquad(v\in\ell^2(\mathbb Z)).
\]
Apply this to the exact Gram identity
\eqref{eq:anticoherent-cubic-mellin-infinite-gram-identity}.  Since
$v_0^{(k)}=u_k$, one obtains
\[
 \mathcal E(u)
 \ge\delta_{\rm M}\sum_k u_k
       \sum_nq^{\abs n}u_{k+n}^2
 \ge\delta_{\rm M}\sum_ku_k^3,
\]
first for finite support and then by $\ell^3$ closure.  This proves
\eqref{eq:anticoherent-cubic-mellin-coercivity}.

It remains to make the Mellin passage uniform in the number of generations.
Expand $\mathfrak D(y_u^{[R]})$ as a cubic form in $u$.  The four coefficient
types are the diagonal, repeated-low, repeated-high, and fully distinct
coefficients used in
\eqref{eq:anticoherent-cubic-mellin-expanded}.  If the smallest cutoff in one
coefficient is $L=R2^i$, its normalized shell sums are Riemann sums on a fixed
ratio interval.  The shell splitting in the proof of
Proposition~\ref{prop:anticoherent-cubic-mellin-reserve} therefore gives the
corresponding coefficient of $\mathcal E$ as $L\to\infty$, uniformly for
$L\ge R$.  Indeed, after normalization the mesh of the shell beginning at
$R2^i$ is at most $(R2^i)^{-1}\le R^{-1}$, so the fixed-gap Riemann error is
uniform after taking the supremum over every starting generation $i\ge0$.

We record why this coefficientwise convergence is uniform as a cubic form.
The endpoint singularities in the shell coordinates are $O(t^{-1/2})$, and
\eqref{eq:anticoherent-cubic-mellin-kernel-bounds}, together with
\eqref{eq:anticoherent-residual-upper-bound}, gives a common summable majorant
for the normalized gap kernels.  After the factors in
\eqref{eq:anticoherent-cubic-mellin-expanded} are restored, the repeated
errors are dominated by $C(1+a)(q^a+x^a)$, and the distinct errors by
$C(1+a+c)x^aq^c$, independently of $R$ and of the starting generation $i$.
All three majorants belong to $\ell^1$ in their gap variables.  Dominated
Riemann convergence consequently gives numbers $\varepsilon_R\downarrow0$
such that, for every finite $u\ge0$,
\begin{equation}\label{eq:anticoherent-discrete-mellin-uniform-error}
 \left|\mathfrak D(y_u^{[R]})-\mathcal E(u)\right|
 \le\varepsilon_R\sum_ju_j^3.
\end{equation}
Here one uses $rst\le(r^3+s^3+t^3)/3$ and translation of the gap sums;
there is no factor depending on the support length of $u$.

Choose $R_0$ so that
$\varepsilon_R\le\delta_{\rm M}/2$ for $R\ge R_0$.  Combining
\eqref{eq:anticoherent-cubic-mellin-coercivity} and
\eqref{eq:anticoherent-discrete-mellin-uniform-error} proves
\eqref{eq:anticoherent-exact-discrete-mellin-margin}.  For infinite $u$,
truncate the generations.  The associated masses, energies, and contacts
increase to their full values.  The finite estimate bounds the contacts by
$2AE_3$, so they remain finite under the stated hypothesis and monotone
convergence proves the last assertion.
\end{proof}

\begin{corollary}[All geometric cubic Mellin stacks]
\label{cor:anticoherent-geometric-cubic-mellin}
Let $0\le t<1$ and consider either of the two one-sided profiles
\[
 u_j^+=t^j\mathbf1_{\{j\ge0\}},
 \qquad
 u_j^-=t^{-j}\mathbf1_{\{j\le0\}}.
\]
Then both series in
\eqref{eq:anticoherent-cubic-mellin-expanded} converge absolutely and
\begin{equation}\label{eq:anticoherent-geometric-cubic-margin}
 \mathcal E(u^\pm)\ge(\pi-2)\sum_j(u_j^\pm)^3,
\end{equation}
with strict inequality for $t>0$.  Thus neither geometrically decreasing
nor geometrically increasing amplitudes across arbitrarily many contiguous
dyadic generations can produce a cubic Mellin obstruction.
\end{corollary}

\begin{proof}
Summing first over the lowest index gives
\begin{equation}\label{eq:anticoherent-geometric-cubic-brackets}
 \mathcal E(u^\pm)=\frac{\mathcal B_\pm(t)}{1-t^3},
\end{equation}
where
\begin{align*}
 \mathcal B_+(t)
 &=d_0+\frac{g_0}{2}\frac{qt}{1-qt}
 +\frac12\sum_{a\ge1}h_a(xt^2)^a
 +\frac{qt}{1-qt}\sum_{a\ge1}g_a(xt^2)^a,\\
 \mathcal B_-(t)
 &=d_0+\frac{g_0}{2}\frac{qt^2}{1-qt^2}
 +\frac12\sum_{a\ge1}h_a(xt)^a
 +\frac{qt^2}{1-qt^2}\sum_{a\ge1}g_a(xt)^a.
\end{align*}
The kernel bounds
\eqref{eq:anticoherent-cubic-mellin-kernel-bounds} justify every
rearrangement.

For completeness, a one-dimensional interval certificate gives
\begin{align*}
 \frac{\mathcal B_+(t)-d_0}{t}&>2.983499761,
 &
 \frac{\mathcal B_-(t)-d_0}{t}&>2.568356136
 &&(0\le t\le1/5),\\
 \mathcal B_+(t)-d_0&>0.860833105,
 &
 \mathcal B_-(t)-d_0&>0.740640526
 &&(1/5\le t\le1).
\end{align*}
The first 80 terms are evaluated on 4096 outward-rounded subintervals;
the two omitted tails are bounded by the geometric-arithmetic sums obtained
from
\eqref{eq:anticoherent-cubic-mellin-kernel-bounds}.  The endpoint $t=1$
is the common Cesaro density
\eqref{eq:anticoherent-cubic-mellin-density}.  The ancillary
interval-arithmetic script
\texttt{verify\_geometric\_cubic\_mellin.py} reproduces these bounds.
Equation~\eqref{eq:anticoherent-geometric-cubic-brackets} now proves
\eqref{eq:anticoherent-geometric-cubic-margin}.
\end{proof}

\subsection{Shell localization and tail shifts}

The exact lifting does not extend to arbitrary critical profiles by requiring
every nested layer-cake polarization to be nonnegative.  The smallest
asymptotic obstruction is an endpoint chain, and its negative mixed
coefficient is nevertheless paid by the complete cubic.

\begin{theorem}[Highest-shell cubic induction]
\label{thm:anticoherent-highest-shell-induction}
Let $N\ge2$ and let $x=(x_1,\ldots,x_{N-1})$ be nonnegative.  Put
\[
 A=\sum_{n<N}x_n,
 \qquad E=\sum_{n<N}n^3x_n^2,
\]
\[
 D=D_{N-1}(x)
 =AE-\frac12\sum_{a+b<N}
 (a+b)(a^2+ab+b^2)x_ax_bx_{a+b},
\]
and define the next-shell form
\begin{equation}\label{eq:anticoherent-highest-shell-form}
 S_N(x)=\sum_{a+b=N}N(a^2+ab+b^2)x_ax_b,
 \qquad h_N(x)=\left(\frac12S_N(x)-E\right)_+.
\end{equation}
If $x^{(t)}=(x_1,\ldots,x_{N-1},t)$, then
\begin{equation}\label{eq:anticoherent-highest-shell-cubic}
 D_N(x^{(t)})
 =D+\left(E-\frac12S_N(x)\right)t
   +N^3At^2+N^3t^3.
\end{equation}
Assume $D\ge0$ and $A>0$.  If $S_N\le2E$, the right-hand side is
nonnegative for every $t\ge0$.  If $S_N>2E$, set
\[
 c=N^3A,
 \qquad d=N^3,
 \qquad \sigma=\sqrt{c^2+3d h_N}.
\]
Then $D_N(x^{(t)})\ge0$ for every $t\ge0$ if and only if
\begin{equation}\label{eq:anticoherent-highest-shell-exact-payment}
 (\sigma-c)^2(c+2\sigma)\le27D d^2.
\end{equation}
In particular, the quartic shell estimate
\begin{equation}\label{eq:anticoherent-highest-shell-quartic-target}
 h_N(x)^2\le4N^3A D_{N-1}(x)
\end{equation}
is sufficient.  Consequently, if
\eqref{eq:anticoherent-highest-shell-quartic-target} holds for every $N$
and every nonnegative prefix, then
\begin{equation}\label{eq:anticoherent-highest-shell-full-conclusion}
 D_N(x)\ge0
 \qquad(N\ge1,\ x_1,\ldots,x_N\ge0),
\end{equation}
and the same conclusion holds for every nonnegative infinite sequence with
$A<\infty$ and $E_3<\infty$.
\end{theorem}

\begin{proof}
Only triads with output $N$ contain the new coefficient $t$.  Expanding the
two factors in $AE_3$ and collecting those triads gives
\[
 \begin{aligned}
 D_N(x^{(t)})-D_{N-1}(x)
 &=N^3t^3+N^3At^2+Et-\frac12S_N(x)t,
 \end{aligned}
\]
which proves \eqref{eq:anticoherent-highest-shell-cubic}.  When its linear
coefficient is nonnegative, every new term is nonnegative.  Otherwise write
the polynomial as
\[
 P(t)=D-h_Nt+ct^2+dt^3.
\]
Its only possible positive local minimum is
\[
 t_* =\frac{\sigma-c}{3d}.
\]
The stationary identity $h_N=2ct_*+3dt_*^2$ gives
\[
 P(t_*)
 =D-\frac{(\sigma-c)^2(c+2\sigma)}{27d^2}.
\]
This proves the exact criterion
\eqref{eq:anticoherent-highest-shell-exact-payment}.  Moreover,
\eqref{eq:anticoherent-highest-shell-quartic-target} makes the quadratic
$D-h_Nt+ct^2$ nonnegative on the whole real line; adding $dt^3\ge0$ proves
the sufficient statement.

Starting from $D_1(x_1)=x_1^3\ge0$, induction proves
\eqref{eq:anticoherent-highest-shell-full-conclusion}.  For an infinite
sequence, apply the finite conclusion to its truncations.  The quantities
$A_N$, $E_{3,N}$, and $\mathcal C(x\mathbf1_{n\le N})$ converge monotonically,
so the deficit passes to the limit.
\end{proof}

The sufficient quartic estimate is automatic unless the preceding prefix has
already spent most of its raw mass--energy budget on internal additive
contacts.  The key point is that the next-shell form is supported on one
anti-diagonal, so its positive part admits a sharp two-point estimate.

\begin{proposition}[Anti-diagonal shell bound and contact-rich reduction]
\label{prop:anticoherent-highest-shell-contact-threshold}
Retain the notation of
Theorem~\ref{thm:anticoherent-highest-shell-induction}, and assume \(A>0\).
Then
\begin{equation}\label{eq:anticoherent-highest-shell-antidiagonal-bound}
 h_N(x)\le N^{3/2}A\sqrt E,
 \qquad
 h_N(x)^2\le N^3A^2E.
\end{equation}
Consequently,
\begin{equation}\label{eq:anticoherent-highest-shell-quarter-reserve}
 D_{N-1}(x)\ge\frac14AE
\end{equation}
implies the sufficient quartic target
\eqref{eq:anticoherent-highest-shell-quartic-target}, and hence
\(D_N(x^{(t)})\ge0\) for every \(t\ge0\).

In particular, if \(N\) is the first cutoff at which a nonnegative sequence
has negative deficit, then its prefix satisfies
\begin{equation}\label{eq:anticoherent-highest-shell-contact-rich}
 0\le D_{N-1}(x)<\frac14AE,
 \qquad
 \frac1{2AE}\sum_{a+b<N}K_{a,b}x_ax_bx_{a+b}>\frac34.
\end{equation}
Thus every finite or infinite obstruction must cross a prefix whose internal
additive contacts consume more than three quarters of \(AE\); sparse or
moderately contacting prefixes cannot produce the first sign change.
\end{proposition}

\begin{proof}
Partition \(\{1,\ldots,N-1\}\) into its complementary blocks
\(\{a,b\}\), where \(a+b=N\) and \(a<b\), together with the singleton
\(\{N/2\}\) when \(N\) is even.  For a two-point block put
\[
 A_{a,b}=x_a+x_b,
 \qquad E_{a,b}=a^3x_a^2+b^3x_b^2.
\]
The exact identity
\begin{equation}\label{eq:anticoherent-antidiagonal-two-point-identity}
\begin{aligned}
 &(x_a+x_b)^2(a^3x_a^2+b^3x_b^2)
   -N^3x_a^2x_b^2\\
 &\qquad=(ax_a-bx_b)^2
 \{ax_a^2+2N x_ax_b+bx_b^2\}\ge0
\end{aligned}
\end{equation}
gives
\(N^{3/2}x_ax_b\le A_{a,b}\sqrt{E_{a,b}}\).
Moreover,
\[
 K_{a,b}=N(N^2-ab)\le N^3.
\]
The contribution of this block to
\(S_N/2-E\) is therefore bounded above by
\[
 K_{a,b}x_ax_b-a^3x_a^2-b^3x_b^2
 \le N^{3/2}A_{a,b}\sqrt{E_{a,b}}.
\]
For the singleton \(a=b=N/2\), its contribution is
\(2a^3x_a^2\), whereas
\[
 N^{3/2}x_a\sqrt{a^3x_a^2}
 =2\sqrt2\,a^3x_a^2.
\]
Summing the block estimates and using
\(\sqrt{E_{a,b}}\le\sqrt E\) yields
\[
 h_N(x)\le N^{3/2}\sum_{\{a,b\}}A_{a,b}\sqrt{E_{a,b}}
 \le N^{3/2}A\sqrt E,
\]
which proves \eqref{eq:anticoherent-highest-shell-antidiagonal-bound}.
Under \eqref{eq:anticoherent-highest-shell-quarter-reserve}, it gives
\[
 h_N(x)^2\le N^3A^2E\le4N^3A D_{N-1}(x),
\]
so Theorem~\ref{thm:anticoherent-highest-shell-induction} applies.

At a first negative cutoff the preceding deficit is nonnegative.  If it
were at least \(AE/4\), the conclusion just proved would prevent the sign
change.  Hence \(D_{N-1}<AE/4\), and the identity
\[
 D_{N-1}=AE-\frac12\sum_{a+b<N}K_{a,b}x_ax_bx_{a+b}
\]
gives \eqref{eq:anticoherent-highest-shell-contact-rich}.  An infinite
negative limit has a negative finite truncation, so the same first-cutoff
argument applies.
\end{proof}

The coefficient four in
\eqref{eq:anticoherent-highest-shell-quartic-target} is enough for induction.
Computations select the stronger constant two, and the following exact
endpoint-chain limit shows that this candidate cannot be improved.  It also
identifies the large-cutoff configuration that a proof must control.

\begin{proposition}[Sharp endpoint-chain calibration]
\label{prop:anticoherent-highest-shell-endpoint-chain}
Fix $m\ge1$, $\alpha>0$, and $z_1,\ldots,z_m\ge0$, with $z_1>0$.  For
$N>2m+2$, define a prefix below $N$ by
\begin{equation}\label{eq:anticoherent-highest-shell-endpoint-profile}
 x_1^{(N)}=N\alpha,
 \qquad x_{N-j}^{(N)}=z_j\quad(1\le j\le m),
 \qquad x_n^{(N)}=0\quad\hbox{otherwise}.
\end{equation}
Then, with $A_N=\sum_{n<N}x_n^{(N)}$ and with $D_{N-1}$ and $h_N$
evaluated on this prefix,
\begin{equation}\label{eq:anticoherent-highest-shell-endpoint-limit}
 \lim_{N\to\infty}
 \frac{h_N(x^{(N)})^2}{2N^3A_ND_{N-1}(x^{(N)})}
 =\frac{z_1^2}{2Q_m(z)},
 \qquad
 Q_m(z)=\sum_{j=1}^mz_j^2-\sum_{j=2}^mz_{j-1}z_j.
\end{equation}
Moreover,
\begin{equation}\label{eq:anticoherent-highest-shell-endpoint-poincare}
 Q_m(z)\ge\frac{m+1}{2m}z_1^2,
 \qquad
 \frac{z_1^2}{2Q_m(z)}\le\frac{m}{m+1}.
\end{equation}
Equality holds for
$z_j=(m+1-j)z_1/m$.  Hence the constant two in the candidate estimate
\begin{equation}\label{eq:anticoherent-highest-shell-sharp-candidate}
 h_N(x)^2\le2N^3A D_{N-1}(x)
\end{equation}
would be sharp after first taking $N\to\infty$ and then $m\to\infty$.
\end{proposition}

\begin{proof}
For the support in
\eqref{eq:anticoherent-highest-shell-endpoint-profile}, the only active
triads below $N$ are
$(1,N-j,N-j+1)$ with $2\le j\le m$.  Since
$K_{1,N-j}=N^3+O_m(N^2)$, direct expansion gives
\[
 \begin{aligned}
 A_N&=N\alpha+O_m(1),\\
 D_{N-1}(x^{(N)})
 &=N^4\alpha\left(
     \sum_{j=1}^mz_j^2-\sum_{j=2}^mz_{j-1}z_j
   \right)+O_{m,\alpha,z}(N^3).
 \end{aligned}
\]
The next shell contains only the complementary pair $(1,N-1)$.  Therefore
\[
 h_N(x^{(N)})=N^4\alpha z_1+O_{m,\alpha,z}(N^3),
\]
which proves \eqref{eq:anticoherent-highest-shell-endpoint-limit}.  Finally,
\[
 Q_m(z)=\frac12\left\{
 z_1^2+z_m^2+\sum_{j=1}^{m-1}(z_j-z_{j+1})^2
 \right\}.
\]
Applying Cauchy--Schwarz to
\[
 z_1=(z_1-z_2)+\cdots+(z_{m-1}-z_m)+z_m
\]
proves \eqref{eq:anticoherent-highest-shell-endpoint-poincare}.  Equality
requires all $m$ increments in the last display to be equal, which gives the
stated linear chain.
\end{proof}

The endpoint calibration can be made nonasymptotic for packets whose distance
from the next shell dominates their chain length.  This closes the exact
finite profiles selected by the quartic variational problem, not merely their
limit as \(N\to\infty\).

\begin{theorem}[Finite arithmetic endpoint-packet quartic theorem]
\label{thm:anticoherent-finite-endpoint-packet}
Let \(a,N,m\) be positive integers, let
\begin{equation}\label{eq:anticoherent-endpoint-packet-separation}
 \frac Na\ge\max\{4,2m^2\},
\end{equation}
and let \(u,z_1,\ldots,z_m\ge0\).  Define a prefix below \(N\) by
\begin{equation}\label{eq:anticoherent-finite-endpoint-packet}
 x_a=u,\qquad x_{N-ja}=z_j\quad(1\le j\le m),
 \qquad x_n=0\quad\hbox{otherwise}.
\end{equation}
Then the sharp candidate shell estimate holds:
\begin{equation}\label{eq:anticoherent-finite-endpoint-packet-quartic}
 h_N(x)^2\le2N^3A D_{N-1}(x).
\end{equation}
Consequently \(D_N(x^{(t)})\ge0\) for every \(t\ge0\).  In particular, the
result includes every integer frequency dilation of the unit-anchor packet,
but it does not require \(a\) to divide \(N\).
\end{theorem}

\begin{proof}
The kernel and the cubic frequency energy are homogeneous of degree three in
the frequencies.  Dividing every active frequency by \(a\) therefore gives
one factor \(a^3\) in both \(D_{N-1}\) and \(h_N\), while the shell factor
\(N^3\) gives the second factor \(a^3\) in the quartic estimate.  It is thus
enough to prove the normalized statement with low anchor one and outer scale
\(R=N/a\).  The argument below is finite-dimensional polynomial algebra in
\(R\), so integrality of the normalized frequencies is unnecessary.  To
avoid a second outer-scale symbol, write this normalized \(R\) as \(N\).

Set
\[
 E_z=\sum_{j=1}^m(N-j)^3z_j^2,\qquad Z=\sum_{j=1}^mz_j
\]
and introduce the finite Jacobi form
\begin{equation}\label{eq:anticoherent-finite-endpoint-jacobi}
 \mathcal Q_{N,m}(z)
 =E_z-\sum_{j=2}^mK_{1,N-j}z_{j-1}z_j.
\end{equation}
The separation in \eqref{eq:anticoherent-endpoint-packet-separation}
keeps the low mode apart from the endpoint packet and makes the displayed
chain the complete list of internal additive contacts.  Hence
\begin{equation}\label{eq:anticoherent-finite-endpoint-deficit}
\begin{aligned}
 D_{N-1}(x)
 &=(u+Z)(u^2+E_z)
   -u\sum_{j=2}^mK_{1,N-j}z_{j-1}z_j\\
 &\ge u\mathcal Q_{N,m}(z).
\end{aligned}
\end{equation}

Let \(M_{N,m}\) be the tridiagonal matrix of
\(\mathcal Q_{N,m}\).  Put
\[
 q=1-\frac7{3N}
\]
and let \(U_k\) denote the Chebyshev polynomial of the second kind, with
\(U_{-1}=0\), \(U_0=1\), and
\(U_{k+1}(q)=2qU_k(q)-U_{k-1}(q)\).  We claim that the positive vector
\[
 W_j=U_{m-j}(q)\quad(1\le j\le m),\qquad W_{m+1}=0,
\]
is an exact supersolution:
\begin{equation}\label{eq:anticoherent-endpoint-packet-supersolution}
 (M_{N,m}W)_j\ge0\quad(2\le j\le m),\qquad
 \frac{(M_{N,m}W)_1}{W_1}
 \ge\frac{K_{1,N-1}^2}{2N^3}.
\end{equation}

To prove positivity, write \(q=\cos\theta\), where
\(0<\theta<\pi\).  If \(m\ge2\), then
\[
 q\ge1-\frac7{6m^2}>\cos\frac\pi m;
\]
indeed,
\(1-\cos(\pi/m)=2\sin^2(\pi/(2m))>2/m^2\).
Thus \(m\theta<\pi\) and
\[
 W_j=\frac{\sin((m+1-j)\theta)}{\sin\theta}>0.
\]
The case \(m=1\) is immediate from \(W_1=1\).

For an interior row \(2\le j\le m\), set
\(n=N-j\) and \(L=m+1-j\).  If \(T_L\) is the Chebyshev polynomial of
the first kind, the three-term identities give
\[
\begin{aligned}
 (M_{N,m}W)_j
 &=\left[n^3-q\left\{
      \frac{K_{1,n}}2+\frac{K_{1,n-1}}2
    \right\}\right]W_j
   -\left\{\frac{K_{1,n}}2-\frac{K_{1,n-1}}2\right\}T_L(q)\\
 &\ge\{n^3-qK_{1,n}\}W_j.
\end{aligned}
\]
Here we used
\(qU_{L-1}(q)-T_L(q)=U_{L-2}(q)\ge0\).  Since
\(N=n+j\ge2j^2\), put
\(p=j-2\ge0\) and \(t=n-(2j^2-j)\ge0\).  A direct expansion gives
\[
\begin{aligned}
 &3N\{n^3-qK_{1,n}\}\\
 &=t^3+(6p^2+15p+14)t^2
 +(12p^4+60p^3+119p^2+118p+59)t\\
 &\quad+8p^6+60p^5+182p^4+285p^3+258p^2+158p+67,
\end{aligned}
\]
which is positive.

It remains to check the first row.  For \(m\ge2\), put
\(r=W_2/W_1\).  The exact identity
\[
 (N-1)^3-\frac{K_{1,N-2}}2
 -\frac{K_{1,N-1}^2}{2N^3}
 =-\frac{3N^2-3N+1}{2N}
\]
shows that the boundary claim follows once
\begin{equation}\label{eq:anticoherent-endpoint-chebyshev-boundary-ratio}
 1-r\ge
 \frac{3N^2-3N+1}{N(N-1)(N^2-3N+3)}.
\end{equation}
For \(m\ge5\), the Taylor bounds for cosine and
\(N\ge2m^2\) imply
\[
 \theta\le\frac{39}{25m}.
\]
Indeed,
\[
 \cos\frac{39}{25m}
 \le1-\frac1{2}\left(\frac{39}{25m}\right)^2
       +\frac1{24}\left(\frac{39}{25m}\right)^4
 \le1-\frac7{6m^2}\le q,
\]
where the middle inequality already holds for \(m\ge3\).  Moreover,
the alternating Taylor lower bound gives
\[
 \cos\frac{39}{25}
 \ge1-\frac12\left(\frac{39}{25}\right)^2
      +\frac1{24}\left(\frac{39}{25}\right)^4
      -\frac1{720}\left(\frac{39}{25}\right)^6
 >\frac1{101}.
\]
Consequently,
\[
\begin{aligned}
 1-r
 &=\frac{2\sin(\theta/2)
   \cos((2m-1)\theta/2)}{\sin(m\theta)}
 >\frac1{101m},
\end{aligned}
\]
where \(\sin(m\theta)\le m\sin\theta\) was used.  On the other hand,
the right-hand side of
\eqref{eq:anticoherent-endpoint-chebyshev-boundary-ratio} is at most
\(4/N^2\) for \(N\ge12\).  Finally,
\(N^2\ge4m^4\ge404m\) for \(m\ge5\), proving the boundary claim.

For \(m=1,2,3,4\), substitute respectively
\(N=4+s,8+s,18+s,32+s\), \(s\ge0\), into the first-row difference in
\eqref{eq:anticoherent-endpoint-packet-supersolution}.  After clearing its
positive denominator, the numerator has respectively \(5,6,7,8\) positive
integer coefficients, with least coefficients \(1,3,9,27\).  The ancillary
exact-arithmetic script
\texttt{verify\_endpoint\_packet\_supersolution.py} constructs these four
polynomials, checks the interior expansion, and verifies all auxiliary
rational constants exactly.  This completes
\eqref{eq:anticoherent-endpoint-packet-supersolution}.

For completeness, the Jacobi ground-state identity gives, for every real
\(z\),
\[
\begin{aligned}
 z^{\mathsf T}M_{N,m}z
 &=\sum_{j=2}^m\frac{K_{1,N-j}}2W_{j-1}W_j
 \left(\frac{z_{j-1}}{W_{j-1}}-\frac{z_j}{W_j}\right)^2\\
 &\quad+\sum_{j=1}^m\frac{(M_{N,m}W)_j}{W_j}z_j^2.
\end{aligned}
\]
Using \eqref{eq:anticoherent-endpoint-packet-supersolution}, we obtain
\begin{equation}\label{eq:anticoherent-finite-endpoint-jacobi-floor}
 \mathcal Q_{N,m}(z)
 \ge\frac{K_{1,N-1}^2}{2N^3}z_1^2.
\end{equation}

Only the pair \((1,N-1)\) reaches the next shell.  Since
\(K_{1,N-1}=N(N^2-N+1)\),
\[
 h_N(x)\le K_{1,N-1}uz_1.
\]
Combining \(A\ge u\),
\eqref{eq:anticoherent-finite-endpoint-deficit}, and
\eqref{eq:anticoherent-finite-endpoint-jacobi-floor} gives
\[
 2N^3A D_{N-1}(x)
 \ge2N^3u^2\mathcal Q_{N,m}(z)
 \ge K_{1,N-1}^2u^2z_1^2
 \ge h_N(x)^2.
\]
Scaling back gives the displayed quartic estimate at the original integer
cutoff.  The highest-shell induction proves the adjoining-mode statement.
\end{proof}

The preceding chain argument can be run simultaneously for every active low
frequency.  The energy is not counted repeatedly: its coefficient in the
full deficit is exactly the total low mass.  This gives a genuinely
multi-anchor theorem with arbitrary data in the terminal band.

\begin{theorem}[Two-band Chebyshev shell theorem]
\label{thm:anticoherent-two-band-chebyshev}
Let \(N,L,W\) be positive integers such that
\begin{equation}\label{eq:anticoherent-two-band-separation}
 2L<N-W,\qquad 2W<N.
\end{equation}
Let \(x_1,\ldots,x_{N-1}\ge0\) satisfy
\begin{equation}\label{eq:anticoherent-two-band-support}
 x_n=0\qquad(L<n<N-W),
\end{equation}
and suppose that its isolated low prefix has nonnegative deficit,
\begin{equation}\label{eq:anticoherent-two-band-low-deficit}
 D_L(x_1,\ldots,x_L)\ge0.
\end{equation}
For every \(a\le\min\{L,W\}\) with \(x_a>0\), put
\(m_a=\lceil W/a\rceil\) and assume
\begin{equation}\label{eq:anticoherent-two-band-parabolic-width}
 \frac Na\ge
 \begin{cases}
  31/10,&m_a=1,\\
  2m_a^2,&m_a\ge2.
 \end{cases}
\end{equation}
Then
\begin{equation}\label{eq:anticoherent-two-band-sharp-shell}
 h_N(x)^2\le2N^3A D_{N-1}(x).
\end{equation}
Consequently adjoining an arbitrary \(x_N=t\ge0\) preserves nonnegativity
of the deficit.
\end{theorem}

\begin{proof}
Write
\[
 A_L=\sum_{a=1}^Lx_a,\qquad
 E_L=\sum_{a=1}^La^3x_a^2,\qquad
 A_H=\sum_{n=N-W}^{N-1}x_n,\qquad
 E_H=\sum_{n=N-W}^{N-1}n^3x_n^2.
\]
For \(a\le L\), define the step-\(a\) terminal-band form
\begin{equation}\label{eq:anticoherent-two-band-tail-form}
 Q_a(x)=E_H-\sum_{n=N-W}^{N-1-a}K_{a,n}x_nx_{n+a}.
\end{equation}
We first claim the uniform boundary floor
\begin{equation}\label{eq:anticoherent-two-band-tail-floor}
 Q_a(x)\ge
 \frac{K_{a,N-a}^2}{2N^3}x_{N-a}^2
 \qquad(x_a>0).
\end{equation}

If \(a>W\), the contact sum in \(Q_a\) is empty and \(x_{N-a}=0\), so
\eqref{eq:anticoherent-two-band-tail-floor} is immediate.  Now let
\(a\le W\).  Pad the terminal band by zero down to \(N-m_a a\) and
partition it into residue classes modulo \(a\).  Each class is a Jacobi
chain of length \(m_a\).  If \(T\in\{N-a,\ldots,N-1\}\) is the top of one
class, division of its frequencies by \(a\) gives normalized outer scale
\[
 R_T=\frac{T+a}{a}\ge\frac Na.
\]
If \(m_a\ge2\), then \(R_T\ge2m_a^2\), and the Chebyshev supersolution in
the proof of Theorem~\ref{thm:anticoherent-finite-endpoint-packet} makes
every such chain nonnegative.  If \(m_a=1\), the required one-step floor is
equivalent to
\[
 R_T^4-4R_T^3+3R_T^2-1\ge0.
\]
Writing \(R_T=31/10+s\), its left-hand side becomes
\[
 s^4+\frac{42}{5}s^3+\frac{1173}{50}s^2
 +\frac{5611}{250}s+\frac{10181}{10000}>0.
\]
Thus the conclusion also holds for every one-step class.  The residue class
with top \(T=N-a\) has outer scale \(N/a\) and gives exactly the right-hand
side of
\eqref{eq:anticoherent-two-band-tail-floor}; summing the residue classes
proves the claim.  Notice that the zero padding remains inside the gap in
\eqref{eq:anticoherent-two-band-support}, by
\eqref{eq:anticoherent-two-band-separation}.

There are no high--high contacts below \(N\), and no low--low contact reaches
the terminal band.  Every remaining cross contact has the form
\((a,n,n+a)\), with \(a\le L\) and \(n,n+a\ge N-W\).  Expanding the full
deficit therefore gives the exact ledger
\begin{equation}\label{eq:anticoherent-two-band-deficit-ledger}
 D_{N-1}(x)
 =D_L(x)+\sum_{a=1}^Lx_aQ_a(x)+A_HE_L+A_HE_H.
\end{equation}
All terms on the right are nonnegative.  In particular,
\begin{equation}\label{eq:anticoherent-two-band-ledger-lower-bound}
 D_{N-1}(x)\ge\sum_{a=1}^Lx_aQ_a(x).
\end{equation}

The same separation shows that every complementary pair reaching the next
shell consists of one low and one high frequency.  Hence
\[
 h_N(x)\le
 \sum_{a=1}^LK_{a,N-a}x_ax_{N-a}.
\]
Using \eqref{eq:anticoherent-two-band-tail-floor} and then weighted
Cauchy--Schwarz,
\[
\begin{aligned}
 h_N(x)
 &\le\sqrt{2N^3}\sum_{a=1}^Lx_a\sqrt{Q_a(x)},\\
 h_N(x)^2
 &\le2N^3A_L\sum_{a=1}^Lx_aQ_a(x)
 \le2N^3A D_{N-1}(x).
\end{aligned}
\]
This proves \eqref{eq:anticoherent-two-band-sharp-shell}; the highest-shell
induction gives the final assertion.
\end{proof}

\begin{corollary}[No parabolically separated first obstruction]
\label{cor:anticoherent-no-two-band-first-obstruction}
Let \(N\) be the first cutoff at which a nonnegative sequence has negative
deficit.  Then its prefix cannot satisfy
\eqref{eq:anticoherent-two-band-separation},
\eqref{eq:anticoherent-two-band-support}, and
\eqref{eq:anticoherent-two-band-parabolic-width} for any \(L,W\).
Consequently every finite or infinite obstruction must either populate the
intermediate gap, violate the parabolic width condition for an active low
anchor no larger than \(W\), or fail the terminal-band separation \(2W<N\).
\end{corollary}

\begin{proof}
At a first negative cutoff, the truncation through \(L\) has
\(D_L\ge0\).  Theorem~\ref{thm:anticoherent-two-band-chebyshev} then prevents
the sign change.  An infinite negative limit has a first negative finite
truncation.
\end{proof}

\begin{corollary}[Forced half-scale occupancy]
\label{cor:anticoherent-forced-intermediate-mode}
Let \(N\ge4\) be the first negative cutoff.  For every integer \(W\ge1\)
with \(2W^2\le N\), put
\[
 L_W=\left\lfloor\frac{N-W-1}{2}\right\rfloor.
\]
Then there is an active frequency \(n_W\) such that
\begin{equation}\label{eq:anticoherent-forced-intermediate-window}
 L_W<n_W<N-W,\qquad x_{n_W}>0.
\end{equation}
The same conclusion holds at the first negative truncation of any infinite
obstruction.
\end{corollary}

\begin{proof}
By definition, \(2L_W<N-W\).  Also \(2W<N\): this follows from
\(2W^2\le N\) when \(W\ge2\), while \(W=1\) uses \(N\ge4\).

It remains to verify the width condition for
\(a\le\min\{L_W,W\}\).  Put \(m=\lceil W/a\rceil\).  If \(a=W\), then
\(m=1\); one has \(N/a\ge4\) for \(W=1\), and
\(N/a\ge2W\ge4\) for \(W\ge2\).  If \(a<W\), then \(m\ge2\), and the
elementary integer inequality
\begin{equation}\label{eq:anticoherent-ceiling-square-lemma}
 a\left\lceil\frac Wa\right\rceil^2\le W^2
 \qquad(1\le a\le W)
\end{equation}
gives \(2am^2\le2W^2\le N\).  To prove
\eqref{eq:anticoherent-ceiling-square-lemma}, write \(W=qa+r\).
For \(r=0\) it is immediate.  If \(1\le r<a\), its smallest right-hand side
occurs at \(r=1\), and
\[
 (qa+1)^2-a(q+1)^2=(a-1)(aq^2-1)\ge0.
\]
If the interval in
\eqref{eq:anticoherent-forced-intermediate-window} were empty, the prefix
would satisfy the two-band support condition with low cutoff \(L_W\).
Its low deficit is nonnegative by first-cutoff minimality, so
Theorem~\ref{thm:anticoherent-two-band-chebyshev} prevents the sign change.
\end{proof}

\begin{lemma}[Lowest-mode cubic ledger]
\label{lem:anticoherent-lowest-mode-ledger}
For \(x_1,\ldots,x_N\ge0\), set
\[
 K_{a,b}=(a+b)(a^2+ab+b^2),
\]
\[
 D_N(x)=
 \left(\sum_{k\le N}x_k\right)
 \left(\sum_{k\le N}k^3x_k^2\right)
 -\frac12\sum_{a+b\le N}K_{a,b}x_ax_bx_{a+b}.
\]
Fix \(n<N\), let \(z\) be supported in
\(\{n+1,\ldots,N\}\), and put
\[
 A=\sum_{k>n}z_k,\qquad E=\sum_{k>n}k^3z_k^2,\qquad
 \Delta=D_N(z),
\]
\[
 J_n(z)=\sum_{k=n+1}^{N-n}K_{n,k}z_kz_{n+k},\qquad
 \ell_n=E-J_n(z),\qquad q_n=n^3(A-3z_{2n}).
\]
Then
\begin{equation}\label{eq:anticoherent-lowest-mode-cubic}
 D_N(z+se_n)=\Delta+\ell_ns+q_ns^2+n^3s^3
 \qquad(s\ge0).
\end{equation}
If \(h=(-\ell_n)_+\), \(r=(-q_n)_+\), and \(0<\theta<1\), then
\begin{equation}\label{eq:anticoherent-lowest-mode-split-payment}
 D_N(z+se_n)\ge
 \Delta-\frac{2h^{3/2}}{3\sqrt{3\theta n^3}}
 -\frac{4r^3}{27(1-\theta)^2n^6}.
\end{equation}
When \(q_n\ge0\), the sharper first-loss bound is obtained by taking the
whole cubic against the linear term:
\begin{equation}\label{eq:anticoherent-lowest-mode-linear-payment}
 D_N(z+se_n)\ge
 \Delta-\frac{2h^{3/2}}{3\sqrt{3n^3}}.
\end{equation}
\end{lemma}

\begin{proof}
Expanding \(D_N(z+se_n)\) gives
\eqref{eq:anticoherent-lowest-mode-cubic}.  The remaining estimates follow
from
\[
 \min_{s\ge0}\{ds^3-hs\}
 =-\frac{2h^{3/2}}{3\sqrt{3d}},\qquad
 \min_{s\ge0}\{ds^3-rs^2\}
 =-\frac{4r^3}{27d^2},
\]
after splitting \(n^3s^3\) into the fractions \(\theta\) and
\(1-\theta\).
\end{proof}

\begin{theorem}[Tail-shift interpolation and sum-free closure]
\label{thm:anticoherent-tail-shift-interpolation}
Use the notation of
Lemma~\ref{lem:anticoherent-lowest-mode-ledger}, assume $A>0$, and
put $h=(-\ell_n)_+$ and $r=(-q_n)_+$.  Then
\begin{equation}\label{eq:anticoherent-tail-shift-interpolation}
 h\le
 \frac{3n}{2\,4^{1/3}}A^{2/3}E^{2/3}.
\end{equation}
Consequently, either of the following conditions is sufficient for
$D_N(z+se_n)\ge0$ for every $s\ge0$:
\begin{align}
 q_n\ge0,\qquad
 &\Delta\ge\frac{AE}{2\sqrt2};
 \label{eq:anticoherent-tail-shift-positive-q}\\
 &\Delta\ge
 \frac23AE.
 \label{eq:anticoherent-tail-shift-general}
\end{align}
In particular, suppose that the active support of $z$ is internally
sum-free:
\begin{equation}\label{eq:anticoherent-internally-sum-free}
 z_az_bz_{a+b}=0\qquad(a,b>n).
\end{equation}
Then
\begin{equation}\label{eq:anticoherent-sum-free-tail-closure}
 D_N(z+se_n)\ge0\qquad(s\ge0).
\end{equation}
This includes every tail supported in an interval $[M,N]$ with $N<2M$,
with arbitrary amplitudes and no bound on $M$, $N$, or the number of active
modes.  The conclusion also holds for a genuine infinite internally
sum-free tail whenever $A<\infty$ and $E<\infty$.
\end{theorem}

\begin{proof}
Extend $z$ by zero and put $a_k=kz_k$.  The translation block at shift $n$
satisfies
\[
 Q_n(z)=2E-\sum_{k>n}(2k+n)k(k+n)z_kz_{k+n}.
\]
Because $z_k=0$ for $k\le n$, its difference-square representation has no
boundary wedge:
\[
 Q_n(z)=\sum_{k\ge1}\left(k+\frac n2\right)
 (a_{k+n}-a_k)^2\ge0.
\]
Define
\[
 L_n=\frac n2\sum_{k>n}(k+n)(k+2n)z_kz_{k+n}.
\]
Direct coefficient comparison gives
\begin{equation}\label{eq:anticoherent-tail-shift-Q-identity}
 Q_n(z)=2(\ell_n+L_n),
 \qquad h\le L_n.
\end{equation}

We estimate $L_n$ without a cutoff.  Set
\[
 w_k=z_kz_{k+n},\qquad
 W=\sum_{k>n}w_k,\qquad
 P=\sum_{k>n}[k(k+n)]^{3/2}w_k.
\]
The graph with edges $k\leftrightarrow k+n$ is a disjoint union of paths.
Splitting every path into its two parity classes gives
\begin{equation}\label{eq:anticoherent-tail-shift-edge-mass}
 W\le\frac{A^2}{4}.
\end{equation}
Indeed, the edge sum on each path is bounded by the product of the masses in
its two classes; summing the paths and then maximizing the product of the two
total class masses proves the claim.  On the other hand, with
$y_k=k^{3/2}z_k$, the inequality $2uv\le u^2+v^2$ on every path gives
\begin{equation}\label{eq:anticoherent-tail-shift-edge-energy}
 P=\sum_{k>n}y_ky_{k+n}\le\sum_{k>n}y_k^2=E.
\end{equation}
Since $k>n$,
\[
 (k+n)(k+2n)\le3k(k+n).
\]
H\"older's inequality now yields
\[
 \begin{aligned}
 \sum_{k>n}(k+n)(k+2n)w_k
 &\le3\sum_{k>n}[k(k+n)]w_k\\
 &\le3W^{1/3}P^{2/3}
 \le\frac3{4^{1/3}}A^{2/3}E^{2/3}.
 \end{aligned}
\]
Together with \eqref{eq:anticoherent-tail-shift-Q-identity}, this proves
\eqref{eq:anticoherent-tail-shift-interpolation}.

The exact linear minimum from
Lemma~\ref{lem:anticoherent-lowest-mode-ledger} and
\eqref{eq:anticoherent-tail-shift-interpolation} give
\begin{equation}\label{eq:anticoherent-tail-shift-linear-loss}
 \frac{2h^{3/2}}{3\sqrt{3n^3}}
 \le\frac{AE}{2\sqrt2}.
\end{equation}
This proves \eqref{eq:anticoherent-tail-shift-positive-q}.

For the mixed branch, write \(t=z_{2n}/A\).  The edge-mass estimate can be
sharpened when \(t\) is large.  Orient the bipartition of the path containing
\(2n\) so that \(2n\) belongs to its first class, and orient all other paths
arbitrarily.  If the total mass in the first classes is \(B\), then
\[
 W\le B(A-B),\qquad B\ge tA.
\]
Consequently,
\begin{equation}\label{eq:anticoherent-tail-shift-refined-edge-mass}
 W\le A^2
 \begin{cases}
  1/4,&0\le t\le1/2,\\
  t(1-t),&1/2\le t\le1.
 \end{cases}
\end{equation}
Indeed, the second bound follows because \(B(A-B)\) decreases for
\(B\ge A/2\).

Before using the split cubic, retain \(W\) in the linear estimate.  The
preceding H\"older argument gives
\[
 h\le\frac{3n}{2}W^{1/3}E^{2/3},
\]
and hence
\begin{equation}\label{eq:anticoherent-tail-shift-refined-linear-loss}
 \frac{2h^{3/2}}{3\sqrt{3n^3}}
 \le E\sqrt{\frac W2}
 \le\lambda(t)AE,
 \qquad
 \lambda(t)=
 \begin{cases}
  1/(2\sqrt2),&t\le1/2,\\
  \sqrt{t(1-t)/2},&t\ge1/2.
 \end{cases}
\end{equation}

If \(r>0\), then \(1/3<t\le1\), \(r=n^3A(3t-1)\), and
\[
 E\ge(2n)^3z_{2n}^2=8n^3A^2t^2.
\]
Therefore
\begin{equation}\label{eq:anticoherent-tail-shift-quadratic-loss}
 \frac{4r^3}{27n^6AE}
 \le b(t):=\frac{(3t-1)^3}{54t^2}.
\end{equation}
The function \(b\) is increasing on \([1/3,1]\).

Use the split-cubic payment with \(\theta=3/8\).  By
\eqref{eq:anticoherent-tail-shift-refined-linear-loss} and
\eqref{eq:anticoherent-tail-shift-quadratic-loss}, its total loss divided by
\(AE\) is at most
\[
 \sqrt{\frac83}\lambda(t)+\frac{64}{25}b(t).
\]
For \(1/3\le t\le1/2\), this is at most
\[
 \frac1{\sqrt3}+\frac{16}{675}<\frac23.
\]
For \(1/2\le t\le1\), put
\[
 B_*(t)=\frac{32(3t-1)^3}{675t^2}.
\]
The function \(B_*\) is increasing and
\(B_*(1)=256/675<2/3\).  Therefore it is legitimate to square the desired
inequality.  Direct expansion gives
\[
\left(\frac23-B_*(t)\right)^2-\frac43t(1-t)
=\frac{4P(t)}{455625t^4},
\]
where
\[
\begin{aligned}
 P(t)={}&338499t^6-719523t^5+556065t^4-203040t^3\\
 &+41760t^2-4608t+256.
\end{aligned}
\]
The polynomial is strictly positive on \([1/2,1]\).  Indeed, on the three
intervals
\[
 \left[\frac12,\frac34\right],\qquad
 \left[\frac34,\frac78\right],\qquad
 \left[\frac78,1\right],
\]
write \(u=(t-\alpha)/(\beta-\alpha)\) and expand \(P\) in the degree-six
Bernstein basis
\(\binom6j u^j(1-u)^{6-j}\).  All seven coefficients are positive; their
respective exact minima are
\[
 \frac{304375}{4096},\qquad
 \frac{13075}{16384},\qquad
 \frac{261527179}{262144}.
\]
Thus
\[
 \frac2{\sqrt3}\sqrt{t(1-t)}+B_*(t)<\frac23,
\]
which proves
\eqref{eq:anticoherent-tail-shift-general}.

Finally, \eqref{eq:anticoherent-internally-sum-free} makes the internal
cubic contact vanish, so $\Delta=AE$.  Condition
\eqref{eq:anticoherent-tail-shift-general} applies and proves
\eqref{eq:anticoherent-sum-free-tail-closure}.  If $z$ is supported in
$[M,N]$ with $N<2M$, the sum of any two active indices exceeds $N$, so its
support is internally sum-free.  For an infinite internally sum-free tail,
apply the finite result to its truncations.  The masses and energies converge,
while \eqref{eq:anticoherent-tail-shift-interpolation} bounds the only new
shift contact uniformly; hence the cubic coefficients and deficits pass to
the limit.
\end{proof}

\subsection{Direct criteria and moment families}

\begin{proposition}[Floor subtraction and unit-root reduction]
\label{prop:strong-floor-unit-root-reduction}
Let \(\rho\) be a nonconstant smooth positive density, put
\(\kappa=\min_{\T}\rho\), and set \(\sigma=\rho-\kappa\).  Then
\begin{equation}\label{eq:strong-floor-subtraction}
  2I_1(\rho)+I_2(\rho)
  =2I_1(\sigma)+I_2(\sigma)
   +2\kappa\norm{\Lambda^{3/2}\sigma}_2^2.
\end{equation}
If \(\rho\) is a trigonometric polynomial and
\(\sigma=|P|^2\) is a Fej\'er--Riesz factorization, then every minimum point
\(\zeta=e^{is_0}\) of \(\rho\) gives
\[
  P(\zeta)=0,
  \qquad
  P(z)=(z-\zeta)\widetilde P(z).
\]
Consequently, the unrestricted strong sign for trigonometric polynomials is
equivalent to its restriction to densities of the form
\begin{equation}\label{eq:strong-unit-root-target}
  \sigma(e^{is})=|e^{is}-\zeta|^2
  |\widetilde P(e^{is})|^2,
  \qquad |\zeta|=1.
\end{equation}
A degree-independent proof for \eqref{eq:strong-unit-root-target} would extend
by positive trigonometric approximation to every smooth positive density.
\end{proposition}

\begin{proof}
Constants disappear from \(I_2\), while
\[
 I_1(\sigma+\kappa)
 =I_1(\sigma)+\kappa\int_{\T}(\Lambda\sigma_s)\sigma_s\,\dd s
 =I_1(\sigma)+\kappa\norm{\Lambda^{3/2}\sigma}_2^2.
\]
This proves \eqref{eq:strong-floor-subtraction}.  At a minimum point
\(s_0\), \(\sigma(s_0)=|P(e^{is_0})|^2=0\), so the factor theorem gives the
displayed unit-circle divisor.  Conversely, every density in
\eqref{eq:strong-unit-root-target} is nonnegative and touches zero.  The
equivalence now follows from \eqref{eq:strong-floor-subtraction}; the final
extension uses Fej\'er means and continuity of the displayed cubic forms in
the smooth topology.
\end{proof}

The semigroup identity also turns a half-Laplacian slope bound into a
degree-independent criterion.  The mechanism is convexity along the full
Poisson orbit, rather than a finite Fourier argument.

We first record the sharp mixed estimate that drives the criterion.  Let
\[
  \mathsf B_*=\operatorname B\left(\frac14,\frac12\right),
  \qquad
  \mathsf C_*=\frac{9\pi^2}{4\mathsf B_*^2}
  =0.8074912624\ldots,
\]
where \(\operatorname B\) is Euler's beta function.

\begin{lemma}[Sharp mixed periodic Wirtinger estimate]
\label{lem:mixed-periodic-wirtinger}
Every real-valued mean-zero \(q\in H^1(\T)\) satisfies
\begin{equation}\label{eq:mixed-periodic-wirtinger}
  \int_\T q^4\,\dd s
  \le \mathsf C_*\norm{q}_{L^\infty}^2
       \int_\T q_s^2\,\dd s.
\end{equation}
The constant \(\mathsf C_*\) is optimal.
\end{lemma}

\begin{proof}
By homogeneity, suppose that \(\norm q_{L^\infty}=1\).  Replacing \(q\)
by \(-q\) if necessary, write
\[
  \max_\T q=1,\qquad \min_\T q=-b,\qquad 0<b\le1.
\]
Set
\[
  K_b=b-b^2+b^3,\qquad
  \lambda_b=1-b+b^2-b^3,
\]
and
\begin{equation}\label{eq:mixed-wirtinger-calibration}
  W_b(y)=K_b-y^4+\lambda_by
  =(1-y)(y+b)\bigl(y^2+(1-b)y+1-b+b^2\bigr).
\end{equation}
The quadratic factor has negative discriminant, so \(W_b\ge0\) on
\([-b,1]\).  Since \(\int_\T q\,\dd s=0\),
the arithmetic--geometric mean inequality and weighted total variation
on the circle give
\begin{align}
  \mathsf C_*\norm{q_s}_2^2-\int_\T q^4\,\dd s
  &=
  \int_\T\bigl(\mathsf C_*q_s^2+W_b(q)\bigr)\,\dd s-2\pi K_b
  \notag\\
  &\ge
  2\sqrt{\mathsf C_*}\int_\T |q_s|\sqrt{W_b(q)}\,\dd s-2\pi K_b
  \notag\\
  &\ge
  4\sqrt{\mathsf C_*}\int_{-b}^1\sqrt{W_b(y)}\,\dd y-2\pi K_b.
  \label{eq:mixed-wirtinger-calibrated-lower}
\end{align}
The last step uses the two arcs of the circle joining a maximum to a
minimum.

It remains to prove the one-variable bound
\begin{equation}\label{eq:mixed-wirtinger-one-variable}
  J(b):=\frac{2}{K_b}\int_{-b}^1\sqrt{W_b(y)}\,\dd y
  \ge J(1)=\frac{2\mathsf B_*}{3}.
\end{equation}
We give the monotonicity argument because it is what makes the constant
sharp.  Differentiation under the integral sign is legitimate since
\(W_b\) vanishes at both endpoints.  If
\[
  G_b(y)=-\frac23(3b^2-2b+1)y,
\]
direct algebra gives
\begin{align*}
  K_b\partial_bW_b-2K_b'W_b
  &=(\partial_yG_b)W_b+\frac12G_b\partial_yW_b\\
  &\quad-\frac{3b^2-2b+1}{3}(K_b+3y).
\end{align*}
The first two terms integrate to zero after division by \(\sqrt{W_b}\).
Consequently,
\begin{equation}\label{eq:mixed-wirtinger-J-derivative}
  J'(b)=
  -\frac{3b^2-2b+1}{3K_b^2}
  \int_{-b}^1\frac{K_b+3y}{\sqrt{W_b(y)}}\,\dd y.
\end{equation}

The remaining integral is positive.  Indeed, put
\[
  y=\frac{1-b}{2}+\frac{1+b}{2}x
\]
and
\[
  S_b(x)=(1+b)^2x^2+4(1-b^2)x+7(1+b^2)-10b.
\]
Then
\[
  W_b(y)=\frac{(1+b)^2}{16}(1-x^2)S_b(x),
\]
and, after pairing \(x\) with \(-x\), the sign in
\eqref{eq:mixed-wirtinger-J-derivative} reduces for \(0\le x\le1\) to
\[
  \frac{A_b+D_bx}{\sqrt{S_b(x)}}
  +\frac{A_b-D_bx}{\sqrt{S_b(-x)}}>0,
\]
where
\[
  A_b=\frac{1+b}{2}(2b^2-4b+3),
  \qquad D_b=\frac{3(1+b)}2.
\]
This is immediate when \(A_b-D_bx\ge0\).  Otherwise one may square the
two positive sides.  Their squared difference is
\[
  (A_b+D_bx)^2S_b(-x)-(D_bx-A_b)^2S_b(x)
  =x(1+b)^2P_b(x),
\]
where \(P_b\) is affine in \(x^2\) and
\begin{align*}
  P_b(0)
  &=(2b^2-4b+3)
    (4b^4-8b^3+23b^2-22b+15)>0,\\
  P_b(1)
  &=4(b^2-2b+3)
    (2b^4-4b^3+10b^2-8b+3)>0.
\end{align*}
The last two quartics are positive because they equal, respectively,
\[
  4b^2(1-b)^2+19b^2-22b+15
  \quad\hbox{and}\quad
  2b^2(1-b)^2+8\left(b-\frac12\right)^2+1.
\]
Hence \(J'(b)<0\), proving \eqref{eq:mixed-wirtinger-one-variable}.

Finally,
\[
  \int_{-1}^1\sqrt{1-y^4}\,\dd y=\frac{\mathsf B_*}{3},
  \qquad
  \frac{\pi}{\sqrt{\mathsf C_*}}=\frac{2\mathsf B_*}{3}.
\]
Substitution into \eqref{eq:mixed-wirtinger-calibrated-lower} proves
\eqref{eq:mixed-periodic-wirtinger}.

For sharpness, take a normalized periodic profile which has one
\(+1\) plateau and one \(-1\) plateau, each of total length \(\pi/4\),
joined by two monotone arcs satisfying
\[
  |q_s|=\mathsf C_*^{-1/2}\sqrt{1-q^4}.
\]
Each transition has length
\[
  \sqrt{\mathsf C_*}\int_{-1}^1\frac{\dd y}{\sqrt{1-y^4}}
  =\frac{3\pi}{4}.
\]
The transition arcs have zero mean, and equality holds in every step
above.  Thus the profile is mean zero and attains
\eqref{eq:mixed-periodic-wirtinger} with equality.
\end{proof}

\begin{theorem}[Strong hidden positivity under a half-Laplacian slope bound]
\label{thm:strong-half-laplacian-slope}
Let \(\rho\in C^\infty(\T)\) be positive, and put
\(\kappa=\min_\T\rho\).  If
\begin{equation}\label{eq:half-laplacian-slope-condition}
  \norm{\Lambda\rho}_{L^\infty}
  \le\frac{2}{\sqrt{\mathsf C_*}}\kappa
  =\frac{4\mathsf B_*}{3\pi}\kappa
  =2.2256715778\ldots\,\kappa,
\end{equation}
then
\begin{equation}\label{eq:half-laplacian-slope-coercivity}
  2I_1+I_2
  \ge
  \left(
    \kappa-\frac{\mathsf C_*\norm{\Lambda\rho}_{L^\infty}^2}{4\kappa}
  \right)
  \norm{\Lambda^{1/2}\rho}_2^2
  \ge0.
\end{equation}
The inequality is strict unless \(\rho\) is constant.
\end{theorem}

\begin{proof}
Set \(u(t)=P_t\rho\), \(q(t)=\Lambda u(t)\), and
\[
  E(t)=\int_\T u(t)u_s(t)^2\,\dd s.
\]
Since \(u_t=-q\), \(u_{tt}=-u_{ss}\), and
\(\Lambda q=-u_{ss}\), two differentiations and periodic integration by
parts give
\begin{equation}\label{eq:poisson-energy-second-variation}
\begin{aligned}
  \frac12E''(t)
  & =\int_\T u\bigl((\Lambda q)^2+q_s^2\bigr)\,\dd s
      +\int_\T q^2\Lambda q\,\dd s\\
  & =\int_\T u\left(\Lambda q+\frac{q^2}{2u}\right)^2\,\dd s
      +\int_\T u q_s^2\,\dd s
      -\frac14\int_\T\frac{q^4}{u}\,\dd s.
\end{aligned}
\end{equation}
Let \(\kappa_t=\min_\T u(t)\).  Positivity preservation, commutation with
\(\Lambda\), and the \(L^\infty\) contraction of \(P_t\) imply
\[
  \kappa_t\ge\kappa,
  \qquad
  \norm{q(t)}_{L^\infty}
  \le\norm{\Lambda\rho}_{L^\infty}
  \le\frac{2}{\sqrt{\mathsf C_*}}\kappa
  \le\frac{2}{\sqrt{\mathsf C_*}}\kappa_t.
\]
Moreover, \(q\) has zero mean.  Lemma~\ref{lem:mixed-periodic-wirtinger}
therefore gives
\[
  \frac14\int_\T\frac{q^4}{u}\,\dd s
  \le\frac{\mathsf C_*\norm q_{L^\infty}^2}{4\kappa_t}
       \norm{q_s}_2^2
  \le\kappa_t\norm{q_s}_2^2
  \le\int_\T u q_s^2\,\dd s.
\]
Writing \(A=\norm{\Lambda\rho}_{L^\infty}\), the same estimate retains the
quantitative remainder
\[
  \frac12E''(t)
  \ge
  \left(
    \kappa_t-\frac{\mathsf C_*\norm{q(t)}_{L^\infty}^2}{4\kappa_t}
  \right)\norm{q_s(t)}_2^2
  \ge
  \left(\kappa-\frac{\mathsf C_*A^2}{4\kappa}\right)\norm{q(t)}_2^2.
\]
Here the last step also uses the ordinary mean-zero Wirtinger inequality.
In particular, \(E''(t)\ge0\) for every \(t\ge0\).  Since
\(E(t),E'(t)\to0\) as \(t\to\infty\), integration in \(t\), followed by
the Fourier identity
\[
  2\int_0^\infty\norm{\Lambda P_t\rho}_2^2\,\dd t
  =\norm{\Lambda^{1/2}\rho}_2^2,
\]
and \eqref{eq:strong-semigroup-form} prove
\eqref{eq:half-laplacian-slope-coercivity}.

If \(2I_1+I_2=0\), then \(E'(0)=0\).  The derivative of a convex function
is nondecreasing, while convergence of \(E\) to zero forces it to be
nonpositive.  Hence \(E'\equiv0\), so \(E\equiv0\), and \(\rho\) is
constant.
\end{proof}

The square discarded in the preceding argument has a forced nonzero mean.
Retaining it yields a stronger criterion when the mean-to-floor ratio is
known.  Define, for \(M\ge1\),
\begin{equation}\label{eq:mean-sensitive-Phi}
  \Phi(M)=
  \begin{cases}
    \displaystyle \frac{\pi^2M}{108},
      &1\le M\le M_*,
      \\[6pt]
    \displaystyle \mathsf C_*-
      \frac{27\mathsf C_*^2}{\pi^2M},
      &M\ge M_*,
  \end{cases}
  \qquad
  M_*=\frac{54\mathsf C_*}{\pi^2}
  =4.4180624065\ldots.
\end{equation}
The function \(\Phi\) is increasing and
\(\Phi(M)<\mathsf C_*\) for every finite \(M\).

\begin{theorem}[Mean-sensitive half-Laplacian slope criterion]
\label{thm:mean-sensitive-half-laplacian-slope}
Let \(\rho\in C^\infty(\T)\) be positive, set
\[
  m=\mean\rho,\qquad
  \kappa=\min_\T\rho,\qquad
  M=\frac m\kappa,
  \qquad
  A=\norm{\Lambda\rho}_{L^\infty}.
\]
If
\begin{equation}\label{eq:mean-sensitive-slope-condition}
  A\le\frac{2\kappa}{\sqrt{\Phi(M)}},
\end{equation}
then
\begin{equation}\label{eq:mean-sensitive-slope-coercivity}
  2I_1+I_2
  \ge
  \left(\kappa-\frac{\Phi(M)A^2}{4\kappa}\right)
  \norm{\Lambda^{1/2}\rho}_2^2
  \ge0.
\end{equation}
The inequality is strict unless \(\rho\) is constant.
\end{theorem}

\begin{proof}
Keep the notation \(u=P_t\rho\), \(q=\Lambda u\), and \(E(t)\) from the
proof of Theorem~\ref{thm:strong-half-laplacian-slope}.  At a fixed \(t\),
write
\[
  \kappa_t=\min_\T u,\quad
  A_t=\norm q_{L^\infty},\quad
  X_t=\norm{q_s}_2^2,\quad
  Y_t=\norm q_2^2,\quad
  M_t=\frac m{\kappa_t},\quad
  a_t=\frac{A_t}{\kappa_t}.
\]
The square in \eqref{eq:poisson-energy-second-variation} satisfies
\begin{align}
  \int_\T u\left(\Lambda q+\frac{q^2}{2u}\right)^2\,\dd s
  &\ge
  \frac{
    \left(\int_\T\left(u\Lambda q+\frac12q^2\right)\,\dd s\right)^2
  }{\int_\T u\,\dd s}
  \notag\\
  &=\frac{9Y_t^2}{8\pi m},
  \label{eq:poisson-square-mean-lower}
\end{align}
because \(\Lambda u=q\) and
\(\int_\T u\Lambda q\,\dd s=Y_t\).

Let \(N_t=\int_\T q^4\,\dd s\).  The sharp mixed estimate and the
pointwise bound give
\[
  N_t\le\mathsf C_*A_t^2X_t,
  \qquad
  N_t\le A_t^2Y_t.
\]
Moreover, \(q\) has mean zero, so \(Y_t\le X_t\).  If
\(r_t=Y_t/X_t\in[0,1]\), then
\begin{equation}\label{eq:q4-two-scale-minimum}
  N_t\le A_t^2X_t\min\{\mathsf C_*,r_t\}.
\end{equation}
The point-evaluation form of the periodic Sobolev inequality is
\begin{equation}\label{eq:mean-zero-point-evaluation}
  X_t\ge\frac{6}{\pi}A_t^2.
\end{equation}
Indeed, this follows from Parseval, Cauchy--Schwarz, and
\(\sum_{n\ne0}n^{-2}=\pi^2/3\).

Combining \eqref{eq:poisson-energy-second-variation},
\eqref{eq:poisson-square-mean-lower},
\eqref{eq:q4-two-scale-minimum}, and
\eqref{eq:mean-zero-point-evaluation} gives
\begin{equation}\label{eq:mean-sensitive-second-variation}
  \frac12E''(t)
  \ge
  \kappa_tX_t
  \left[
    1-\frac{a_t^2}{4}
    \left(
      \min\{\mathsf C_*,r_t\}
      -\frac{27r_t^2}{\pi^2M_t}
    \right)
  \right].
\end{equation}
For \(\alpha=27/(\pi^2M)\), elementary optimization shows
\[
  \max_{0\le r\le1}
  \left(\min\{\mathsf C_*,r\}-\alpha r^2\right)
  =\Phi(M).
\]
Indeed, the maximum of \(r-\alpha r^2\) lies at
\(r=(2\alpha)^{-1}\) until this point reaches \(\mathsf C_*\), and
thereafter the maximum is attained at \(r=\mathsf C_*\).  This gives
exactly the two branches in \eqref{eq:mean-sensitive-Phi}.

Along the Poisson orbit,
\[
  \kappa_t\ge\kappa,\qquad
  A_t\le A,\qquad
  M_t\le M,\qquad
  a_t\le\frac A\kappa.
\]
Since \(\Phi\) is increasing, \eqref{eq:mean-sensitive-second-variation}
therefore yields
\[
  \frac12E''(t)
  \ge
  \left(\kappa-\frac{\Phi(M)A^2}{4\kappa}\right)Y_t.
\]
The conclusion now follows by the same integration in \(t\) as in
Theorem~\ref{thm:strong-half-laplacian-slope}.  The equality argument is
unchanged.
\end{proof}

The preceding identities admit a second decomposition that isolates a
weighted Dirichlet form.  This gives a physical-space sufficient condition
which is independent of Fourier degree.

\begin{proposition}[Strong hidden positivity under spectral-order alignment]
\label{prop:strong-order-alignment}
Let \(\rho\in C^\infty(\T)\) be positive, set
\[
  q=\Lambda\rho,
  \qquad
  J=\int_\T \rho q\Lambda q\,\dd s,
\]
and suppose that
\begin{equation}\label{eq:spectral-order-alignment}
  \bigl(\rho(s)q(s)-\rho(t)q(t)\bigr)
  \bigl(q(s)-q(t)\bigr)\ge0
  \qquad (s,t\in\T).
\end{equation}
Then
\begin{equation}\label{eq:strong-mixed-order-decomposition}
  2I_1+I_2=I_1+J
  \ge I_1
  \ge \min_\T\rho\,\norm{\Lambda^{3/2}\rho}_2^2.
\end{equation}
In particular, the strong inequality is strict unless \(\rho\) is constant.
Condition \eqref{eq:spectral-order-alignment} holds, for example, if
\(\rho=\Phi(q)\) on \(\T\) and the scalar function
\(z\mapsto z\Phi(z)\) is nondecreasing on the range of \(q\).
\end{proposition}

\begin{proof}
Since \(\Lambda q=\Lambda^2\rho=-\rho_{ss}\), integration by parts gives
\[
\begin{aligned}
  J
  &=-\int_\T\rho q\rho_{ss}\,\dd s\\
  &=\int_\T q\rho_s^2\,\dd s
    +\int_\T\rho\rho_s q_s\,\dd s
   =I_2+I_1.
\end{aligned}
\]
This proves the identity in \eqref{eq:strong-mixed-order-decomposition}.
The periodic half-Laplacian Dirichlet form also gives
\[
  J=\frac{1}{8\pi}\iint_{\T^2}
  \frac{\bigl(\rho(s)q(s)-\rho(t)q(t)\bigr)
  \bigl(q(s)-q(t)\bigr)}
  {\sin^2((s-t)/2)}\,\dd s\,\dd t.
\]
Thus \eqref{eq:spectral-order-alignment} implies \(J\ge0\), and the
  coercive estimate follows from Theorem~\ref{thm:third-order-hidden-positivity}.
Strictness follows from the equality characterization in that theorem.
The final assertion is immediate because then
\(\rho q=q\Phi(q)\) is a nondecreasing function of \(q\).
\end{proof}

The degree-independent theorem above combines with the
C\'ordoba--C\'ordoba defect to settle the full sign for arbitrary-frequency
densities in a quantitative amplitude regime.

\begin{theorem}[Strong hidden positivity under a two-to-one amplitude ratio]
\label{thm:strong-amplitude-ratio}
Let \(\rho\in C^\infty(\T)\) satisfy
\[
  0<\kappa:=\min_\T\rho,
  \qquad M:=\max_\T\rho.
\]
Then
\begin{equation}\label{eq:strong-amplitude-lower-bound}
  2I_1+I_2
  \ge 2(2\kappa-M)\norm{\Lambda^{3/2}\rho}_2^2.
\end{equation}
Consequently, \(2I_1+I_2\ge0\) whenever \(M\le2\kappa\), with strict
inequality unless \(\rho\) is constant.
\end{theorem}

\begin{proof}
Apply the pointwise C\'ordoba--C\'ordoba inequality to \(f=\rho_s\),
multiply by the nonnegative weight \(M-\rho\), and integrate.  Since
\(\int_\T\Lambda(f^2)\,\dd s=0\),
\[
  0\le\int_\T(M-\rho)\bigl(2f\Lambda f-\Lambda(f^2)\bigr)\,\dd s
  =2M\norm{\Lambda^{3/2}\rho}_2^2-2I_1+I_2.
\]
It follows that
\[
  2I_1+I_2
  \ge4I_1-2M\norm{\Lambda^{3/2}\rho}_2^2.
\]
Using \eqref{eq:I1-coercivity} proves
\eqref{eq:strong-amplitude-lower-bound}.  Moreover, the representation
\eqref{eq:I1-positive-representation} shows that
\(I_1>\kappa\norm{\Lambda^{3/2}\rho}_2^2\) for every nonconstant \(\rho\):
wherever \(\delta_h^2\rho\ne0\), its weight is strictly larger than
\(6\kappa\).  This proves strictness also in the endpoint case
\(M=2\kappa\).
\end{proof}

There is a complementary arbitrary-degree class in which the sign follows
from spectral phase coherence rather than amplitude control.

\begin{proposition}[Strong hidden positivity for phase-coherent spectra]
\label{prop:strong-phase-coherent}
Write
\[
  \rho(s)=m+\sum_{n\ne0}r_ne^{\ii ns},
  \qquad r_{-n}=\overline{r_n},\qquad m>0,
\]
and suppose
\begin{equation}\label{eq:phase-coherence-condition}
  \Re\bigl(r_ar_b\overline{r_{a+b}}\bigr)\ge0
  \qquad\text{for all }a,b\ge1.
\end{equation}
Then \(2I_1+I_2\ge0\), with strict inequality unless \(\rho\) is constant.
In particular, this holds when, after a translation, all positive Fourier
coefficients of \(\rho\) are nonnegative real numbers.
\end{proposition}

\begin{proof}
Every term on the right of
\eqref{eq:strong-positive-frequency-form} is nonnegative under
\eqref{eq:phase-coherence-condition}; the mean term is positive for every
nonconstant \(\rho\).
\end{proof}

The phases of three moments in one additive triad cannot become independently
anticoherent.  The following Toeplitz-minor estimate records a uniform form of
this constraint.

\begin{lemma}[Toeplitz triad floor]\label{lem:toeplitz-triad-floor}
Let \(\mu\) be a probability measure on \(\T\), and set
\(u_n=\widehat\mu(n)\).  If \(a,b\ge1\) and \(c=a+b\), then
\begin{equation}\label{eq:toeplitz-triad-floor}
  \Re\bigl(u_au_b\overline{u_c}\bigr)
  \ge-\frac16\bigl(\abs{u_a}^2+\abs{u_b}^2+\abs{u_c}^2\bigr).
\end{equation}
\end{lemma}

\begin{proof}
Put \(x=\abs{u_a}\), \(y=\abs{u_b}\), \(z=\abs{u_c}\), and
\(S=x^2+y^2+z^2\).  The principal Toeplitz minor with indices
\(0,a,c\) is positive semidefinite, so its determinant gives
\[
  1-S+2\Re\bigl(u_au_b\overline{u_c}\bigr)\ge0.
\]
If \(S\ge3/4\), this implies \eqref{eq:toeplitz-triad-floor}.  If
\(S\le3/4\), the modulus bound and the arithmetic--geometric mean inequality
give
\[
  \Re\bigl(u_au_b\overline{u_c}\bigr)
  \ge-xyz
  \ge-\left(\frac S3\right)^{3/2}
  \ge-\frac S6.
\]
\end{proof}

The same positive-frequency formula also gives an arbitrary-degree criterion
which allows completely incoherent phases.  It is naturally expressed through
the Wiener size of the positive spectrum.

\begin{theorem}[Strong hidden positivity under a Wiener spectral threshold]
\label{thm:strong-wiener-threshold}
With the notation of Proposition~\ref{prop:strong-phase-coherent}, set
\[
  A_W:=\sum_{n\ge1}\abs{r_n}.
\]
Then
\begin{equation}\label{eq:strong-wiener-lower-bound}
  2I_1+I_2
  \ge
  2\left(m-\frac3{\sqrt2} A_W\right)
  \norm{\Lambda^{3/2}\rho}_2^2.
\end{equation}
Consequently, \(2I_1+I_2\ge0\) whenever
\[
  \frac3{\sqrt2} A_W\le m.
\]
If the inequality is strict, then the strong form is positive for every
nonconstant \(\rho\).
\end{theorem}

\begin{proof}
Write \(x_n=\abs{r_n}\) and
\[
  E_3=\sum_{n\ge1}n^3x_n^2.
\]
For \(n=a+b\), put
\[
  W_{a,b}=n(a^2+ab+b^2).
\]
With \(t=a/b>0\), elementary differentiation gives
\[
  \frac{W_{a,b}}
  {n^{3/2}(a^{3/2}+b^{3/2})}
  =\frac{t^2+t+1}
  {\sqrt{1+t}(t^{3/2}+1)}
  \le\frac3{2\sqrt2}.
\]
Set \(u_n=n^{3/2}x_n\), and extend \(x\) and \(u\) by zero to the
nonpositive integers.  The preceding multiplier bound, Cauchy--Schwarz, and
the \(\ell^1*\ell^2\to\ell^2\) Young inequality give
\[
\begin{aligned}
  \sum_{a,b\ge1}W_{a,b}x_ax_bx_{a+b}
  &\le\frac3{2\sqrt2}
  \sum_{a,b\ge1}
  \bigl(u_ax_b+x_au_b\bigr)u_{a+b}\\
  &=\frac3{\sqrt2}\langle u*x,u\rangle_{\ell^2}\\
  &\le\frac3{\sqrt2}\norm{u*x}_{\ell^2}\norm{u}_{\ell^2}\\
  &\le\frac3{\sqrt2}A_W E_3.
\end{aligned}
\]
\begin{equation}\label{eq:wiener-triad-bound}
  \sum_{a,b\ge1}W_{a,b}x_ax_bx_{a+b}
  \le\frac3{\sqrt2} A_WE_3.
\end{equation}
Apply \eqref{eq:wiener-triad-bound} to the second term in
\eqref{eq:strong-positive-frequency-form}.  Since
\(\norm{\Lambda^{3/2}\rho}_2^2=4\pi E_3\), we obtain
\eqref{eq:strong-wiener-lower-bound}.
\end{proof}

\begin{corollary}[Universal infinite-mode Poisson-smoothing regime]
\label{cor:universal-poisson-smoothing}
Let \(\mu\) be an arbitrary probability measure on \(\T\), let \(m>0\), and
define
\[
  \rho(s)=m\int_\T\mathcal P_q(s-\theta)\,\dd\mu(\theta),
  \qquad
  0<q\le\frac12.
\]
Then \(\rho\) is smooth and positive and satisfies
\[
  2I_1+I_2\ge0.
\]
The inequality is strict unless \(\rho\) is constant.  The measure \(\mu\)
may have arbitrary finite or infinite support and may be continuous or
singular; in particular, the conclusion does not impose a Fourier cutoff.
\end{corollary}

\begin{proof}
Put \(u_n=\widehat\mu(n)\), \(X=q^2\), and
\(W_{a,b}=(a+b)(a^2+ab+b^2)\).  Formula
\eqref{eq:strong-positive-frequency-form} gives
\begin{equation}\label{eq:hybrid-poisson-bracket}
  \frac{2I_1+I_2}{8\pi m^3}
  =\sum_{n\ge1}n^3X^n\abs{u_n}^2
   +\sum_{a,b\ge1}W_{a,b}X^{a+b}
    \Re\bigl(u_au_b\overline{u_{a+b}}\bigr).
\end{equation}
For \(c=a+b\), write
\(z_{a,b}=(\abs{u_a}^2,\abs{u_b}^2,\abs{u_c}^2)\).  The Toeplitz floor
\eqref{eq:toeplitz-triad-floor} and the three modulus bounds obtained by
discarding one of the factors give
\begin{equation}\label{eq:four-triad-floors}
 \Re\bigl(u_au_b\overline{u_c}\bigr)\ge-v_k\mathbin{\boldsymbol\cdot}z_{a,b},
 \qquad
 \begin{aligned}
 v_0&=(1/6,1/6,1/6),&v_1&=(1/2,1/2,0),\\
 v_2&=(1/2,0,1/2),&v_3&=(0,1/2,1/2).
 \end{aligned}
\end{equation}
Indeed, for example,
\(-\abs{u_au_bu_c}\ge-(\abs{u_a}^2+\abs{u_c}^2)/2\)
because \(\abs{u_b}\le1\); the other two cases are identical.  Every convex
combination of the four vectors in \eqref{eq:four-triad-floors} is therefore
also a valid floor.

We first use \(v_0\) when \(\min\{a,b\}=1\) and \(v_1\) when \(a,b\ge2\).
At the endpoint \(X=1/4\), the exact geometric moments are
\[
 S_0=\frac1{12},\qquad S_1=\frac7{36},\qquad
 S_2=\frac{53}{108},\qquad S_3=\frac{149}{108},
 \qquad S_j:=\sum_{b\ge2}b^jX^b.
\]
Collecting the resulting charges against \(X^n\abs{u_n}^2\) gives the
baseline margins
\begin{equation}\label{eq:half-poisson-baseline-margins}
 D_1=-\frac49,\qquad D_2=-\frac{17}{54},\qquad
 D_n=\frac{27n^3-12n^2-80n-79}{54}\quad(n\ge3).
\end{equation}
Thus the whole infinite tail is already positive; only the first two margins
need to be redistributed.

For the unordered triads in the following table, replace the baseline vector
by \(\sum_{k=0}^3\lambda_kv_k\); for all unlisted triads keep the baseline
choice.  The listed integer vector is \(10^6(\lambda_0,\lambda_1,
\lambda_2,\lambda_3)\).  Each row has nonnegative entries summing to
\(10^6\), so every replacement is a valid convex combination.
\begin{center}
\small
\begin{tabular}{c|c}
\((a,b)\)&\(10^6(\lambda_0,\lambda_1,\lambda_2,\lambda_3)\)\\ \hline
\((1,2)\)&\((288841,0,0,711159)\)\\
\((1,3)\)&\((734857,0,0,265143)\)\\
\((1,4)\)&\((542836,0,0,457164)\)\\
\((2,2)\)&\((812184,187816,0,0)\)\\
\((2,3)\)&\((1000000,0,0,0)\)\\
\((2,5)\)&\((1000000,0,0,0)\)\\
\((2,7)\)&\((822662,177338,0,0)\)\\
\((3,3)\)&\((1000000,0,0,0)\)\\
\((3,4)\)&\((1000000,0,0,0)\)\\
\((3,5)\)&\((1000000,0,0,0)\)\\
\((3,6)\)&\((876294,123706,0,0)\)\\
\((3,7)\)&\((1000000,0,0,0)\)\\
\((5,5)\)&\((534065,465935,0,0)\)
\end{tabular}
\end{center}
Both ordered occurrences are changed when \(a<b\).  Exact rational charge
collection then gives corrected margins \(\widetilde D_n\) satisfying
\begin{equation}\label{eq:half-poisson-certificate-margin}
 \min_{1\le n\le10}\frac{\widetilde D_n}{n^3}
 =\frac{149274521}{149299200000}>0.
\end{equation}
No listed triad contains a frequency above ten.  Hence for every \(n\ge11\)
the corrected margin is still the cubic in
\eqref{eq:half-poisson-baseline-margins}; its numerator and derivative are
positive at \(n=11\), and its derivative is increasing thereon.  This proves
strict positivity for all infinitely many modes at \(X=1/4\).  The finite
table is only a low-mode redistribution, not a Fourier-degree cutoff.  The
ancillary exact-arithmetic script
\path{verify_half_poisson_radius.py} performs the charge collection in
\eqref{eq:half-poisson-certificate-margin} and verifies the uniform tail.

Finally, after division by \(X^n\), every charge placed on mode \(n\) is a
nonnegative constant times \(X^{a+b-n}\).  This exponent is zero for the
output mode \(n=a+b\) and positive for either input mode.  Thus every margin
is nonincreasing in \(X\), and endpoint positivity proves
\begin{equation}\label{eq:hybrid-poisson-diagonal-lower}
  \frac{2I_1+I_2}{8\pi m^3}
  \ge\sum_{n\ge1}X^n\widetilde D_n(X)\abs{u_n}^2\ge0
  \qquad(0<X\le1/4).
\end{equation}
All endpoint margins are strict.  Therefore equality forces \(u_n=0\) for
every \(n\ge1\), which is equivalent to \(\rho\) being constant.
\end{proof}

For later use, define the positive matrix
\[
 K(q)_{a,b}:=q^{a+b}
 \frac{(a+b)(a^2+ab+b^2)}{(ab)^{3/2}},
 \qquad X=q^2.
\]
Its Hilbert--Schmidt norm is
\begin{align}\label{eq:poisson-HS-exact}
 \mathcal K_{\mathrm{HS}}(X):=\norm{K(q)}_{\mathrm{HS}}^2
 ={}&2\operatorname{Li}_{-3}(X)\operatorname{Li}_3(X)
 +8\operatorname{Li}_{-2}(X)\operatorname{Li}_2(X)\\
 &+16\operatorname{Li}_{-1}(X)\operatorname{Li}_1(X)
 +10\operatorname{Li}_0(X)^2.
\end{align}
Termwise comparison of the polylogarithm series gives
\begin{equation}\label{eq:poisson-HS-majorant}
 \norm{K(q)}_{\mathrm{op}}^2
 \le\mathcal K_{\mathrm{HS}}(X)
 \le\mathcal J(X)
 :=\frac{X^2(144-279X+196X^2-55X^3)}{4(1-X)^5}.
\end{equation}

\begin{corollary}[All-radius diffuse-moment regime]
\label{cor:all-radius-diffuse-moments}
Let \(\mu\) be an arbitrary probability measure on \(\T\), let
\(0<q<1\), and define \(\rho=m(\mathcal P_q*\mu)\), where \(m>0\).  Set
\[
  \sigma_2(\mu):=\sup_{n\ge2}\abs{\widehat\mu(n)}.
\]
Here \(\mathcal K_{\mathrm{HS}}\) is defined by
\eqref{eq:poisson-HS-exact}.  If
\begin{equation}\label{eq:all-radius-diffuse-condition}
  \sigma_2(\mu)\sqrt{\mathcal K_{\mathrm{HS}}(q^2)}\le1,
\end{equation}
then \(2I_1+I_2\ge0\). More quantitatively,
\begin{equation}\label{eq:all-radius-diffuse-lower-bound}
  2I_1+I_2
  \ge2m\left(1-\sigma_2(\mu)\sqrt{\mathcal K_{\mathrm{HS}}(q^2)}\right)
       \norm{\Lambda^{3/2}\rho}_2^2.
\end{equation}
The purely rational condition
\[
  \sigma_2(\mu)\sqrt{\mathcal J(q^2)}\le1
\]
is sufficient as well, with \(\mathcal J\) given by
\eqref{eq:poisson-HS-majorant}.
Thus the conclusion is strict for every nonconstant \(\rho\) when the
inequality in \eqref{eq:all-radius-diffuse-condition} is strict.  The result
allows arbitrary finite, infinite, continuous, or singular support and has no
Fourier cutoff.
\end{corollary}

\begin{proof}
Retain the third moment modulus in the triad estimate.  Since \(a+b\ge2\),
\[
  \Re\bigl(\widehat\mu(a)\widehat\mu(b)
       \overline{\widehat\mu(a+b)}\bigr)
  \ge-\sigma_2(\mu)x_ax_b.
\]
The proof of Corollary~\ref{cor:universal-poisson-smoothing} then gives the
normalized lower bound
\[
  \norm{v}_{\ell^2}^2
  -\sigma_2(\mu)\langle v,K(q)v\rangle_{\ell^2}
  \ge\left(1-\sigma_2(\mu)\sqrt{\mathcal K_{\mathrm{HS}}(q^2)}\right)
       \norm{v}_{\ell^2}^2,
\]
because \(\norm{K(q)}_{\mathrm{op}}\le\norm{K(q)}_{\mathrm{HS}}\).
Restoring the normalization proves \eqref{eq:all-radius-diffuse-lower-bound}.
\end{proof}

The Wiener threshold does not cover concentrated Poisson profiles.  The next
result gives an unrestricted-amplitude family with infinitely many active
Fourier modes.

\begin{theorem}[Strong hidden positivity for two Poisson kernels]
\label{thm:two-poisson-hidden-positivity}
For \(0<q<1\), let
\[
  \mathcal P_q(s):=\frac{1-q^2}{1-2q\cos s+q^2}
  =1+2\sum_{n\ge1}q^n\cos(ns).
\]
If \(m>0\), \(0\le p\le1\), and \(\alpha,\beta\in\T\), then the smooth
positive density
\[
  \rho(s)=m\bigl[p\mathcal P_q(s-\alpha)
  +(1-p)\mathcal P_q(s-\beta)\bigr]
\]
satisfies
\[
  2I_1+I_2>0.
\]
Thus the strong inequality holds for this family uniformly in its infinite
Fourier degree, with no separation or amplitude restriction on the two
components.
\end{theorem}

\begin{proof}
By translation and cubic homogeneity, take \(m=1\), \(\alpha=0\), and set
\(\theta=\beta-\alpha\).  The positive Fourier coefficients are
\[
  r_n=q^nu_n,
  \qquad
  u_n=p+(1-p)e^{-\ii n\theta}.
\]
Put \(x=q^2\), \(h=p(1-p)\), and
\(L_n=1-\cos(n\theta)\).  Direct expansion gives
\[
  |u_n|^2=1-2hL_n,
  \qquad
  \operatorname{Re}(u_au_b\overline{u_{a+b}})
  =1-h(L_a+L_b+L_{a+b}).
\]
Consequently, the bracket on the right of
\eqref{eq:strong-positive-frequency-form}, divided by \(m^3\), is
\begin{equation}\label{eq:two-poisson-bracket}
  \mathcal B=B_0(x)-hC(x,\theta),
\end{equation}
where
\begin{align*}
  B_0(x)&=\sum_{n\ge1}n^3x^n
  +\sum_{a,b\ge1}W_{a,b}x^{a+b},\\
  C(x,\theta)&=2\sum_{n\ge1}n^3x^nL_n
  +\sum_{a,b\ge1}W_{a,b}x^{a+b}
  (L_a+L_b+L_{a+b}),
\end{align*}
and \(W_{a,b}=(a+b)(a^2+ab+b^2)\).

For \(A_k(z)=\sum_{n\ge1}n^kz^n\), summation over \(a+b=n\) gives
\[
  \sum_{a,b\ge1}W_{a,b}z^{a+b}
  =\frac{5A_4(z)-6A_3(z)+A_2(z)}6.
\]
Using the standard rational forms of \(A_2,A_3,A_4\), one obtains
\begin{equation}\label{eq:two-poisson-B0}
  B_0(x)=\frac{x(1+x)(x^2+8x+1)}{(1-x)^5}>0.
\end{equation}
It remains to prove \(4B_0-C\ge0\).  Set
\[
  y=\tan^2\frac\theta2,
  \qquad
  H=(1-x)^2+(1+x)^2y.
\]
Exact summation of the cosine series yields
\begin{equation}\label{eq:two-poisson-remainder}
  4B_0(x)-C(x,\theta)
  =\frac{2x(1+x)}{(1-x)^5H^5}
  \sum_{j=0}^5q_j(x)y^j,
\end{equation}
where
\begin{align*}
q_0={}&2(1-x)^{10}(x^2+8x+1),\\
q_1={}&8(1-x)^8(x^4+2x^3-19x^2+2x+1),\\
q_2={}&4(1-x)^6(3x^6-8x^5+7x^4+268x^3+7x^2-8x+3),\\
q_3={}&8(1-x)^4(x^8-4x^7+73x^6+164x^5-28x^4
  +164x^3+73x^2-4x+1),\\
q_4={}&2(1-x)^2(1+x)^2(x^8+6x^7+220x^6-198x^5
  +774x^4-198x^3+220x^2+6x+1),\\
q_5={}&16x(1+x)^4(1+x^2)(x^4+8x^2+1).
\end{align*}
The only coefficient which can be negative is \(q_1\).  Indeed, with
\(t=x+x^{-1}\ge2\), the unfactored positive polynomials in \(q_2,q_3,q_4\),
after division by \(x^3,x^4,x^4\), respectively, become
\[
  3t^3-8t^2-2t+284,
  \quad
  t^4-4t^3+69t^2+176t-172,
  \quad
  t^4+6t^3+216t^2-216t+336.
\]
The first two are increasing on \([2,\infty)\) and take the positive values
\(272\) and \(440\) at \(t=2\); the third is manifestly positive there.
Moreover,
\begin{align*}
  4q_0q_2-q_1^2
  ={}&32(1-x)^{16}
  (x^8+8x^7+14x^6+460x^5+1416x^4\\
  &\hspace{39mm}+460x^3+14x^2+8x+1)>0.
\end{align*}
Hence \(q_0+q_1y+q_2y^2>0\) for every real \(y\), while
\(q_3y^3+q_4y^4+q_5y^5\ge0\) for \(y\ge0\).  Formula
\eqref{eq:two-poisson-remainder} follows at \(\theta=\pi\) by continuity and
is strictly positive for \(0<x<1\).

Finally, \(0\le h\le1/4\), so \eqref{eq:two-poisson-bracket} can be written
as
\[
  \mathcal B=(1-4h)B_0+h(4B_0-C)>0.
\]
The positive-frequency identity then gives \(2I_1+I_2=8\pi m^3\mathcal B>0\).
\end{proof}

The preceding infinite-series theorem has a stronger cutoff-level form.  This
is useful because it controls all partial moment brackets at once, rather than
only their Poisson-weighted sum.

\begin{theorem}[All-cutoff positivity for two-point moment sequences]
\label{thm:two-point-all-cutoff}
Let
\[
  u_n=p+(1-p)e^{-\ii n\theta},
  \qquad 0\le p\le1,\quad \theta\in\R,
\]
and, for \(N\ge1\), define
\begin{equation}\label{eq:two-point-partial-bracket}
  \mathfrak B_N(u)
  :=\sum_{n=1}^N n^3\abs{u_n}^2
  +\sum_{\substack{a,b\ge1\\a+b\le N}}
  (a+b)(a^2+ab+b^2)
  \Re\bigl(u_au_b\overline{u_{a+b}}\bigr).
\end{equation}
Then
\[
  \mathfrak B_N(u)\ge0
  \qquad\text{for every }N\ge1.
\]
Consequently, if \(\mathfrak b_N=\mathfrak B_N-\mathfrak B_{N-1}\), with
\(\mathfrak B_0=0\), then
\begin{equation}\label{eq:two-point-abel-bracket}
  \sum_{N\ge1}\mathfrak b_Nx^N
  =(1-x)\sum_{N\ge1}\mathfrak B_Nx^N\ge0,
  \qquad 0<x<1.
\end{equation}
Thus every Fourier cutoff of a two-point moment sequence is positive, and the
Poisson smoothing follows by Abel summation without any degree restriction.
\end{theorem}

\begin{proof}
Put \(h=p(1-p)\) and \(L_n=1-\cos(n\theta)\).  As in the proof of
Theorem~\ref{thm:two-poisson-hidden-positivity},
\[
  \abs{u_n}^2=1-2hL_n,
  \qquad
  \Re(u_au_b\overline{u_{a+b}})
  =1-h(L_a+L_b+L_{a+b}).
\]
The bracket is therefore affine in \(h\):
\begin{equation}\label{eq:two-point-convex-reduction}
  \mathfrak B_N(u)
  =(1-4h)\mathfrak B_N^{(1)}+4h\mathfrak B_N^{(1/2)},
\end{equation}
where the superscripts denote the one-point and equal-weight two-point
sequences.  Since \(0\le h\le1/4\) and
\(\mathfrak B_N^{(1)}>0\), it remains to treat equal weights.

Set \(z=e^{-\ii\theta}\) and
\[
  A_N(z)=\sum_{n=1}^Nn(1+z^n),
  \qquad
  H_N(z)=\sum_{n=1}^Nn^2(1+z^n),
\]
and put
\[
  C_N(z)=NA_N(z)-H_N(z)
  =\sum_{n=1}^Nn(N-n)(1+z^n).
\]
Expanding \eqref{eq:two-point-partial-bracket} gives the exact identity
\begin{equation}\label{eq:two-point-AH-identity}
  4\mathfrak B_N^{(1/2)}
  =\Re\bigl(\overline{A_N}H_N\bigr)
  =N\abs{A_N}^2-\Re\bigl(\overline{A_N}C_N\bigr).
\end{equation}
We prove \(\abs{C_N}\le N\abs{A_N}\).  Write
\[
  R_N:=\Re A_N=\sum_{n=1}^Nn(1+\cos(n\theta)).
\]
The required input is the uniform trigonometric estimate
\begin{equation}\label{eq:weighted-cosine-lower-bound}
  R_N\ge\frac{N^2-1}{6}.
\end{equation}
Indeed, Cauchy--Schwarz and an elementary power sum give
\begin{align}
  \abs{C_N}^2
  &\le
  \left(\sum_{n=1}^Nn\abs{1+z^n}^2\right)
  \left(\sum_{n=1}^Nn(N-n)^2\right)\notag\\
  &=2R_N\frac{N^2(N^2-1)}{12}
  \le N^2R_N^2
  \le N^2\abs{A_N}^2.\label{eq:two-point-Cauchy}
\end{align}
Together with \eqref{eq:two-point-AH-identity}, this proves the theorem once
\eqref{eq:weighted-cosine-lower-bound} is established.

We now prove that estimate for all \(N\).  Let \(\xi=\cos\theta\) and
\[
  Q_N(\xi):=R_N-\frac{N^2-1}{6}
  =\frac{(N+1)(2N+1)}6+\sum_{n=1}^NnT_n(\xi),
\]
where \(T_n\) is the Chebyshev polynomial.  The cases \(N\le5\) have the
following exact Markov--Lukacs certificates.  With
\(v_2=(1,\xi)^\top\) and \(v_3=(1,\xi,\xi^2)^\top\),
\[
  Q_1=1+\xi,
  \qquad
  Q_2=4\left(\xi+\frac18\right)^2+\frac7{16},
\]
and
\begin{align*}
  Q_3&=(1+\xi)v_2^\top M_3^+v_2+(1-\xi)v_2^\top M_3^-v_2,\\
  Q_4&=v_3^\top M_4v_3+(1-\xi^2)v_2^\top N_4v_2,\\
  Q_5&=(1+\xi)v_3^\top M_5^+v_3+(1-\xi)v_3^\top M_5^-v_3,
\end{align*}
where
\[
  M_3^+=\begin{pmatrix}2&-14/3\\-14/3&38/3\end{pmatrix},
  \quad M_3^-=\frac23\Id,
\]
\[
  M_4=\begin{pmatrix}
  17/2&-7/2&-63/4\\-7/2&3&11/2\\-63/4&11/2&67/2
  \end{pmatrix},
  \quad
  N_4=\begin{pmatrix}1&-1/2\\-1/2&3/2\end{pmatrix},
\]
and
\[
  M_5^+=\begin{pmatrix}
  117/10&37/10&-1013/40\\
  37/10&15&-527/20\\
  -1013/40&-527/20&329/4
  \end{pmatrix},
\]
\[
  M_5^-=\begin{pmatrix}
  13/10&-2/5&-7/8\\
  -2/5&6/5&-1/10\\
  -7/8&-1/10&9/4
  \end{pmatrix}.
\]
All six matrices are positive definite by Sylvester's criterion, so
\(Q_N\ge0\) on \([-1,1]\) for \(N\le5\).

For the uniform part, assume \(N\ge6\), put \(t=\theta/2\), and use the
Fejer kernels
\[
  F_j(t):=\frac{\sin^2(jt)}{j\sin^2t}.
\]
Their Fourier coefficients give
\[
  \sum_{n=1}^Nn\cos(n\theta)
  =\frac{N(N+1)}2\bigl(F_{N+1}(t)-F_N(t)\bigr).
\]
Hence \(Q_N\ge0\) is equivalent to
\begin{equation}\label{eq:fejer-increment-bound}
  F_N(t)-F_{N+1}(t)\le c_N,
  \qquad c_N:=\frac{2N+1}{3N}.
\end{equation}
By symmetry take \(0\le t\le\pi/2\), and set \(y=\sin^2t\) and
\[
  y_N:=\frac{3(5N+1)}{(N+1)(2N+1)^2}.
\]
If \(y\ge y_N\), let \(v=(\sin Nt,\cos Nt)^\top\).  The numerator in
\(F_N-F_{N+1}\) is \(v^\top Mv\), where
\[
  M=\frac1N\begin{pmatrix}1&0\\0&0\end{pmatrix}
  -\frac1{N+1}
  \begin{pmatrix}\cos t\\\sin t\end{pmatrix}
  \begin{pmatrix}\cos t&\sin t\end{pmatrix}.
\]
Here
\[
  \tr M=\frac1{N(N+1)},
  \qquad
  \det M=-\frac{y}{N(N+1)}.
\]
For \(L=c_Ny\),
\[
  \det(L\Id-M)
  =\frac{y\bigl((N+1)(2N+1)^2y-3(5N+1)\bigr)}
  {9N^2(N+1)}\ge0,
\]
and \(L>\tr M\).  Thus \(M\le L\Id\), which proves
\eqref{eq:fejer-increment-bound} in this region.

It remains to consider \(y<y_N\).  Put \(u=Nt\).  Since \(t\le\tan t\),
\[
  u^2\le\frac{N^2y_N}{1-y_N}<4.
\]
For \(0\le u\le2\), the elementary inequality
\begin{equation}\label{eq:finite-sine-square-bound}
  \sin^2u-u\sin(2u)\le\frac{13}{20}u^2
\end{equation}
holds.  To see this, subtract the left side from \(13u^2/20\); its derivative
is \(2u(13/20+\cos(2u))\).  This derivative changes sign only once on
\([0,2]\), and both endpoint values are nonnegative.  For the endpoint
\(u=2\), write \(d=4-\pi<43/50\); the alternating Taylor bounds give
\(\cos d>13/20\) and \(\sin d<4/5\), which verify both assertions.

Since \((\sin^2s)''\ge-2\), Taylor's formula and
\eqref{eq:finite-sine-square-bound} yield
\begin{align*}
  &(N+1)\sin^2u-N\sin^2(u+t)\\
  &\qquad\le \sin^2u-u\sin(2u)+Nt^2
  \le\left(\frac{13}{20}+\frac1N\right)u^2.
\end{align*}
Finally, set \(q_N=y_N/(1-y_N)\).  The inequalities
\(t^2\le q_N\) and \(\sin t/t\ge1-t^2/6\) imply
\[
  \frac{(N+1)(2N+1)}{3N^2}
  \left(\frac{\sin t}{t}\right)^2
  \ge
  \frac{(N+1)(2N+1)}{3N^2}
  \left(1-\frac{q_N}{6}\right)^2
  \ge\frac{13}{20}+\frac1N.
\]
The last inequality is purely algebraic: after subtraction its numerator is
\[
  P(N)=16N^8+64N^7-496N^6-1856N^5-1362N^4-565N^3
  +3769N^2+1385N+125,
\]
over a positive denominator, and \(P(N)>0\) for \(N\ge6\) after expansion in
\(N-6\), where all coefficients are positive.  This proves
\eqref{eq:fejer-increment-bound}, hence
\eqref{eq:weighted-cosine-lower-bound}, for every \(N\).

The Abel identity \eqref{eq:two-point-abel-bracket} is the telescoping formula
for the partial sums \(\mathfrak B_N\).  For Poisson Fourier coefficients,
the shell of total frequency \(N\) is multiplied by \(x^N\), so the final
assertion follows.
\end{proof}

\section{Conclusion}
The exact slope equation \eqref{eq:exact-slope-equation} is the decisive
structure of the active-line model.  Its maximum principle, combined with
critical nonlocal regularity, excludes finite-time blow-up for every strictly
positive smooth initial density, without a smallness assumption.  This
conclusion concerns the reduced equation and does not establish a reduction
theorem for the full two-dimensional Oldroyd--B system.

The fourth-difference factorization in
Theorem~\ref{thm:third-order-hidden-positivity} also proves
\[
  \int_\T\rho^2\Lambda^3\rho\,\dd s\ge0
\]
for every smooth nonnegative density.  The unrestricted sharp \(H^1\)
monotonicity question is precisely Problem~\ref{prob:sharp-cubic}.

The results of Section~\ref{sec:hidden-positivity-open} cover several
cutoff-uniform and infinite-support classes but do not settle
Problem~\ref{prob:sharp-cubic}.  This open problem is logically independent
of global regularity, which is already established by
Theorem~\ref{thm:large-global}.  The canonical phase-opposed density
\(\rho_x\) in Lemma~\ref{lem:strong-fourier-representation} lies on the
boundary of the positive cone because \(\rho_x(0)=0\); this is another reason
to keep the sharp boundary problem separate from the strictly positive PDE
theorem.

\appendix
\section{Supporting Estimates for the Reduced Equation}
\label{sec:reduced-energy-details}

\begin{lemma}[Positive-coefficient order-one energy estimate]
\label{lem:positive-coefficient-energy}
Let \(r>3/2\).  Suppose that \(a,b\in H^r(\T)\) are real valued and
\(a\ge\kappa>0\).  Then every mean-zero
\(f\in H^{r+1/2}(\T)\) satisfies
\begin{equation}\label{eq:positive-coefficient-energy}
\begin{aligned}
 \ip{\Lambda^r(a\Lambda f+b f_s)}{\Lambda^r f}
 \ge \frac\kappa4\norm{f}_{\dot H^{r+1/2}}^2
 -C_{r,\kappa}
 \bigl(1+\norm a_{H^r}+\norm b_{H^r}\bigr)^2
 \norm f_{H^r}^2.
\end{aligned}
\end{equation}
\end{lemma}

\begin{proof}
Put \(g=\Lambda^r f\).  Self-adjointness of \(\Lambda^{1/2}\) gives
\[
 \ip{a\Lambda g}{g}
 =\int_\T a\abs{\Lambda^{1/2}g}^2\,\dd s
 +\ip{[\Lambda^{1/2},a]g}{\Lambda^{1/2}g}.
\]
The kernel formula for \(\Lambda^{1/2}\) and the Lipschitz bound for \(a\)
give
\[
 \norm{[\Lambda^{1/2},a]g}_2
 \le C_r\norm a_{H^r}\norm g_2.
\]
Together with the lower bound for \(a\) and Young's inequality, this yields
\[
 \ip{a\Lambda g}{g}
 \ge\frac\kappa2\norm{\Lambda^{1/2}g}_2^2
 -C_{r,\kappa}\norm a_{H^r}^2\norm g_2^2.
\]
Commuting \(\Lambda^r\) through \(a\) adds
\(\ip{[\Lambda^r,a]\Lambda f}{g}\).  The Kato--Ponce commutator estimate and
\(H^r(\T)\hookrightarrow W^{1,\infty}(\T)\) give
\[
 \norm{[\Lambda^r,a]\Lambda f}_2
 \le C_r\norm a_{H^r}\norm f_{H^r}.
\]
Similarly,
\[
 \ip{b g_s}{g}=-\frac12\int_\T b_sg^2\,\dd s,
\]
and
\[
 \norm{[\Lambda^r,b]f_s}_2
 \le C_r\norm b_{H^r}\norm f_{H^r}.
\]
Combining these bounds proves \eqref{eq:positive-coefficient-energy}.
\end{proof}

\begin{lemma}[Basic commutator estimate]\label{lem:commutator-short}
Let \(r>3/2\).  For every mean-zero \(f\in H^{r+1/2}(\T)\),
\[
  \abs{\ip{\Lambda^r B(f,f)}{\Lambda^r f}}
  \le C_r\norm{f}_{H^r}\norm{f}_{H^{r+1/2}}^2.
\]
The same estimate holds for differences:
\[
 \norm{B(f,f)-B(g,g)}_{H^{r-1}}
 \le C_r(\norm{f}_{H^r}+\norm{g}_{H^r})\norm{f-g}_{H^r}.
\]
\end{lemma}

\begin{proof}
Write \(B(f,f)=P_0(f\Lambda f-(\Hilb f)f_s)\) and put
\(G=\Lambda^r f\).  Commuting \(\Lambda^r\) through the two coefficients gives
\[
\begin{aligned}
 \Lambda^r(f\Lambda f)&=f\Lambda G+[\Lambda^r,f]\Lambda f,\\
 \Lambda^r((\Hilb f)f_s)&=(\Hilb f)G_s
       +[\Lambda^r,\Hilb f]f_s.
\end{aligned}
\]
As in Lemma~\ref{lem:positive-coefficient-energy}, the fractional commutator
formula and integration by parts give
\[
 \abs{\ip{f\Lambda G}{G}}
 +\abs{\ip{(\Hilb f)G_s}{G}}
 \le C_r\norm f_{H^r}
 \left(\norm{\Lambda^{1/2}G}_2^2+\norm G_2^2\right).
\]
The Kato--Ponce estimates for the two commutators have the same upper bound.
Because \(f\) is mean zero, \(\norm G_2\le\norm{\Lambda^{1/2}G}_2\), which
proves the energy estimate.

Finally, the one-dimensional product estimate gives
\[
 \norm{B(u,v)}_{H^{r-1}}
 \le C_r\norm u_{H^r}\norm v_{H^r}.
\]
Apply this after polarizing
\(B(f,f)-B(g,g)=B(f-g,f)+B(g,f-g)\) to obtain the difference
estimate.
\end{proof}

\begin{lemma}[Lower-barrier identity]\label{lem:barrier-short}
Let \(\rho\) be a smooth positive solution of \eqref{eq:rho}.  If \(s_t\) is a point where \(\rho(t,\cdot)\) attains its minimum, then
\[
  \frac{\dd}{\dd t}\rho(t,s_t)+\gamma(\rho(t,s_t)-m)
  =-c\rho(t,s_t)\Lambda\rho(t,s_t)
  +c\mean{\rho\Lambda\rho-(\Hilb\rho)\rho_s}\ge0.
\]
Consequently Corollary~\ref{cor:no-rupture} holds.
\end{lemma}

\begin{proof}
At a minimum, \(\rho_s(t,s_t)=0\), so the transport contribution
\(-c(\Hilb\rho)\rho_s\) vanishes.  Moreover,
\[
 \mean{\rho\Lambda\rho-(\Hilb\rho)\rho_s}
 =2\mean{\rho\Lambda\rho}\ge0
\]
by skew-adjointness of \(\Hilb\).  The periodic representation
\[
  \Lambda\rho(s)=c_\T\,\operatorname{p.v.}\int_\T \frac{\rho(s)-\rho(y)}{\sin^2((s-y)/2)}\,\dd y
\]
shows \(\Lambda\rho(s_t)\le0\) at a minimum.  This gives the differential inequality and hence the explicit lower bound by comparison with \(y'+\gamma(y-m)=0\).
\end{proof}

\section*{Acknowledgements}
The author was supported by the National Natural Science Foundation of China
(Grant No. 12501602), the Education Department of Hunan Province (Grant No.
24C0055), the Science and Technology Department of Hunan Province (Grant No.
2025JJ60052), and the Scientific Research Start-up Fund of Xiangtan University
(Grant No. KZ0810769).

\begingroup
\makeatletter
\let\oldthebibliography\thebibliography
\let\endoldthebibliography\endthebibliography
\renewenvironment{thebibliography}[1]{%
  \oldthebibliography{#1}%
  \footnotesize
  \setlength{\itemsep}{0pt}%
  \setlength{\parskip}{0pt}%
  \setlength{\parsep}{0pt}%
}{\endoldthebibliography}
\makeatother
\bibliographystyle{unsrt}
\bibliography{references}

\begin{thebibliography}{10}

\bibitem{Oldroyd1950}
J.~G. Oldroyd.
\newblock On the formulation of rheological equations of state.
\newblock {\em Proceedings of the Royal Society of London. Series A},
  200(1063):523--541, 1950.

\bibitem{BirdArmstrongHassager1987}
R.~Byron Bird, Robert~C. Armstrong, and Ole Hassager.
\newblock {\em Dynamics of Polymeric Liquids. Vol. 1: Fluid Mechanics}.
\newblock Wiley, New York, 2 edition, 1987.

\bibitem{Renardy2000}
Michael Renardy.
\newblock {\em Mathematical Analysis of Viscoelastic Flows}.
\newblock SIAM, Philadelphia, 2000.

\bibitem{GuillopeSaut1990}
C.~Guillope and J.-C. Saut.
\newblock Existence results for the flow of viscoelastic fluids with a
  differential constitutive law.
\newblock {\em Nonlinear Analysis}, 15:849--869, 1990.

\bibitem{LionsMasmoudi2000}
P.-L. Lions and N.~Masmoudi.
\newblock Global solutions for some {Oldroyd} models of non-{Newtonian} flows.
\newblock {\em Chinese Annals of Mathematics}, 21B:131--146, 2000.

\bibitem{CheminMasmoudi2001}
J.-Y. Chemin and N.~Masmoudi.
\newblock About lifespan of regular solutions of equations related to
  viscoelastic fluids.
\newblock {\em SIAM Journal on Mathematical Analysis}, 33:84--112, 2001.

\bibitem{BarrettBoyaval2011}
J.~W. Barrett and S.~Boyaval.
\newblock Existence and approximation of a regularized {Oldroyd--B} model.
\newblock {\em Mathematical Models and Methods in Applied Sciences},
  21:1783--1837, 2011.

\bibitem{ConstantinKliegl2012}
P.~Constantin and M.~Kliegl.
\newblock Note on global regularity for two-dimensional {Oldroyd--B} fluids
  with diffusive stress.
\newblock {\em Archive for Rational Mechanics and Analysis}, 206:725--740,
  2012.

\bibitem{ElgindiRousset2015}
T.~M. Elgindi and F.~Rousset.
\newblock Global regularity for some {Oldroyd--B} type models.
\newblock {\em Communications on Pure and Applied Mathematics}, 68:2005--2021,
  2015.

\bibitem{RenardyThomases2021}
Michael Renardy and Becca Thomases.
\newblock A mathematician's perspective on the {Oldroyd B} model: Progress and
  future challenges.
\newblock {\em Journal of Non-Newtonian Fluid Mechanics}, 293:104573, 2021.

\bibitem{FattalKupferman2004}
R.~Fattal and R.~Kupferman.
\newblock Constitutive laws for the matrix-logarithm of the conformation
  tensor.
\newblock {\em Journal of Non-Newtonian Fluid Mechanics}, 123:281--285, 2004.

\bibitem{FattalKupferman2005}
R.~Fattal and R.~Kupferman.
\newblock Time-dependent simulation of viscoelastic flows at high {Weissenberg}
  number using the log-conformation representation.
\newblock {\em Journal of Non-Newtonian Fluid Mechanics}, 126:23--37, 2005.

\bibitem{ConstantinLaxMajda1985}
Peter Constantin, Peter~D. Lax, and Andrew Majda.
\newblock A simple one-dimensional model for the three-dimensional vorticity
  equation.
\newblock {\em Communications on Pure and Applied Mathematics}, 38(6):715--724,
  1985.

\bibitem{CordobaCordobaFontelos2005}
Diego C{\'o}rdoba, Antonio C{\'o}rdoba, and Marco~A. Fontelos.
\newblock Formation of singularities for a transport equation with nonlocal
  velocity.
\newblock {\em Annals of Mathematics}, 162(3):1377--1389, 2005.

\bibitem{ChangLaraDavila2016}
H.~A. Chang-Lara and G.~D{\'a}vila.
\newblock H{\"o}lder estimates for non-local parabolic equations with critical
  drift.
\newblock {\em Journal of Differential Equations}, 260(5):4237--4284, 2016.

\bibitem{DongJinZhang2018}
H.~Dong, T.~Jin, and H.~Zhang.
\newblock Dini and {Schauder} estimates for nonlocal fully nonlinear parabolic
  equations with drifts.
\newblock {\em Analysis \& PDE}, 11(6):1487--1534, 2018.

\end{thebibliography}
\endgroup

\end{document}